\documentclass[10pt,a4paper]{amsart}
\usepackage{times}
\usepackage[left=1.15in,right=1.15in,top=1.15in,bottom=1.15in]{geometry} 
\usepackage{amsmath}
\usepackage{amssymb}
\usepackage{amsthm}
\usepackage{amscd}
\usepackage{afterpage}
\usepackage[ansinew]{inputenc}
\usepackage{cite}
\usepackage{bbm}
\usepackage{color}
\usepackage[english=american]{csquotes}
\usepackage[final]{graphicx}
\usepackage[all]{xy}
\usepackage{color}
\usepackage{scrextend}
\usepackage[ampersand]{easylist}
\usepackage{enumitem}
\usepackage{tikz-cd}
\usepackage{tikz,tikz-cd}
\usepackage{algorithm2e}
\usepackage{cite}
\usepackage{tikz-cd}
\usepackage{tikz,tikz-cd}
\usepackage[percent]{overpic}

\definecolor{rev1}{HTML}{cb270f}
\definecolor{rev2}{HTML}{1c8235}

\numberwithin{equation}{section}

\newtheorem{lemma}{Lemma}
\newtheorem{proposition}[lemma]{Proposition}
\newtheorem{theorem}[lemma]{Theorem}
\numberwithin{lemma}{section}
\newtheorem{corollary}[lemma]{Corollary}
\newtheorem{definition}[lemma]{Definition}

\theoremstyle{definition}
\newtheorem{remark}[lemma]{Remark}

\theoremstyle{definition}

\theoremstyle{definition}
\newtheorem{example}[lemma]{Example}

\theoremstyle{definition}

\newcommand*{\MA}[1]{{\color{magenta}#1}}

\DefineNamedColor{named}{JungleGreen} {cmyk}{0.99,0,0.52,0}

\usepackage{array}
\usepackage{booktabs}

\newenvironment{changemargin}[2]{%
  \begin{list}{}{%
    \setlength{\topsep}{0pt}%
    \setlength{\leftmargin}{#1}%
    \setlength{\rightmargin}{#2}%
    \setlength{\listparindent}{\parindent}%
    \setlength{\itemindent}{\parindent}%
    \setlength{\parsep}{\parskip}%
  }%
  \item[]}{\end{list}}

\usepackage{hyperref}

\usepackage[foot]{amsaddr}

\begin{document}
\RestyleAlgo{boxruled}

\title[Foundations of spectral computations]{The foundations of spectral computations via the Solvability Complexity Index hierarchy}

\author{Matthew J. Colbrook} 
\address{Department of Applied Mathematics and Theoretical Physics, University of Cambridge}
\email{m.colbrook@damtp.cam.ac.uk \normalfont(corresponding author)}

\author{Anders C. Hansen} 
\email{a.hansen@damtp.cam.ac.uk}


\begin{abstract}
The problem of computing spectra of operators is arguably one of the most investigated areas of computational mathematics. However, the problem of computing spectra of general bounded infinite matrices has only recently been solved. We establish some of the foundations of computational spectral theory through the Solvability Complexity Index (SCI) hierarchy, an approach closely related to Smale's program on the foundations of computational mathematics and McMullen's results on polynomial root finding with rational maps. Infinite-dimensional problems yield an intricate infinite classification theory, determining which spectral problems can be solved and with what types of algorithms. We provide answers to many longstanding open questions on the existence of algorithms. For example, we show that spectra can be computed, with error control, from point sampling operator coefficients for large classes of partial differential operators on unbounded domains. Further results include: computing spectra of (possibly unbounded) operators on graphs and separable Hilbert spaces with error control; determining if the spectrum intersects a compact set; the computational spectral gap problem and computing spectral classifications at the bottom of the spectrum; and computing discrete spectra, multiplicities, eigenspaces and determining if the discrete spectrum is non-empty. Moreover, the positive results with error control can be used in computer-assisted proofs. In contrast, the negative results preclude computer-assisted proofs for classes of operators as a whole. Our proofs are constructive, yielding a library of new algorithms and techniques that handle problems that before were out of reach. We demonstrate these algorithms on challenging problems, giving concrete examples of the failure of traditional approaches (e.g., ``spectral pollution'') compared to the introduced techniques.
\end{abstract}


\maketitle

\tableofcontents

\section{Introduction}
\label{intro}

The problem of computing spectra of operators has fascinated yet frustrated mathematicians for several decades, resulting in a vast literature (see \S \ref{sec:connection_work}). Indeed, W. Arveson pointed out in the nineties that: ``{\it Unfortunately, there is a dearth of literature on this basic problem, and so far as we have been able to tell, there are no proven techniques}'' \cite{Arveson_role_of94}. This longstanding problem for general infinite matrices has recently been addressed \cite{ben2015can, hansen2011solvability}. Arveson's question, of why ``there are no proven techniques'', can be explained by classification results in the newly established Solvability Complexity Index (SCI) hierarchy \cite{opt_big, CRAS_SCI, Hansen2016ComplexityII, colbrook2021computing,colb1,colbrook4,colbrook2020foundations,colb2,ben2015can,hansen2011solvability, colbrook2021computingSIREV, colbrook2021can}. The fact that algorithms were not found for the general computational spectral problem has a potentially surprising cause: one needs several limits in the computation. Traditional approaches have been dominated by techniques based on one limit, and this is the reason behind Arveson's observation. Moreover, the fact that several limits are required is a phenomenon shared by other areas of computational mathematics. For example, the problem of root-finding of polynomials with rational maps initiated by S. Smale \cite{smale_question} is also subject to the issue of requiring several limits. This result was established by C. McMullen \cite{McMullen1, mcmullen1988braiding} and P. Doyle \& C. McMullen in \cite{Doyle_McMullen}, and their results become classification results in the SCI hierarchy.   

Recent results establishing the SCI hierarchy 
 \cite{opt_big, CRAS_SCI, Hansen2016ComplexityII, colbrook2021computing,colb1,colbrook4,colbrook2020foundations,colb2,ben2015can,hansen2011solvability, colbrook2021computingSIREV, colbrook2021can} reveal that the computational spectral problem becomes an infinite classification theory. There is a vast well of open problems, some of which have been open for decades. For example, the following issue, even when neglecting the requirement of an error parameter, has been open since the early days of spectral computations in the 1950s:

\vspace{1mm}
\begin{displayquote}
	\normalsize
	{\it 
	For which classes of differential operators on unbounded domains do there exist algorithms that converge to the true spectrum, and also guarantee that the output is in the spectrum up to an arbitrary small $\epsilon>0$ parameter (the problem is in $\Sigma_1$ in the SCI hierarchy language)? In other words, the algorithm is verifiable and will never make a mistake. 
}
\end{displayquote}
\vspace{1mm}

A vast literature on computing spectra of differential operators on bounded domains exists. However, these techniques will typically yield non-convergent methods in the unbounded domain case. Even for bounded domains, obtaining error bounds is, in general, well known to be very difficult. The purpose of this paper is to provide solutions to such problems, and this program has three main motivations:
\begin{labeling}{\,\,}
\item[\,\,\,\,\,(I) \textit{Classifications and lower bounds}:] 
Sharp classifications of problems in the SCI hierarchy establish the boundaries of what computers can achieve. Such classifications give a precise measure of the difficulty\footnote{We are referring here to difficulty in terms of \textit{computability}. This is different to computational \textit{complexity}, which only makes sense for $\Delta_1$ problems (i.e., problems for which there exists an algorithm that given $\epsilon$ produces an $\epsilon$-accurate solution). By simply considering diagonal infinite matrices, it is
easy to see that most infinite-dimensional spectral problems of interest are $\notin\Delta_1$.} of a computational problem and prevent the search for algorithms that cannot exist.

\vspace{1mm}

\item[\,\,\,\,\,(II) \textit{New algorithms}:] 
Constructive classifications, which we always provide in this paper, provide algorithms that realise these boundaries. Such algorithms solve problems in the sciences that before were not possible. We provide several examples in this paper.

\vspace{1mm}

\item[\,\,\,\,\,(III) \textit{Computer-assisted proofs}:] Computer-assisted proofs use computers to solve numerical problems rigorously. These have become essential in modern mathematics. What may be surprising is that undecidable (non-computable) problems can be used in computer-assisted proofs. For example, led by T. Hales, the recent proof of Kepler's conjecture (Hilbert's 18th problem) on optimal packings of $3$-spheres relies on such undecidable problems \cite{Hales_Annals, hales_Pi}. Another example is the Dirac--Schwinger conjecture on the asymptotic behaviour of ground states of certain Schr\"{o}dinger operators. This conjecture was proven in a series of papers by C. Fefferman and L. Seco \cite{fefferman1990, fefferman1992, fefferman1993aperiodicity,  fefferman1994, fefferman1994_2, fefferman1995, fefferman1996interval, fefferman1996, fefferman1997} using computer assistance. Fascinatingly, this proof also relies on computing non-computable problems. The SCI hierarchy explains this apparent paradox. In particular, the $\Sigma^A_1$ class described below is crucial. Hales, Fefferman and Seco implicitly prove $\Sigma^A_1$ classifications in the SCI hierarchy in their papers. Our classifications of spectral problems provide new results on which spectral problems can be used in computer-assisted proofs.
\end{labeling}

Table \ref{result_table} provides a summary of the main results of this paper. \S \ref{fhwoiuhg} contains the theorems, and we focus on the following four important open problems:

\begin{labeling}{\,\,}
\item[\,\,\,\,\,(i) \textit{Computing spectra of differential operators}.] Linked to computational PDE theory, there is a rich literature on computing spectra of differential operators on bounded domains (see \cite{Rappaz1,rappaz1997spectral,boffi2000problem,Boffi2,Annalisa2,zhao2007spurious,Annalisa5,Annalisa3,Snorre1} for a small sample). However, in general, it is unknown how to compute spectra of differential operators on unbounded domains. We provide a sharp solution to this problem for large classes of differential operators, realising the boundary of what computers can achieve. We provide convergent algorithms that are also guaranteed to produce an output contained in the spectrum, up to an arbitrarily small error chosen by the user. As such, these algorithms can be used in computer-assisted proofs.

\vspace{1mm}

\item[\,\,\,\,\,(ii) \textit{Computing spectra of unbounded operators on graphs}.] Operators on $l^2(\mathbb{N})$ and, more generally, graphs or lattices are ubiquitous in mathematics and physics. We establish sharp classifications of spectral problems for such operators. In many cases, we provide convergent algorithms with guaranteed error control on the output. Hence these algorithms may be used in computer-assisted proofs. We also consider the decision problem of determining whether spectra (or pseudospectra) intersect a given compact set.

\vspace{1mm}

\item[\,\,\,\,\,(iii) \textit{The spectral gap problem}.] 
The spectral gap problem has a long tradition. It is linked to many important conjectures and problems, such as the Haldane conjecture \cite{golinelli1994finite} and the Yang--Mills mass gap problem in quantum field theory \cite{bombieri2006millennium}. The problem consists of determining whether there is a gap between the lowest element in the spectrum and the next element. We show why this problem is notoriously difficult. The problem is higher up in the SCI hierarchy, even for the simplest of operators. This result means that no algorithm can provide verifiable results on a computer. Hence, these problems cannot be used in computer-assisted proofs without further (typically \textit{global}) assumptions on the class of operators. We extend this result to spectral classification at the bottom of the spectrum.

\vspace{1mm}

\item[\,\,\,\,\,(iv) \textit{Computing discrete spectra and multiplicities}.] Computing discrete spectra is a notoriously difficult problem and previous numerical approaches have found it very difficult to do this reliably, even for special classes of one-dimensional operators (see \S \ref{sec:spec_class_lit_added}). We demonstrate why this is a difficult problem by establishing the correct classification high up in the SCI hierarchy. However, the sharp algorithm we provide is still practical. Its first limit is always contained in the discrete spectrum, and one can obtain the distance of each point of the output to the spectrum. We extend these results to computing multiplicities, eigenspaces, and determining if the discrete spectrum is non-empty.
\end{labeling}

\begin{table}
\begin{center}
\addtolength{\leftskip} {-1.6cm}
    \addtolength{\rightskip}{-1.6cm}
\begin{tabular}{|p{7.6cm}|p{6.2cm}|l|}
\hline
 \rule{0pt}{1.1em}\textbf{Problem Description}&\textbf{SCI Hierarchy Classification}&\textbf{Theorems}\\
 \hline
 \hline

\rule{0pt}{1.1em}Computing spectrum/pseudospectrum of differential operators, whose coefficients have bounded total variation, from point evaluations of coefficients.&$\in\Sigma_1^A$, $\notin\Delta_1^G$

(see Theorem for relaxations)&\ref{PDE1}\\
\hline
\rule{0pt}{1.1em}Computing spectrum/pseudospectrum of differential operators, whose coefficients are entire, from power series of coefficients.&$\in\Sigma_1^A$, $\notin\Delta_1^G$

(see Theorem for relaxations)&\ref{PDE2}\\
\hline

 \rule{0pt}{1.1em}Computing spectrum/pseudospectrum of unbounded operators with known bounded dispersion and known resolvent bound.&$\in\Sigma_1^A$, $\notin\Delta_1^G$

(same for diagonal operators)&\ref{unbounded_theorem}\\
\hline

\rule{0pt}{1.1em}Determining if the spectrum/pseudospectrum of an operator with known bounded dispersion intersects a compact set.&$\in\Pi_2^A$, $\notin\Delta_2^G$

(same for diagonal operators)&\ref{unbounded_test_theorem}\\
\hline
\rule{0pt}{1.1em}Spectral gap problem.&$\in\Sigma_2^A$, $\notin\Delta_2^G$

(same for diagonal operators)&\ref{class_thdfjlwdjfkl}\\
\hline
 \rule{0pt}{1.1em}Spectral classification problem.&$\in\Pi_2^A$, $\notin\Delta_2^G$

(same for diagonal operators)&\ref{class_thdfjlwdjfkl}\\
 \hline
\rule{0pt}{1.1em}Computing $\mathrm{cl}(\mathrm{Sp}_d(A))$ (and multiplicities of eigenvalues) for bounded normal operators. Here, $\mathrm{Sp}_d(A)$ denotes the discrete spectrum of $A$.&With bounded dispersion: $\in\Sigma_2^A$, $\notin\Delta_2^G$

(same for diagonal operators)

Multiplicities: $\in\Pi_2^A$

Without bounded dispersion: $\in\Sigma_3^A$, $\notin\Delta_3^G$ &\ref{discreteojoiojo},  \ref{discrete_dont_now_f}\\
\hline
 \rule{0pt}{1.1em}Determining if the discrete spectrum is non-empty for bounded normal operators.&With bounded dispersion: $\in\Sigma_2^A$, $\notin\Delta_2^G$

Without bounded dispersion: $\in\Sigma_3^A$, $\notin\Delta_3^G$&\ref{discreteojoiojo},  \ref{discrete_dont_now_f}\\

\hline

\end{tabular}
\vspace{3mm}
\caption{Summary of the main results. Bounded dispersion means that we know the asymptotic off-diagonal decay of suitable matrix elements of the operator, see \eqref{bd_disp}. Known resolvent bound means control of the growth of the resolvent $(A-zI)^{-1}$ near the spectrum, see \eqref{comp_ball_gs2} and \eqref{comp_ball_gs}. Appendix \ref{append_pseudo} provides pseudocode for the algorithms.} 
\vspace{-7mm}
\label{result_table}
\end{center}
\end{table}

The rest of this paper is organised as follows. In \S \ref{fdgjoej} we provide a brief summary of the SCI hierarchy to allow the interpretation of Table \ref{result_table} and theorems, with a detailed discussion of the hierarchy delayed until \S \ref{SCI_Hierarchy}. The main results are given in \S \ref{fhwoiuhg} with connections to previous work provided in \S \ref{sec:connection_work}. Proofs are given in \S \ref{pf_unb_gr} -- \S \ref{Sec:proof_spec_gap}. Finally, some computational examples are given in \S \ref{num_test}, and pseudocode is provided in Appendix \ref{append_pseudo}.

\section{Classifications in the SCI Hierarchy}
\label{fdgjoej}

\subsection{The SCI hierarchy}\label{sec:SCI_hierarchy}

We start with the definition of a computational problem. The basic objects of a computational problem are:
$\Omega$, called the \emph{domain}, $\Lambda$ a set of complex-valued functions on $\Omega$, called the \emph{evaluation set}, $(\mathcal{M},d)$ a metric space, and $\Xi:\Omega\to \mathcal{M}$ the \emph{problem function}. The set $\Omega$ is the set of objects that give rise to our computational problems. The problem function $\Xi : \Omega\to \mathcal{M}$ describes what we want to compute (with the metric space giving the notion of convergence). Finally, $\Lambda$ is the collection of functions that provide the information we allow algorithms to read as input.

\begin{definition}[Computational problem]\label{def:comp_prob}
Given (i) a domain $\Omega$, (ii) an evaluation set $\Lambda$, such that for any $A_1, A_2 \in \Omega$, $A_1 = A_2$ if and only if $f(A_1) = f(A_2)$ for all $f \in \Lambda$, (iii) a metric space $\mathcal{M}$, and (iv) a problem function $\Xi:\Omega\to\mathcal{M}$, we call the collection $\{\Xi,\Omega,\mathcal{M},\Lambda\}$ a computational problem.
\end{definition}

The definition of a computational problem is deliberately general to capture any computational problem in the literature. The set-up of this paper has the following typical form: $\Omega$ is a class of operators on a separable Hilbert space $\mathcal{H}$, $\Xi(A) = \mathrm{Sp}(A)$ (the spectrum or other related maps), $(\mathcal{M},d)$ is the collection of closed subsets of $\mathbb{C}$ with an appropriate generalisation of the Hausdorff metric (see \eqref{Hausdorff} and \eqref{eq:Attouch-Wets}), and $\Lambda$ may be the set of complex functions that could provide the matrix elements of $A \in \Omega$ given some orthonormal basis $\{e_j\}$ of $\mathcal{H}$. For example, $\Lambda$ could consist of $f_{i,j}: A\mapsto \langle Ae_j,e_i\rangle$, $i,j\in\mathbb{N}$, the entries of the matrix representation of $A$ with respect to the basis. As another example, $\Lambda$ could be the collection of functions providing point samples of a potential (or coefficient) function of a Schr\"{o}dinger (or more general) partial differential operator. 

The SCI of a class of computational problems is the smallest number of limits needed to compute the solution to the problem. The SCI hierarchy for spectral problems can be informally described as follows \cite{hansen2011solvability,ben2015can,colbrook2020foundations}. For decision problems, the description is similar (see \S \ref{SCI_Hierarchy} for the formal definitions).

{\bf The SCI hierarchy:}
Given a collection $\mathcal{C}$ of computational problems,
\begin{itemize}
\item[(i)] $\Delta^{\alpha}_0 = \Pi^{\alpha}_0 = \Sigma^{\alpha}_0$ is the set of problems that can be computed in finite time, the SCI $=0$.
\item[(ii)] $\Delta^{\alpha}_1$ is the set of problems that can be computed using one limit (the SCI $=1$) with control of the error, i.e., there exists a sequence of algorithms $\{\Gamma_n\}$ such that $d(\Gamma_n(A), \Xi(A)) \leq 2^{-n}, \, \forall A \in \Omega$.
\item[(iii)] $\Sigma^{\alpha}_1$ is the set of problems for which there exists a sequence of algorithms $\{\Gamma_n\}$, such that $\lim_{n\rightarrow\infty}\Gamma_n(A)=\Xi(A), \, \forall A \in \Omega$. Moreover,  $\Gamma_n(A)$ is always contained in a set $X_n(A)$ such that $d(X_n,\Xi(A))\leq 2^{-n}$. We have $\Delta^{\alpha}_1 \subset \Sigma^{\alpha}_1 \subset \Delta^{\alpha}_2 $ (where $\Delta^{\alpha}_2$ is described below).
\item[(iv)] $\Pi^{\alpha}_1$ is the set of problems for which there exists a sequence of algorithms $\{\Gamma_n\}$, such that $\lim_{n\rightarrow\infty}\Gamma_n(A)=\Xi(A), \, \forall A \in \Omega$. Moreover,  there exists sets $X_n(A)$ such that $\Xi(A)\subset X_n(A)$ and $d(X_n,\Gamma_n(A))\leq 2^{-n}$. We have $\Delta^{\alpha}_1 \subset \Pi^{\alpha}_1 \subset \Delta^{\alpha}_2 $ (where $\Delta^{\alpha}_2$ is described below).
\item[(v)] $\Delta^{\alpha}_2$ is the set of problems that can be computed using one limit (the SCI $=1$) without the requirement of error control, i.e., there exists a sequence of algorithms $\{\Gamma_n\}$ such that $\lim_{n\rightarrow \infty}\Gamma_n(A) = \Xi(A), \, \forall A \in \Omega$.
\item[(vi)] $\Delta^{\alpha}_{m+1}$, for $m \in \mathbb{N}$, is the set of problems that can be computed by using $m$ limits, (the SCI $\leq m$), i.e., there exists a family of algorithms $\{\Gamma_{n_m, \hdots, n_1}\}$ such that 
$$
\lim_{n_m \rightarrow\infty}\hdots \lim_{n_1\rightarrow\infty}\Gamma_{n_m,\hdots, n_1}(A) = \Xi(A), \quad\, \forall A \in \Omega.
$$
\item[(vii)] $\Sigma^{\alpha}_{m}$ is the set of problems that can be computed by passing to $m$ limits, and computing the $m$th limit is a $\Sigma^{\alpha}_1$ problem.
\item[(viii)] $\Pi^{\alpha}_{m}$ is the set of problems that can be computed by passing to $m$ limits, and computing the $m$th limit is a $\Pi^{\alpha}_1$ problem. 
 \end{itemize}
 
Schematically, the SCI hierarchy can be viewed in the following way.

\begin{equation}\label{SCI_hierarchy}
\begin{tikzpicture}[baseline=(current  bounding  box.center)]
  \matrix (m) [matrix of math nodes,row sep=1.2em,column sep=1.5em] {
  \Pi_0^{\alpha}   &                    & \MA{\Pi_1^{\alpha}} &    &  \Pi_2^{\alpha}&  & {}\\
  \Delta_0^{\alpha}&  \Delta_1^{\alpha} & \Sigma_1^{\alpha}\cup\Pi_1^{\alpha} & \Delta_2^{\alpha}&      \Sigma_2^{\alpha}\cup\Pi_2^{\alpha} & \Delta_3^{\alpha}& \cdots\\
	\Sigma_0^{\alpha}&                    & \MA{\Sigma_1^{\alpha}} & &  \Sigma_2^{\alpha}&  &{} \\
  };
 \path[-stealth, auto] (m-1-1) edge[draw=none]
                                    node [sloped, auto=false,
                                     allow upside down] {$=$} (m-2-1)
																		(m-3-1) edge[draw=none]
                                    node [sloped, auto=false,
                                     allow upside down] {$=$} (m-2-1)
																		
																		(m-2-2) edge[draw=none]
                                    node [sloped, auto=false,
                                     allow upside down] {$\subsetneq$} (m-2-3)
																		(m-2-3) edge[draw=none]
                                    node [sloped, auto=false,
                                     allow upside down] {$\subsetneq$} (m-2-4)
																		(m-2-4) edge[draw=none]
                                    node [sloped, auto=false,
                                     allow upside down] {$\subsetneq$} (m-2-5)
																		(m-2-5) edge[draw=none]
                                    node [sloped, auto=false,
                                     allow upside down] {$\subsetneq$} (m-2-6)
																		(m-2-6) edge[draw=none]
                                    node [sloped, auto=false,
                                     allow upside down] {$\subsetneq$} (m-2-7)

												(m-2-1) edge[draw=none]
                                    node [sloped, auto=false,
                                     allow upside down] {$\subsetneq$} (m-2-2)
											 (m-2-2) edge[draw=none]
                                    node [sloped, auto=false,
                                     allow upside down] {$\subsetneq$} (m-1-3)
											(m-2-2) edge[draw=none]
                                    node [sloped, auto=false,
                                     allow upside down] {$\subsetneq$} (m-3-3)
											 (m-1-3) edge[draw=none]
                                    node [sloped, auto=false,
                                     allow upside down] {$\subsetneq$} (m-2-4)
																		(m-3-3) edge[draw=none]
                                    node [sloped, auto=false,
                                     allow upside down] {$\subsetneq$} (m-2-4)
																		(m-2-4) edge[draw=none]
                                    node [sloped, auto=false,
                                     allow upside down] {$\subsetneq$} (m-1-5)
											(m-2-4) edge[draw=none]
                                    node [sloped, auto=false,
                                     allow upside down] {$\subsetneq$} (m-3-5)
																		(m-1-5) edge[draw=none]
                                    node [sloped, auto=false,
                                     allow upside down] {$\subsetneq$} (m-2-6)
																		(m-3-5) edge[draw=none]
                                    node [sloped, auto=false,
                                     allow upside down] {$\subsetneq$} (m-2-6)
											(m-2-6) edge[draw=none]
                                    node [sloped, auto=false,
                                     allow upside down] {$\subsetneq$} (m-1-7)
																		(m-2-6) edge[draw=none]
                                    node [sloped, auto=false,
                                     allow upside down] {$\subsetneq$} (m-3-7);
																		
\end{tikzpicture}
\end{equation}
The $\Sigma_1^{\alpha}$ and $\Pi_1^{\alpha}$ classes are crucial in computer-assisted proofs, since they guarantee algorithms that will not make mistakes (see \S \ref{sec:comp_ass_proofs}). Figure \ref{sigma_meaning} shows the $\Sigma_1^{\alpha}$ and $\Pi_1^{\alpha}$ classes for the Hausdorff metric.

\begin{figure}
\vspace{4mm}
\centering
\begin{minipage}[b]{0.9\textwidth}
    \begin{overpic}[width=1\textwidth,trim={0mm 0mm 0mm 0mm},clip]{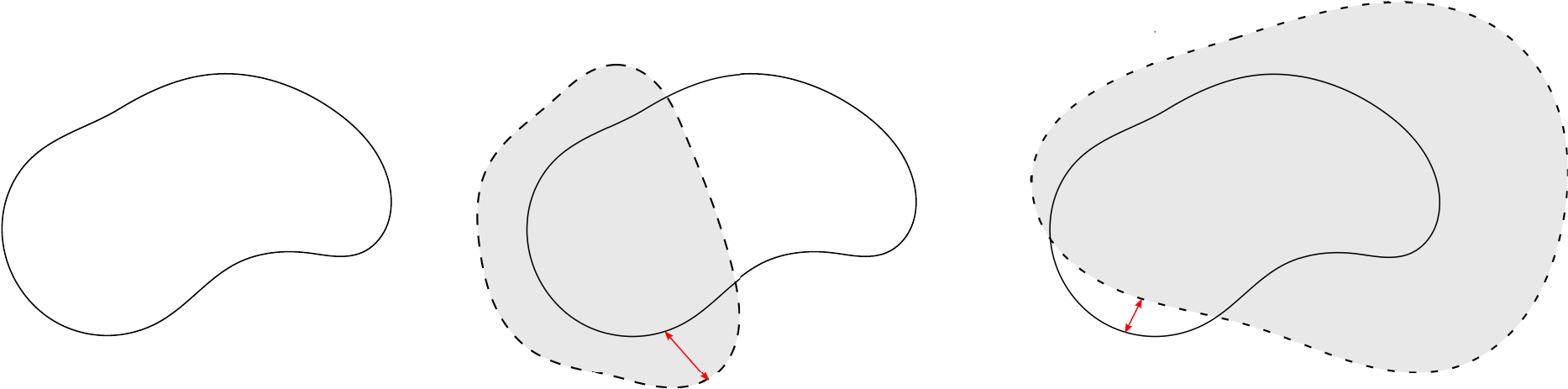}
		\put (9,12) {$\displaystyle \Xi(A)$}
		\put (38,26) {$\displaystyle \Sigma_1$ convergence}
		\put (37,9) {$\displaystyle \Gamma_n(A)$}
		\put (48,0.5) {$\displaystyle\leq 2^{-n}$}
		\put(47.8,1.5){\line(-3,1){4}}
		\put (75,26) {$\displaystyle \Pi_1$ convergence}
		\put (72,11) {$\displaystyle \Gamma_n(A)$}
		\put (75.7,0.5) {$\displaystyle\leq 2^{-n}$}
		\put(75.5,1.5){\line(-1,1){3}}
     \end{overpic}
\end{minipage}
\caption{Meaning of $\Sigma_1$ and $\Pi_1$ convergence for problem function $\Xi$ computed in the Hausdorff metric. The left plot shows the desired set $\Xi(A)$. The shaded areas show the output of the algorithm $\Gamma_n(A)$. $\Sigma_1$ convergence (middle plot) means convergence as $n\rightarrow\infty$ and each output point in $\Gamma_n(A)$ is at most distance $2^{-n}$ from $\Xi(A)$. Similarly, in the case of $\Pi_1$ (right plot), we have convergence as $n\rightarrow\infty$ and any point in $\Xi(A)$ is at most distance $2^{-n}$ from $\Gamma_n(A)$.}
\label{sigma_meaning}
\end{figure}

\begin{remark}[The model of computation $\alpha$]
The $\alpha$ in the superscript indicates the model of computation, which is described in \S \ref{SCI_Hierarchy}. For $\alpha = G$, the underlying algorithm is general and can use any tools at its disposal. The reader may think of a Blum--Shub--Smale (BSS) machine \cite{BCSS} or a Turing machine \cite{turing1937computable} (a general algorithm is more powerful than either model). However, $\alpha = A$ means that only arithmetic operations and comparisons are allowed. In particular, if rational inputs are considered, the algorithm is a Turing machine, and in the case of real inputs, a BSS machine. Hence, a result of the form
$$
\notin \Delta_k^G \text{ is stronger than } \notin \Delta_k^A.
$$  
Indeed, a $\notin \Delta_k^G$ result is universal and holds for any model of computation. Similarly, 
$$
\in \Delta_k^A \text{ is stronger than } \in \Delta_k^G.
$$
Of course, these comments also hold for the $\Pi_k$ and $\Sigma_k$ classes. In this paper, we always prove lower bounds for $\alpha = G$ and upper bounds for $\alpha = A$ (Table \ref{result_table}). Hence, we combine the strongest forms of results in terms of models of computation.
\end{remark}

\subsection{The SCI hierarchy and computer-assisted proofs}
\label{sec:comp_ass_proofs}
$\Delta_1^A$ is the class of problems that are computable according to Turing's definition of computability \cite{turing1937computable}. In particular, there exists an algorithm such that for any $\epsilon > 0$, the algorithm can produce an $\epsilon$-accurate output. Unlike the finite-dimensional case, most infinite-dimensional spectral problems are not in $ \Delta_1^A.$ The simplest way to see this is to consider the problem of computing spectra of infinite diagonal matrices. This problem is the simplest of infinite computational spectral problems, but it does not lie in $\Delta_1^A$. Hence, it should come as no surprise that very few interesting infinite-dimensional spectral problems are actually in $\Delta_1^A$. Instead, most existing results on spectral computations provide algorithms that yield $\Delta_2^A$ classification results. This means that an algorithm will converge, but error control may not be possible. 

Problems that are not in $\Delta_1^A$ are computed daily in the sciences, simply because numerical simulations may be suggestive rather than providing a rock-solid truth. Moreover, the lack of error control may be compensated for by comparing with experiments. However, this is not possible in computer-assisted proofs, where $100\%$ rigour is the only approach accepted.  It may be surprising that famous conjectures have been proven with numerical calculations of problems that are not in $\Delta_1^A$. A striking example is the proof of Kepler's conjecture \cite{Hales_Annals, hales_Pi}, where the decision problems computed are not in $\Delta_1^A$. The decision problems are of the form of deciding feasibility of linear programs given irrational inputs (shown in \cite{opt_big} to not lie in $\Delta_1^A$). Similarly, to prove the Dirac-Schwinger conjecture, asymptotics of the ground state of the operator 
\[
H_{dZ} = \sum_{k=1}^d(-\Delta_{x_k} - Z|x_k|^{-1}) + \sum_{1\leq j \leq k \leq d}|x_j-x_k|^{-1},
\]
as $Z \rightarrow \infty$ were obtained via a computer-assisted proof \cite{fefferman1990, fefferman1992, fefferman1993aperiodicity,  fefferman1994, fefferman1994_2, fefferman1995, fefferman1996interval, fefferman1996, fefferman1997} by Fefferman and Seco, and relied on problems that were not in $\Delta_1^A$. The SCI hierarchy can describe these paradoxical phenomena.

\subsubsection{The $\Sigma^A_1$ and $\Pi^A_1$ classes}
The key to the above paradoxical phenomena lies in the $\Sigma^A_1$ and $\Pi^A_1$ classes. These classes of problems are larger than $\Delta^A_1$, but can still be used in computer-assisted proofs. Indeed, suppose we consider computational spectral problems that are in $\Sigma^A_1$. In that case, there is an algorithm that will never provide incorrect output. The output may not include the whole spectrum, but it is always sound. Thus, a computer-assisted proof could disprove conjectures about operators never having spectra in a certain area of the complex plane. Similarly, $\Pi^A_1$ problems can be approximated from above, and thus conjectures on the spectrum being in a certain area could be disproved by computer simulations.

In both of the above examples (the proof of the Dirac-Schwinger conjecture and Kepler's conjecture), it is implicitly shown  that the relevant computational problems in the computer-assisted proofs are in $\Sigma^A_1$.

\section{Main Results}
\label{fhwoiuhg}
Our main results are sharp classifications in the SCI hierarchy, with algorithms, settling some of the open classification problems in computational spectral theory. We are concerned with the following problem:

\begin{itemize}
\item[] {\it Given a computational spectral problem, where is it in the SCI hierarchy?}
\end{itemize}
We consider the following four main problems: computing spectra of general differential operators, computing spectra of unbounded operators on graphs, the computational spectral gap problem, and computing discrete spectra with multiplicities.

In addition to the spectrum, we consider the pseudospectrum
\[
\mathrm{Sp}_{\epsilon}(A) := \mathrm{cl}(\{z\in\mathbb{C}: \|R(z,A)\| > 1/\epsilon \}), \quad \epsilon > 0,
\]
where $\mathrm{cl}$ denotes closure and $R(z,A)=(A-z I)^{-1}$. When computing the spectrum of bounded operators, we let $(\mathcal{M},d)$ be the set of all non-empty compact subsets of $\mathbb{C}$ provided with the Hausdorff metric $d=d_{\mathrm{H}}$:
\begin{equation}\label{Hausdorff}
d_\mathrm{H}(X,Y) = \max\left\{\sup_{x \in X} \inf_{y \in Y} d(x,y), \sup_{y \in Y} \inf_{x \in X} d(x,y) \right\},
\end{equation}
where $d(x,y) = |x-y|$ is the usual Euclidean distance.
In the case of unbounded operators, we use the \emph{Attouch--Wets metric} defined by
\begin{equation}\label{eq:Attouch-Wets}
d_{\mathrm{AW}}(C_1,C_2)=\sum_{n=1}^{\infty} 2^{-n}\min\left\{{1,\underset{\left|x\right|\leq n}{\sup}\left|\mathrm{dist}(x,C_1)-\mathrm{dist}(x,C_2)\right|}\right\},
\end{equation}
for $C_1,C_2\in\mathrm{Cl}(\mathbb{C})$. Here, $\mathrm{Cl}(\mathbb{C})$ denotes the set of closed non-empty subsets of $\mathbb{C}$.

\subsection{Computing spectra of differential operators on unbounded domains}
\label{dfohjg}
There is a rich literature on computing spectra of differential operators on bounded domains. The computation is often done with finite element, finite difference or spectral methods by discretising the operator on a suitable finite-dimensional space, and then using algorithms for finite-dimensional matrix eigenvalue problems on the discretised operator \cite{Rappaz1,rappaz1997spectral,boffi2000problem,Boffi2,Annalisa2,zhao2007spurious,Annalisa5,Annalisa3,Snorre1}. However, it is generally unknown how to compute spectra of differential operators on unbounded domains, or where this problem lies in the SCI hierarchy (e.g., is it possible in one limit?). 

For $N \in \mathbb{N}$, consider the operator formally defined on $L^2(\mathbb{R}^d)$ by
\begin{equation}\label{eq:diff_op}
Tu(x)=\sum_{k\in\mathbb{Z}_{\geq0}^d,\left|k\right|\leq N}a_k(x)\partial^ku(x),
\end{equation}
where throughout we use multi-index notation with $\left|k\right|=\max\{\left|k_1\right|,...,\left|k_d\right|\}$ and $\partial^k=\partial^{k_1}_{x_1}\partial^{k_2}_{x_2}...\partial^{k_d}_{x_d}$. We assume that the coefficients $a_k(x)$ are complex-valued functions on $\mathbb{R}^d$, and that $T$ can be defined on an appropriate domain $\mathcal{D}(T)$ such that $T$ is closed with non-empty spectrum. Our aim is to compute the spectrum and $\epsilon$-pseudospectrum of $T$ from the coefficients $a_k$. We consider two cases. First, the algorithm can point sample the coefficients, and second, the algorithm can access the coefficients in a Taylor series of each of the coefficients\footnote{We take Taylor series about the origin, but any point will do.} $a_k$ in the case that the $a_k$ are entire. Note that these are two very different computational problems.

\begin{remark}[The open problem of computing spectra of differential operators]
There is no existing theory guaranteeing even a finite SCI for this problem, even when each $a_k$ is a polynomial. For example, a standard procedure is to discretise the differential operator via finite differences, truncate the resulting infinite matrix, and then handle the finite matrix with standard algorithms designed for finite-dimensional problems. Such an approach would at best give a $\Delta^A_2$ classification, and, in general, this approach may not always converge. Despite this, we prove below that one can achieve $\Sigma^A_1$ classification for a large class of operators.
\end{remark}

\subsubsection{The set-up}
We let $\Omega$ consist of all such $T$ such that the following assumptions hold:

\begin{enumerate}
	\item[(1)] The set $C_{0}^\infty(\mathbb{R}^d)$ of smooth, compactly supported functions forms a core of $T$ and its adjoint $T^*$.
	\item[(2)] The adjoint operator $T^*$ can be initially defined on $C_{0}^\infty(\mathbb{R}^d)$ via
	$$
  T^*u(x)=\sum_{k\in\mathbb{Z}_{\geq0}^d,\left|k\right|\leq N}\tilde{a}_k(x)\partial^ku(x),
  $$
	where $\tilde{a}_k(x)$ are complex-valued functions on $\mathbb{R}^d$.
	\item[(3)] For each of the functions $a_k(x)$ and $\tilde{a}_k(x)$, there exists $A_k>0$ and $B_k\in\mathbb{Z}_{\geq0}$ such that 
	$$
	\max\left\{\left|a_k(x)\right|,\left|\tilde{a}_k(x)\right|\right\}\leq A_k\left(1+\left|x\right|^{2B_k}\right),\quad \forall x\in \mathbb{R}^d.
	$$
	That is, we have at most polynomial growth.
	\item[(4)] We have access to a sequence $\{g_m\}_{m\in\mathbb{N}}$ of strictly increasing continuous functions $g_m:\mathbb{R}_{\geq 0}\rightarrow\mathbb{R}_{\geq 0}$, that vanish at zero and diverge at infinity, such that
\begin{equation}
\label{comp_ball_gs2}
g_m(\mathrm{dist}(z,\mathrm{Sp}(T)))\leq\left\|R(z,T)\right\|^{-1},\quad \forall z\in B_m(0),
\end{equation}
where $B_m(0)$ is the closed ball of radius $m$ about the origin. In this case, we say that $T$ has \textit{resolvent bounded by $\{g_m\}$}. This implicitly assumes that $\mathrm{Sp}(T)$ (and hence each $\mathrm{Sp}_{\epsilon}(T)$) is non-empty.
\end{enumerate}

We consider the operator $T$ defined as the closure of \eqref{eq:diff_op} initially defined on $C_{0}^\infty(\mathbb{R}^d)$. The initial domain $C_0^\infty(\mathbb{R}^d)$ is commonly encountered in applications, and it is straightforward to adapt our methods to other initial domains such as Schwartz space.  

\begin{remark}
To handle non-self-adjoint operators, we need to control the resolvent as in \eqref{comp_ball_gs2}. Without such control, the spectral problem is not in $\Delta^G_2$, even for tridiagonal infinite matrices. If $T$ has $\mathrm{Sp}(T)\neq \emptyset$, a simple compactness argument shows the existence of a suitable sequence $\{g_m\}$. We may not be able to control the growth of the resolvent across the whole complex plane by a single function. For self-adjoint (and, more generally, normal) $T$, we can take $g_m(x)=x$. Operators with $g_m(x) = x$ are known as $G_1$ and include the well studied class of hyponormal operators (operators $A$ with $A^*A-AA^*\geq0$) \cite{putnam1979operators}. There are examples where suitable functions $\{g_m\}$ not equal to the identity are known for non-normal operators, such as perturbations of self-adjoint operators \cite[e.g., Theorem 7.7.1]{gil2003operator}. As another example, if an operator is similar to a normal operator with a similarity transformation $S$ that has bounded condition number $\kappa(S)$, we can take $g_m(x)=x/\kappa(S)$. Nonetheless, in general, knowledge of $\{g_m\}$ is a strong assumption on the behaviour of the resolvent and may be difficult to apply to practical examples. However, in what follows, the functions $\{g_m\}$ are \textit{not needed} to compute pseudospectra.
\end{remark}

\subsubsection{General case with function evaluations}

In this section, we treat the computation of spectra and pseudospectra of $T\in\Omega$ from point evaluations of the coefficients $a_k$ and $\tilde{a}_k$. For dimension $d$ and $r>0$, consider the space
\begin{equation}
\mathcal{A}_r=\{f\in M([-r,r]^d):\left\|f\right\|_{\infty}+\mathrm{TV}_{[-r,r]^d}(f)<\infty\},
\end{equation}
where $M([-r,r]^d)$ denotes the set of measurable functions on the hypercube $[-r,r]^d$ and $\mathrm{TV}_{[-r,r]^d}$ the total variation norm in the sense of Hardy and Krause \cite{niederreiter1992random}. This space is a Banach algebra when equipped with the norm
$
\left\|f\right\|_{\mathcal{A}_r}=\left\|f\right\|_{\infty}+(3^d+1)\mathrm{TV}_{[-r,r]^d}(f).
$
We assume that each of the (appropriate restrictions of) $a_k$ and $\tilde{a}_k$ lie in $\mathcal{A}_r$ for all $r>0$, and that we are given a sequence of positive numbers $\{c_n\}_{n\in\mathbb{N}}$ such that
\begin{equation}
\label{tot_var_bound}
\max\{\left\|a_k\right\|_{\mathcal{A}_n},\left\|\tilde{a}_k\right\|_{\mathcal{A}_n}\}\leq c_n,\quad  n\in\mathbb{N}\text{ and } \left|k\right|\leq N.
\end{equation}

This information is entirely analogous to using bounded dispersion for matrix problems encountered in \S \ref{gen_unbd}. We shall see that it cannot be omitted if one wishes to gain error control in the sense of $\Sigma_1$.
Let 
\[
\Omega_{\mathrm{TV}}^1 = \{T \in\Omega\, | \, \text{ such that assumptions (1) -- (4) and \eqref{tot_var_bound} hold}\}.
\]
A special subclass of $\Omega_{\mathrm{TV}}^1$ are Schr\"odinger operators $-\Delta+V$. The fact that computing spectra and pseudospectra of Schr\"odinger operators (by point sampling the potential) with bounded potentials of bounded total variation lies in $\Sigma_1^A$ was shown in \cite{ben2015can}. (Unbounded sectorial potentials without total variation bounds, that induce a compact resolvent, were also treated in \cite{ben2015can} without error control.) Part of Theorem \ref{PDE1} generalises this result to arbitrary differential operators with polynomially bounded coefficients. We let $\Lambda^1$ contain functions that point sample the functions $\{g_m\}_{m\in\mathbb{N}}$ at points in $\mathbb{Q}_{\geq 0}$ and $\{a_k,\tilde{a}_k\}_{\left|k\right|\leq N}$ at points in $\mathbb{Q}^d$, as well as the constants $\{A_k,B_k\}_{\left|k\right|\leq N},\{c_n\}_{n\in\mathbb{N}}$. Consider the weaker assumption on $\Lambda^1$ that we can evaluate $b_n>0$ (and not the $A_k$, $B_k$'s and the $c_n$'s) such that 
\begin{equation}
\sup_{n\in\mathbb{N}}\frac{\max\{\left\|a_k\right\|_{\mathcal{A}_n},\left\|\tilde{a}_k\right\|_{\mathcal{A}_n}:\left|k\right|\leq N\}}{b_n}<\infty.
\end{equation}
With a slight abuse of notation, we use $\Omega_{\mathrm{TV}}^2$ to denote the class of problems where we have this weaker requirement. We can now define the mappings
\begin{align*}
&\Xi_1^k: \Omega_{\mathrm{TV}}^k\ni T\mapsto \mathrm{Sp}(T)\in \mathrm{Cl}(\mathbb{C}),\quad \text{for }k=1,2,\\
&\Xi_2^k: \Omega_{\mathrm{TV}}^k\ni T\mapsto \mathrm{Sp}_{\epsilon}(T) \in \mathrm{Cl}(\mathbb{C}),\quad \text{for } k=1,2,
\end{align*}
where we equip $\mathrm{Cl}(\mathbb{C})$ with the Attouch--Wets metric $d_{\mathrm{AW}}$. The following theorem contains our result.

\vspace{1mm}
\begin{theorem}[Differential operators and point samples]
\label{PDE1}
Let $\Xi_j^1,\Xi_j^2,\Omega_{\mathrm{TV}}^1$ and $\Omega_{\mathrm{TV}}^2$ be as above. Then for $j=1$ or $2$, we have that
$$
\Delta^G_1\not\owns\{\Xi_j^1, \Omega_{\mathrm{TV}}^1\} \in \Sigma^A_1\quad \text{and} \quad
\Sigma^G_1\cup\Pi^G_1\not\owns \{\Xi_j^2, \Omega_{\mathrm{TV}}^2\} \in \Delta^A_2.
$$
\end{theorem}
\vspace{1mm}

\begin{remark}
The proof of Theorem \ref{PDE1} relies on Theorem \ref{unbounded_theorem} that covers unbounded operators on graphs. Thus, in the proof sections below, the theorems are proven in a different order than they are presented.
\end{remark}

\begin{remark}
\label{FEM_remark2}
The proof also shows that even if the information $\{A_k,B_k\}_{\left|k\right|\leq N}$ is added to the evaluation set for operators in $\Omega_{\mathrm{TV}}^2$, we would still have $\{\Xi_j^2, \Omega_{\mathrm{TV}}^2\}\notin \Sigma^G_1\cup\Pi^G_1$. Though we have chosen $\mathbb{R}^d$ as the geometrical domain of our operators, the result can easily be adapted to other domains for which we can build a suitable basis to represent the operator. Examples include the half-line (e.g., for radially symmetric Dirac operators in quantum chemistry), intervals using orthogonal polynomial series, or products of the above geometries. One can also extend our results to more complicated domains using finite elements, non-orthogonal bases, and generalised pencil eigenvalue problems, but this will be the topic of future work.
\end{remark}

\subsubsection{Entire coefficients}
\label{an_coeffs}
In this section, we assume that the functions $a_k$ and $\tilde{a}_k$ are entire. In particular, we assume we can evaluate $\{c_j\}_{j\in \mathbb{N}}$, an enumeration (where we know the ordering) of the coefficients $a_k^m$ where $a_k(x) = \sum_{m\in(\mathbb{Z}_{\geq0})^d} a_k^mx^m$. Note that this means we can compute the corresponding coefficients of the $\tilde{a}_k(x)$ using finitely many arithmetic operations on $\{c_j\}$. As well as the information $\{g_m\}$, $\{c_j\}$ and $\{A_k,B_k\}$, we assume our algorithms can read the following information. Given
$$
a_k(x)=\sum_{m\in(\mathbb{Z}_{\geq 0})^d}a_k^mx^m,\quad \tilde{a}_k(x)=\sum_{m\in(\mathbb{Z}_{\geq 0})^d}\tilde{a}_k^mx^m,
$$
for each $n\in\mathbb{N}$ we know a constant $d_n$ such that
\begin{equation}
\label{bound_tail}
\left|a_k^m\right|,\left|\tilde{a}_k^m\right|\leq d_n(n+1)^{-\left|m\right|}, \qquad \forall m\in(\mathbb{Z}_{\geq 0})^d, |k| \leq N.
\end{equation}
Suitable $d_n$ must exist since the power series converges absolutely on the whole of $\mathbb{R}^d$. Let 
\[
\Omega^1_{\mathrm{AN}} = \{T \in\Omega\, | \, \text{ such that assumptions (1) -- (4) hold, the functions $a_k$ are entire and \eqref{bound_tail} hold}\}.
\]
We let $\Lambda^1$ contain functions that point sample the functions $\{g_m\}_{m\in\mathbb{N}}$ at points in $\mathbb{Q}_{\geq 0}$, and access the constants $\{A_k,B_k\}_{\left|k\right|\leq N}$, $\{c_n\}_{n\in\mathbb{N}}$, and $\{d_n\}_{n \in \mathbb{N}}$. The proof makes clear that $\{d_n\}_{n \in \mathbb{N}}$ can be replaced by any suitable information that allows us to control the remainder term in the truncated Taylor series uniformly on compact subsets of $\mathbb{R}^d$. For example, we could use Cauchy's formula, together with bounds on the functions $a_k$ on compact subsets of $\mathbb{C}^d$. We also consider a weaker requirement on $\Lambda^1$ by replacing knowledge of $A_k,B_k$ and $d_n$ by some sequence of positive numbers $b_n$ with
$$
\sup_{n\in\mathbb{N}}\sup_{m\in(\mathbb{Z}_{\geq 0})^d}\frac{\max\{\left|a_k^m\right|(n+1)^{\left|m\right|},\left|\tilde{a}_k^m\right|(n+1)^{\left|m\right|}:\left|k\right|\leq N\}}{b_n}<\infty.
$$
With a slight abuse of notation, we use $\Omega_{\mathrm{AN}}^2$ to denote the class of problems where we have this weaker requirement. Moreover, let $\Omega_p$ denote the class of operators in $\Omega^2_{\mathrm{AN}}$ such that each $a_k$ is a polynomial (where we can let $b_n=n!$, for example).
We can now define the mappings
\begin{align*}
&\Xi_1^{k+2}: \Omega_{\mathrm{AN}}^k\ni T\mapsto \mathrm{Sp}(T)\in \mathrm{Cl}(\mathbb{C}),\quad \text{for }k=1,2,\\
&\Xi_2^{k+2}: \Omega_{\mathrm{AN}}^k\ni T\mapsto \mathrm{Sp}_{\epsilon}(T) \in \mathrm{Cl}(\mathbb{C}),\quad \text{for } k=1,2,
\end{align*}
where we equip $\mathrm{Cl}(\mathbb{C})$ with the Attouch--Wets metric $d_{\mathrm{AW}}$. The following theorem contains our result.
\vspace{1mm}
\begin{theorem}[Differential operators and entire coefficients]
\label{PDE2}
Let $\Xi_j^3,\Xi_j^4,\Omega_{\mathrm{AN}}^1$, $\Omega_{\mathrm{AN}}^2$ and $\Omega_p$ be as above. Then for $j=1$ or $2$, we have that
\[
\Delta^G_1\not\owns\{\Xi_j^3, \Omega_{\mathrm{AN}}^1\} \in \Sigma^A_1,\quad \Sigma^G_1\cup\Pi^G_1\not\owns \{\Xi_j^4, \Omega_{\mathrm{AN}}^2\} \in \Delta^A_2, \quad \Sigma^G_1\cup\Pi^G_1\not\owns \{\Xi_j^4, \Omega_p\} \in \Delta^A_2.
\]
\end{theorem}
\vspace{1mm}

The new algorithms in Theorems \ref{PDE1} and \ref{PDE2} yielding the above $\Sigma^A_1$ results on unbounded domains are also useful on bounded domains. Standard algorithms for computing spectra of differential operators on bounded domains often have results on qualitative convergence rates. However, typically they do not have the above feature of error control. Moreover, it can be challenging to determine which portion of the output of standard algorithms can be trusted. This well-known problem occurs even if the algorithm is convergent \cite{howmany}, and when this happens, the algorithms cannot be used for computer-assisted proofs. In the language of the SCI hierarchy, these standard algorithms provide, at best, $\Delta^A_2$ classifications of the problems and not the sharp $\Sigma^A_1$ classification. Hence, we draw the following conclusion:
\begin{changemargin}{-0.3cm}{-0.3cm}
\begin{displayquote}
\vspace{1mm}
\normalsize
{\it Computing spectra of differential operators through discretising the operator and computing eigenvalues of the resulting finite matrix is typically not an optimal method. Such methods may not yield the sharp $\Sigma^A_1$ classification providing certainty about the output. However, as demonstrated above, optimal algorithms exist that provide error control and certainty about the computed output.}\vspace{1mm}
\end{displayquote}
\end{changemargin}

\subsection{Computing spectra of unbounded operators on graphs}
\label{gen_unbd}
Given a closed operator $A$ with domain $\mathcal{D}(A)\subset l^2(\mathbb{N})$ and non-empty spectrum, we consider the problem functions
\[
\Xi_1(A) = \mathrm{Sp}(A), \quad \Xi_2(A) = \mathrm{Sp}_{\epsilon}(A).
\]
Let $\mathcal{C}(l^2(\mathbb{N}))$ denote the set of closed, densely defined operators on $l^2(\mathbb{N})$, and consider the following assumptions:
\begin{itemize}
\item[(1)] The subspace $\mathrm{span}\{e_n:n\in\mathbb{N}\}$ forms a core of $A$ and $A^*$, where $\{e_j\}_{j\in\mathbb{N}}$ is the canonical basis for $l^2(\mathbb{N})$. This allows us to associate an infinite matrix $A_{ij}=\langle Ae_j,e_i\rangle$ with $A$, and ensures that the operator is fully determined by its action on finite sums of basis functions (e.g., see Theorem \ref{unif_conv_gamma}).
\item[(2)] Given $f:\mathbb{N}\rightarrow\mathbb{N}$ with $f(n)\geq n$, we define 
\begin{equation}
\begin{split}
D_{f,n}(A) :=\max\big\{{\left\|(I-P_{f(n)})AP_n\right\|},{\left\|(I-P_{f(n)})A^*P_n\right\|}\big\},
\end{split}
\label{bd_disp}
\end{equation}
where $P_n$ is the projection onto the span of $\{e_1,\hdots, e_n\}$ of the canonical basis. We say that an operator has \textit{bounded dispersion} with respect to $f$ if $\lim_{n\rightarrow\infty}D_{f,n}(A)=0$. We assume knowledge of a null sequence $\{c_n\}_{n\in \mathbb{N}} \subset \mathbb{Q}$ and such an $f$ with $D_{f,n}(A)\leq c_n$.
\item[(3)] As in the case of \S \ref{gen_unbd}, we have access to functions $\{g_m\}$ (see \eqref{comp_ball_gs2} and the assumptions on $\{g_m\}$) such that
\begin{equation}
\label{comp_ball_gs}
g_m(\mathrm{dist}(z,\mathrm{Sp}(A)))\leq\left\|R(z,A)\right\|^{-1},\quad \forall z\in B_m(0).
\end{equation}
Recall that if this holds, we say that $A$ has resolvent bounded by $\{g_m\}$. Note that this implicitly assumes that the spectrum is non-empty.
\end{itemize}

\begin{remark}
The concept of bounded dispersion in \eqref{bd_disp} generalises the notion of a banded matrix. Moreover, given any operator with assumption (1), there exists an $f$ such that $\lim_{n\rightarrow\infty}D_{f,n}(A)=0$. The theorem we prove is for the class of operators that have $\lim_{n\rightarrow\infty}D_{f,n}(A)=0$ given a fixed $f$. The function $f$ is used to construct certain \textit{rectangular} truncations of our operators, which is a key difference to previous methods that typically use \textit{square} truncations.
\end{remark}

\subsubsection{Defining $\Omega$ and $\Lambda$}~

\vspace{1mm}

{\bf Operators on $l^2(\mathbb{N})$:} Let $f$ be as in assumption (2), and let $\hat\Omega$ be the class of all $A\in\mathcal{C}(l^2(\mathbb{N}))$ with non-empty spectrum such that (1) and (2) hold. Given a sequence of functions $g=\{g_m\}$ as in (3), let $\Omega_g$ be the class of all $A \in \hat\Omega$ such that (\ref{comp_ball_gs}) holds. Finally, let $\Omega_D$ denote the set of diagonal operators in $\hat\Omega$.

\vspace{1mm}

{\bf Operators on graphs:} Consider a connected, undirected graph $\mathcal{G}$, such the set of vertices $V=V(\mathcal{G})$ is countably infinite. We treat operators on $l^2(V)$ that are closed and densely defined of the form
\begin{equation}
A=\sum_{v, w\in V}\alpha(v,w)\left|v\right\rangle\left\langle w\right|,
\label{kthNeigh}
\end{equation}
for some $\alpha: V\times V \rightarrow \mathbb{C}$ (below, we assume we can sample $\alpha$). We use the classical Dirac notation in \eqref{kthNeigh}, identifying any $v \in V$ by the element in $\psi_v \in l^2(V)$ such that $\psi_v(v) = 1$ and $\psi_v(w) = 0$ for $w \neq v$. We assume that the linear span of such vectors forms a core of both $A$ and $A^*$. We also assume that for any $v\in V$, the set of vertices $w$ with $\alpha(v,w)\neq 0$ or $\alpha(w,v)\neq 0$ is finite. Let $\Omega^\mathcal{G}$ be the class of all such $A$ with non-empty spectrum and let $\Omega_g^\mathcal{G}$ be the class of operators in $\Omega^\mathcal{G}$ of known $g=\{g_m\}$ such that (\ref{comp_ball_gs}) holds. Finally, we assume that with respect to some given enumeration $v_1,v_2,...$ of $V$, we have access to a function $S:\mathbb{N}\rightarrow\mathbb{N}$ such that if $m>S(n)$, then $\alpha(v_n,v_m)=\alpha(v_m,v_n)=0$.

\vspace{1mm}

{\bf Defining $\Lambda$:}
For operators on $l^2(\mathbb{N})$, $\Lambda$ contains the collection of matrix value evaluation functions $A\mapsto \langle Ae_j,e_i\rangle$, functions describing the dispersion and the family of the functions $\{g_m\}$ controlling the growth of the resolvent. For operators on $l^2(V)$, $\Lambda$ contains the functions $\alpha$, the function $S$ and, in the case of $\Omega_g^\mathcal{G}$, the family $\{g_m\}$.

\vspace{1mm}

We can now state the main result of this section:
\begin{theorem}[Unbounded operators on graphs]
\label{unbounded_theorem}
Let $\Xi_1$ be the problem function $\mathrm{Sp}(\cdot)$ and $\Xi_2$ be the problem function $\mathrm{Sp}_{\epsilon}(\cdot)$ for $\epsilon>0$, where these map into the metric space $(\mathrm{Cl}(\mathbb{C}),d_{\mathrm{AW}})$. Then
\[
\Delta_1^G\not\owns\{\Xi_1,\Omega_D\}\in\Sigma_1^A,\qquad \Delta_1^G\not\owns\{\Xi_1,\Omega_g\}\in\Sigma_1^A, \qquad \Delta_1^G\not\owns\{\Xi_1,\Omega_g^{\mathcal{G}}\}\in\Sigma_1^A,
\]
and
\[
\Delta_1^G\not\owns\{\Xi_2,\Omega_D\}\in\Sigma_1^A, \qquad \Delta_1^G\not\owns\{\Xi_2,\hat\Omega\}\in\Sigma_1^A, \qquad \Delta_1^G\not\owns\{\Xi_2,\Omega^{\mathcal{G}}\}\in\Sigma_1^A.
\]
Furthermore, the routines \texttt{CompSpecUB} and \texttt{PseudoSpecUB} in Appendix \ref{append_pseudo} realise the sharp $\Sigma^A_1$ inclusions, and in the case of $\Xi_2$, the output is guaranteed to be inside the true pseudospectrum.
\end{theorem}

\begin{remark}The algorithm used to compute the pseudospectrum can be applied to cases where the spectrum or pseudospectrum are empty, and we provide a computational example of this below.\end{remark}

Finally, we consider two discrete problems, which also include the case when the spectrum is empty. Let $K$ be a non-empty and compact subset of $\mathbb{C}$ and denote the collection of such subsets by $\mathcal{K}(\mathbb{C})$. Consider
\begin{align*}
\Xi_3:(A,K) \rightarrow&\text{ ``Is $\mathrm{Sp}(A)\cap K=\emptyset$?''}\\
\Xi_4:(A,K) \rightarrow&\text{ ``Is $\mathrm{Sp}_{\epsilon}(A)\cap K=\emptyset$?''}
\end{align*}
Here we consider the space $\mathcal{M}=\{0,1\}$ with the discrete topology, where $1$ is interpreted as ``Yes'' and $0$ as ``No''. Thus our computational problem is a decision problem. The information we consider available to the algorithms in the $l^2(\mathbb{N})$ ($l^2(V(\mathcal{G}))$) case are the matrix elements of $A$ (the functions $\alpha$), the dispersion function $f$ and dispersion bounds $\{c_n\}$ (the finite sets $S_v$) and a sequence of finite sets $K_n\subset\mathbb{Q}+i\mathbb{Q}$, with the property that $d_\mathrm{H}(K_n,K)\leq 2^{-(n+1)}$. The following shows that the discrete problems $\Xi_3$ and $\Xi_4$ are harder than computing the spectrum.

\begin{theorem}[Does a set intersect the spectrum/pseudospectrum?]
\label{unbounded_test_theorem}
We have the following classifications for $j=3,4$:
$$
\Delta^G_2\not\owns\{\Xi_j,\hat\Omega\times\mathcal{K}(\mathbb{C})\}\in\Pi^A_2, \qquad \Delta^G_2\not\owns\{\Xi_j,\Omega_D\times\mathcal{K}(\mathbb{C})\}\in\Pi^A_2,\qquad
\Delta^G_2\not\owns\{\Xi_j,\Omega^{\mathcal{G}}\times\mathcal{K}(\mathbb{C})\}\in\Pi^A_2.
$$
The routines \texttt{TestSpec} and \texttt{TestPseudoSpec} in Appendix \ref{append_pseudo}, used for $\Xi_3$ and $\Xi_4$ respectively, realise the sharp $\Pi^A_2$ classifications. Furthermore, the proof makes clear that the lower bounds also hold when we restrict the allowed compact sets to any \textit{fixed} compact subset of $\mathbb{R}$.
\end{theorem}

\begin{remark}
By considering singletons $K=\{z\}$, we can test whether a point lies in the spectrum or pseudospectrum. Even when restricting to such $K$, the proof shows that the classification remains the same.
\end{remark}

\begin{remark}
One could consider the problem of computing $\inf_{z\in K}\|R(z,A)\|^{-1}$. This quantity is zero if and only if $z\in\mathrm{Sp}(K)$. The problem of computing $\inf_{z\in K}\|R(z,A)\|^{-1}$ has $\mathrm{SCI}=1$ (using the metric space $\mathbb{R}$). Thus Theorem \ref{unbounded_test_theorem} is a demonstration of the following issue. It is often harder to solve the decision problem of whether a convergent sequence has a specific given limit ($0$ in the case of $\Xi_3$), than to compute the limit (in our case $\inf_{z\in K}\|R(z,A)\|^{-1}$, which could be non-zero). As discussed in Remark \ref{rem:needed_rd2} below, we emphasise that this holds regardless of the model of computation and is not an issue of finite-precision or round-off errors. Rather, it is due to the information our algorithms have access to in $\Lambda$. If we had a bound on how close our approximation is to $\inf_{z\in K}\|R(z,A)\|^{-1}$, then we could convert this into a $\Sigma_1^A$ tower for the problems in Theorem \ref{unbounded_test_theorem}. However, such information cannot be computed from our $\Lambda$ and corresponds to a very strong form of global information on the matrix representations of the relevant operators.
\end{remark}

\subsection{The spectral gap problem and classifications of the spectrum}
\label{spec_gap_se_class}
The spectral gap problem has a long tradition and is linked to many important problems such as the Haldane conjecture \cite{golinelli1994finite} and the Yang--Mills mass gap problem in quantum field theory \cite{bombieri2006millennium}. It is fundamental in physics, and \cite{Cubitt} showed that the spectral gap problem is undecidable when considering the thermodynamic limit of finite-dimensional Hamiltonians.

In this paper, we consider the general infinite-dimensional problem. We formulate the question as follows. Let $\widehat{\Omega}_{\mathrm{SA}}$ be the set of all self-adjoint and bounded below operators $A$ on $l^2(\mathbb{N})$, for which the linear span of the canonical basis form a core of $A$. Note that we do not assume that $A$ is bounded above. We say that $A\in\widehat{\Omega}_{\mathrm{SA}}$ is gapped if the minimum of $\mathrm{Sp}(A)$ is an isolated eigenvalue with multiplicity one. Otherwise we say that it is gappless. We also let $\widehat{\Omega}_{\mathrm{D}}$ denote the operators in $\widehat{\Omega}_{\mathrm{SA}}$ that are diagonal and define 
\begin{equation}\label{eq:gap}
\Xi_{\mathrm{gap}}: \widehat{\Omega}_{\mathrm{SA}}, \widehat{\Omega}_{\mathrm{D}} \ni A \mapsto \text{ ``Is the spectrum of } A \text{ gapped?''} 
\end{equation}

The above spectral gap problem is extended as follows. Let $\widehat{\Omega}_{\mathrm{SA}}^f\subset \widehat{\Omega}_{\mathrm{SA}}$ be the subclass of operators that have (known) bounded dispersion with respect to the function $f$. Let $a(A)=\inf\{x:x\in\mathrm{Sp}(A)\}$, then one of the four cases must hold:
\begin{enumerate}
	\item $a(A)$ lies in the discrete spectrum and has multiplicity $1$,
	\item $a(A)$ lies in the discrete spectrum and has multiplicity $>1$,
	\item $a(A)$ lies in the essential spectrum but is an isolated point of the spectrum,
	\item $a(A)$ is a cluster point of $\mathrm{Sp}(A)$.
\end{enumerate}
For example, if $A$ is compact, self-adjoint and non-negative, only (3) or (4) can hold. If $A$ is compact and self-adjoint but has negative eigenvalues, only (1) or (2) can hold. We consider the classification problem $\Xi_{\mathrm{class}}$ which maps $\widehat{\Omega}_{\mathrm{SA}}^f$ to the discrete space $\{1,2,3,4\}$ (with the natural ordering).

\begin{theorem}[Spectral gap and classification]
\label{class_thdfjlwdjfkl}
Let $\Xi_{\mathrm{gap}}$ be as in \eqref{eq:gap} and $\widehat{\Omega}_{\mathrm{SA}}, \widehat{\Omega}_{\mathrm{D}}$ as above. Similarly, let $\Xi_{\mathrm{class}}$ and $\widehat{\Omega}_{\mathrm{SA}}^f$ be as above.   
Then
\[
\Delta^G_2\not\owns\{\Xi_{\mathrm{gap}}, \widehat{\Omega}_{\mathrm{SA}}\} \in \Sigma^A_2,\quad \Delta^G_2\not\owns\{\Xi_{\mathrm{gap}}, \widehat{\Omega}_{\mathrm{D}}\} \in \Sigma^A_2.
\]
In particular, the routine \texttt{SpecGap} in Appendix \ref{append_pseudo} realises the sharp $\Sigma^A_2$ inclusions.
Moreover, 
\[
\Delta^G_2\not\owns\{\Xi_{\mathrm{class}}, \widehat{\Omega}_{\mathrm{SA}}^f\} \in \Pi^A_2,\quad\Delta^G_2\not\owns\{\Xi_{\mathrm{class}}, \widehat{\Omega}_{\mathrm{D}}\} \in \Pi^A_2,
\]
and \texttt{SpecClass} in Appendix \ref{append_pseudo} realises the sharp $\Pi^A_2$ inclusions.

\end{theorem}

\begin{remark}[Diagonal vs. full matrix]
Theorem \ref{class_thdfjlwdjfkl} shows that there is no difference in the classification of the spectral gap problem between the set of diagonal matrices and the collection of full matrices. 
\end{remark}

\subsection{Computing discrete spectra, multiplicities and approximate eigenvectors}
\label{disc_results}

For any normal operator $A$, there is a simple decomposition of $\mathrm{Sp}(A)$ into the discrete spectrum and the essential spectrum, denoted by $\mathrm{Sp}_d(A)$ and $\mathrm{Sp}_{\mathrm{ess}}(A)$ respectively. The discrete spectrum consists of isolated points of the spectrum that are also eigenvalues of finite multiplicity. The essential spectrum has numerous definitions for non-normal operators, but for normal operators is defined as the set of $z$ such that $A-zI$ is not a Fredholm operator.

\subsubsection{When we can bound the dispersion.}

Let $\Omega_{\mathrm{N}}^d$ denote the class of bounded normal operators on $l^2(\mathbb{N})$ with (known) bounded dispersion and with non-empty discrete spectrum. Denote by $\Omega_{\mathrm{D}}^{d}$ the class of bounded diagonal self-adjoint operators in $\Omega_{\mathrm{N}}^d$.  Define the problem function
\begin{equation}\label{eq:disc}
\Xi_1^d: \Omega_{\mathrm{N}}^d, \Omega_{\mathrm{D}}^d \ni A \mapsto \mathrm{cl}(\mathrm{Sp}_d(A)). 
\end{equation}
We take the closure and restrict to operators with non-empty discrete spectrum since we want convergence with respect to the Hausdorff metric. However, the algorithm we build, $\Gamma_{n_2,n_1}$, has the property that $\lim_{n_1\rightarrow\infty}\Gamma_{n_2,n_1}(A) \subset  \mathrm{Sp}_d(A)$, so this is not restrictive in practice.

We also let $\Omega_{\mathrm{N}}^f$ denote the class of bounded normal operators with (known) bounded dispersion with respect to the function $f$. In addition, let $\Omega_{\mathrm{D}}$ denote the class of bounded diagonal self-adjoint operators and consider the following discrete problem function
\begin{equation}\label{eq:disc2}
\Xi_2^d: \Omega_{\mathrm{N}}^f, \Omega_{\mathrm{D}} \ni A \mapsto \text{ ``Is }\mathrm{Sp}_d(A)\neq\emptyset{}\text{?''} 
\end{equation}
For $\Xi_2^d$ we consider the space $\mathcal{M}=\{0,1\}$ with the discrete topology, where $1$ is interpreted as ``Yes'' and $0$ as ``No''. Thus the computational problem is a decision problem.

\begin{theorem}
\label{discreteojoiojo}
Let $\Xi_1^d$, $\Omega_{\mathrm{N}}^d$ and $\Omega_{\mathrm{D}}^d$, as well as $\Xi_2^d$, $\Omega_{\mathrm{N}}^f$ and $\Omega_{\mathrm{D}}$, be as above. Then,
\[
\Delta^G_2\not\owns\{\Xi_1^d,\Omega_{\mathrm{N}}^d\} \in\Sigma^A_2,\quad
\Delta^G_2\not\owns\{\Xi_1^d,\Omega_{\mathrm{D}}^d\} \in\Sigma^A_2
\]
and
\[
\Delta^G_2\not\owns\{\Xi_2^d,\Omega_{\mathrm{N}}^f\} \in\Sigma^A_2,\quad\Delta^G_2\not\owns\{\Xi_2^d,\Omega_{\mathrm{D}}\} \in\Sigma^A_2.
\]
In particular, the routines \texttt{DiscreteSpec} and \texttt{DiscSpecEmpty} in Appendix \ref{append_pseudo} realise the sharp $\Sigma_2^A$ inclusions for $\Xi_1^d$ and $\Xi_2^d$ respectively.
\end{theorem}

The constructed algorithm $\Gamma_{n_2,n_1}$ (routine \texttt{DiscreteSpec}) has the following property. Given $A\in\Omega_{\mathrm{N}}^d$ and $z\in\mathrm{Sp}_d(A)$, the following holds. If $\epsilon>0$ is such that $\mathrm{Sp}(A)\cap B_{2\epsilon}(z)=\{z\}$, there is at most one point in $\Gamma_{n_2,n_1}(A)$ that also lies in $B_{\epsilon}(z)$. In other words, any point of $\mathrm{Sp}_d(A)$ has at most one point in $\Gamma_{n_2,n_1}(A)$ approximating it. Furthermore, the limit $\lim_{n_1\rightarrow\infty}\Gamma_{n_2,n_1}(A)=\Gamma_{n_2}(A)$ is contained in the discrete spectrum and increases to $\mathrm{cl}(\mathrm{Sp}_d(A))$ in the Hausdorff metric.

\subsubsection{Eigenvectors and multiplicities.}

Suppose that $z_{n_2,n_1}\in\Gamma_{n_2,n_1}(A)$ (the output of \texttt{DiscreteSpec}) with $$\lim_{n_2\rightarrow\infty}z_{n_2,n_1}=z_{n_2}=z\in \mathrm{Sp}_d(A).$$
Our tower also computes a function $h_{n_2,n_1}(A,\cdot)$ over the output $\Gamma_{n_2,n_1}(A)$ such that
$$\lim_{n_2\rightarrow\infty}\lim_{n_1\rightarrow\infty}h_{n_2,n_1}(A,z_{n_2,n_1})=h(A,z)$$
(where $h(A,z)$ denotes the multiplicity of the eigenvalue $z$) in $\mathbb{Z}_{\geq0}$ with the discrete metric. The routine \texttt{Multiplicity} in Appendix \ref{append_pseudo} computes $h_{n_2,n_1}$.

\texttt{ApproxEigenvector} in Appendix \ref{append_pseudo} approximates eigenvectors. For simplicity, we stick to eigenspaces of multiplicity $1$, but these ideas can be easily extended to higher multiplicities to approximate the whole eigenspace. Given $z_{n_1}$ in the output $\Gamma_{n_2,n_1}(A)$ of the algorithm \texttt{DiscreteSpec} and an approximation
\begin{equation}
\label{sing_val_bound}
\sigma_{\mathrm{inf}}(P_{f(n_1)}(A-z_{n_1}I)|_{P_{n_1}\mathcal{H}})\leq E({n_1},z_{n_1}),
\end{equation}
where $\sigma_{\mathrm{inf}}$ denotes the smallest singular value, can we find a $x_{n_1}$ of unit norm satisfying $\left\|(A-z_{n_1}I)x_{n_1}\right\|\leq E({n_1},z_{n_1})+c_{n_1}$ (recall that $c_{n_1}$ is the dispersion bound)? The discussion in \S \ref{disc_num} shows that such a sequence is an approximate eigenvector sequence.

\begin{theorem}
\label{evector_theorem}
Suppose $A\in\Omega_{\mathrm{N}}^f$. Let $\delta>0$ and $z_{n_1}\in\Gamma_{n_2,n_1}(A)$ such that $z_{n_1}\rightarrow z\in\mathrm{Sp}_d(A)$. Suppose we also have the computed bound (\ref{sing_val_bound}), then we can compute a corresponding vector $x_{n_1}$ (of finite support) satisfying 
$$
\left\|(A-z_{n_1}I)x_{n_1}\right\|<\left\|x_{n_1}\right\|\big(E({n_1},z_{n_1})+c_{n_1}+\delta\big)\text{ and }1-\delta<\left\|x_{n_1}\right\|<1+\delta
$$
in finitely many arithmetic operations. 
\end{theorem}

\subsubsection{What happens when we cannot bound the dispersion?}

Whilst Theorem \ref{discreteojoiojo} shows that computing the discrete spectrum requires two limits, the constructed tower of algorithms $\{\Gamma_{n_2,n_1}\}$ is still useful since
$
\lim_{n_1\rightarrow\infty}\Gamma_{n_2,n_1}(A)\subset\mathrm{Sp}_d(A).
$
Moreover, Theorem \ref{evector_theorem} shows that we can still effectively approximate eigenspaces with error control. But what happens if we do not know a dispersion function $f$ as in \eqref{bd_disp}? To answer this, let $\Omega_1^d$ denote the class of bounded normal operators with non-empty discrete spectrum and $\Omega_2^d$ the class of bounded normal operators. As the following theorem reveals, we get a jump in the SCI hierarchy. 

\begin{theorem}
\label{discrete_dont_now_f}
Let $\Xi_i^d$ and $\Omega_i^d$ be as above. Then,
\[
\Delta^G_3\not\owns\{\Xi_1^d,\Omega_1^d\} \in\Sigma^A_3
\quad \text{and} \quad \Delta^G_3\not\owns\{\Xi_2^d,\Omega_2^d\} \in\Sigma^A_3.
\]
\end{theorem}

\subsubsection{Spectral classification is a much harder problem than the spectrum.}\label{sec:spec_class_lit_added} Theorems \ref{discreteojoiojo} and \ref{discrete_dont_now_f} show that computing spectral classifications is a much harder problem than computing the spectrum (Theorem \ref{unbounded_theorem}). This difficulty is reflected in software packages such as SLEIGN(2) \cite{bailey1978automatic,bailey2001algorithm}, SLEDGE \cite{pruess1993mathematical,eastham1996using,fulton2005automatic,fulton2005computing}, and MATSLISE(2) \cite{ledoux2005matslise,ledoux2016matslise} for Sturm--Liouville problems. Even for such structured problems in one dimension, it very difficult to develop reliable algorithms that classify the spectrum. Similar problems occur when counting the number of negative bound states of suitable Schr\"odinger operators \cite{marletta1991certification}.

As well as the classification into discrete and essential spectrum, one can consider the absolutely continuous, singular continuous and pure point parts of the spectrum. These computations also require more than one limit (three in the case of singular continuous spectra), both for the relevant spectral sets, and the corresponding spectral measures \cite{colbrook2021computing,colbrook2021computingSIREV}. However, computing the full spectral measures can be done in one limit. For applications of spectral measures using these algorithms, see \cite{colbrook2021computingTI,johnstone2021bulk}.

\section{Connection to Previous Work}\label{sec:connection_work}
{\bf The SCI hierarchy:} Our paper is part of the program on the SCI hierarchy \cite{hansen2011solvability,ben2015can,colb1,colb2,boche2019solvability,colbrook2021can,firenet_SIAM_NEWS,colbrook2022computing,colbrook2021computing,colbrook4,colbrook2020foundations,colbrookPSEUDO, opt_big, CRAS_SCI, Hansen2016ComplexityII, colbrook2021computingSIREV, benartzi2020computing, Jonathan_res}, which is a direct continuation of S. Smale's work and his program on the foundations of computational mathematics \cite{Smale2, Smale_Acta_Numerica, BSS_Machine, BCSS}. Related to our paper are the results by C. McMullen\cite{McMullen1, mcmullen1988braiding} and P. Doyle \& C. McMullen \cite{Doyle_McMullen} on polynomial root-finding, which are classification results in the SCI hierarchy, and the contributions by  L. Blum, F. Cucker, M. Shub \& S. Smale  \cite{BSS_Machine, BCSS, ShubDuke, Cucker_AH_real}. Further examples are the results by C. Fefferman and L. Seco \cite{fefferman1990, fefferman1992, fefferman1993aperiodicity,  fefferman1994, fefferman1994_2, fefferman1995, fefferman1996interval, fefferman1996, fefferman1997}, proving the Dirac--Schwinger conjecture on the asymptotic behaviour of the ground state energy of a family of Schr\"{o}dinger operators, which implicitly prove $\Sigma^A_1$ classifications in the SCI hierarchy. This is also the case in T. Hales' Flyspeck program \cite{Hales_Annals, hales_Pi} leading to the proof of Kepler's conjecture (Hilbert's 18th problem) which also implicitly proves $\Sigma^A_1$ classifications. Many other problems in the foundations of computations, such as the work by S. Weinberger \cite{Weinberger}, can be viewed in the context of the SCI hierarchy. 

{\bf Classical results on computing spectra:} Due to the vast literature on spectral computation, we can only cite a small subset related to this paper. The ideas of using computational and algorithmic approaches to obtain spectral information date back to leading physicists and mathematicians such as H. Goldstine \cite{Goldstine}, T. Kato \cite{kato1949upper}, F. Murray \cite{Goldstine},  E. Schr\"odinger \cite{schrodinger1940method}, J. Schwinger \cite{Schwinger} and J. von Neumann \cite{Goldstine}. Schwinger introduced finite-dimensional approximations to quantum systems in infinite-dimensional spaces that allow for spectral computations. An interesting observation is that Schwinger's ideas were already present in the work of H. Weyl \cite{weyl1950theory}.
The work by H. Goldstine, F. Murray and J. von Neumann \cite{Goldstine} was one of the first to establish rigorous convergence results, and their work yields $\Delta^A_1$ classification for certain self-adjoint finite-dimensional problems. 
In \cite{Digernes} T. Digernes, V. S. Varadarajan  and S. R. S. Varadhan proved convergence of spectra of Schwinger's finite-dimensional discretisation matrices for a specific class of Schr\"{o}dinger operators with certain types of potential, which yields a $\Delta^A_2$ classification in the SCI hierarchy. 

The finite-section method, which has been intensely studied for spectral computation, and has often been viewed in connection with Toeplitz theory, is very similar to Schwinger's idea of approximation using a finite-dimensional subspace. The reader may want to consult the pioneering work by A. B{\"o}ttcher \cite{Albrecht_Fields, Bottcher_pseu} and A. B{\"o}ttcher \& B. Silberman \cite{Bottcher_book, bottcher2006analysis}. 
W. Arveson  \cite{Arveson_cnum_lin94, Arveson_noncommute93,Arveson_role_of94,Arveson_Improper93,  Arveson_discrete91} 
and N. Brown \cite{brown2007quasi, Brown_2006, Brown_Memoars} pioneered the combination of spectral computation and the $C^*$-algebra literature (which dates back to the work of A. B{\"o}ttcher \& B. Silberman \cite{Albrecht1983}), both for the general spectral computation problem as well as for Schr\"{o}dinger operators. See also the work by N. Brown, K. Dykema, and D. Shlyakhtenko \cite{brown2002}, where variants of finite section analysis are implicitly used. Arveson also considered spectral computation in terms of densities, which is related to Szeg{\H o}'s work \cite{Szego} on finite section approximations. Similar results are also obtained by A. Laptev and Y. Safarov \cite{Laptev}. 
Typically, when applied to appropriate subclasses of operators, finite section approaches yield $\Delta^A_2$ classification results. There are also other approaches based on the infinite QR algorithm in connection with Toda flows with infinitely many variables pioneered by P. Deift, L. C. Li, and C. Tomei \cite{deift}. See also the work by P. Deift, J. Demmel, C. Li, and C. Tomei \cite{deift1991bidiagonal}.
E. B. Davies considered second order spectra methods
\cite{Davies00, Davies_sec_order}, and E. Shargorodsky \cite{Shargorodsky1} demonstrated how second order spectra methods \cite{Davies_sec_order} will never recover the whole spectrum.

{\bf Recent results on computing spectra:} There are many recent directions in computational spectral theory that are related to our work. 

\begin{itemize}
\item[(i)] {\it Infinite-dimensional numerical linear algebra:}
S. Olver,  A.Townsend and M. Webb have provided a foundational and practical framework for infinite-dimensional numerical linear algebra and foundational results on computations with infinite data structures \cite{Olver_Townsend_Proceedings, Olver_SIAM_Rev, Olver_code1, Olver_code2}. This includes efficient codes as well as theoretical results. The infinite-dimensional QL and QR algorithms, inspired by the work of Deift et. al. \cite{deift, deift1991bidiagonal} mentioned above, are important parts of this program that yield classifications in the SCI hierarchy of computing extreme elements in the spectrum, see also \cite{Hansen_JFA, colb2} for the infinite-dimensional QR algorithm. The recent work of M. Webb and S. Olver \cite{webb2021spectra} on computing spectra of Jacobi operators is also formulated in the SCI hierarchy.

\item[(ii)] {\it Finite section approaches:} In the cases where the finite section method works, it will typically yield $\Delta_2^A$ classifications in the SCI hierarchy, and occasionally $\Delta_1^A$ classifications, see, for example, the work by A. B{\"o}ttcher, H. Brunner, A. Iserles \& S. N{\o}rsett \cite{Arieh2}, A. B{\"o}ttcher, S. Grudsky \& A. Iserles \cite{bottcher_grudsky_iserles_2011}, H. Brunner, A. Iserles \& S. N{\o}rsett \cite{Arieh_IMA, brunner2011}, M. Marletta \cite{Marletta_pollution} and M. Marletta \& R. Scheichl \cite{Marletta_Spec_gaps}. The latter papers also discuss the failure of the finite section approach for certain classes of operators, see also \cite{Hansen_PRS, Hansen_JFA}.

\item[(iii)] {\it Resonances:} 
We would like to mention the recent work by M. Zworski \cite{Zworski1, zworski1999resonances} on computing resonances that can be viewed in terms of the SCI hierarchy. In particular, the computational approach in \cite{Zworski1} is based on expressing resonances as limits of non-self-adjoint spectral problems, and hence the SCI hierarchy is inevitable, see also \cite{sjostrand1999asymptotic}.  The recent work of J. Ben-Artzi, M. Marletta \& F. R\"{o}sler \cite{Jonathan_res,benartzi2020computing} on computing resonances is also formulated in terms of the SCI hierarchy.

\item[(iv)] {\it Computer-assisted proofs:} We have already mentioned the results by C. Fefferman and L. Seco \cite{fefferman1990, fefferman1992, fefferman1993aperiodicity,  fefferman1994, fefferman1994_2, fefferman1995, fefferman1996interval, fefferman1996, fefferman1997} on computer-assisted proofs proving classification results in the SCI hierarchy. However, recent results using computer-assisted proofs in spectral theory also includes the work of M. Brown, M. Langer, M. Marletta, C. Tretter, \& M. Wagenhofer \cite{brown2010} and S. B{\"o}gli, M.  Brown, M. Marletta, C.  Tretter \& M. Wagenhofer \cite{bogli2014guaranteed}.

\end{itemize}
Finally, since writing this paper, the first author has developed rigorous data-driven algorithms for spectral properties of Koopman operators (operators on infinite-dimensional spaces that globally linearise non-linear dynamical systems) \cite{colbrook2021rigorous,ResDMD_JFM,colbrook2022mpedmd}. For these problems, $\Lambda$ consists of snapshot data of the system.

\subsection*{Acknowledgements}
The authors would like to thank Percy Deift, Charlie Fefferman, Tom Hales, Ari Laptev, Steve Smale, Maciej Zworski and Shmuel Weinberger for helpful discussions. MJC acknowledges support from the UK Engineering and Physical Sciences Research Council (EPSRC) grant EP/L016516/1. ACH acknowledges support from a Royal Society University Research Fellowship, the Leverhulme Prize 2017, as well as the UK Engineering and Physical Sciences Research Council (EPSRC) grant EP/L003457/1.

\section{Mathematical Preliminaries}\label{SCI_Hierarchy}
In this section, we formally define the SCI hierarchy. We have already presented the definition of a computational problem $\{\Xi,\Omega,\mathcal{M},\Lambda\}$ in \S \ref{sec:SCI_hierarchy}. The goal is to find algorithms that approximate the function $\Xi$. More generally, the main pillar of our framework is the concept of a tower of algorithms, which is needed to describe problems that need several limits in the computation. However, first we need the definition of a general algorithm.
\begin{definition}[General Algorithm]\label{Gen_alg}
Given a computational problem $\{\Xi,\Omega,\mathcal{M},\Lambda\}$, a \emph{general algorithm} is a mapping $\Gamma:\Omega\to \mathcal{M}$ such that for each $A\in\Omega$
\begin{itemize}
\item[(i)] There exists a (non-empty) finite subset of evaluations $\Lambda_\Gamma(A) \subset\Lambda$, 
\item[(ii)] The action of $\,\Gamma$ on $A$ only depends on $\{A_f\}_{f \in \Lambda_\Gamma(A)}$ where $A_f := f(A),$
\item[(iii)] For every $B\in\Omega$ such that $B_f=A_f$ for every $f\in\Lambda_\Gamma(A)$, it holds that $\Lambda_\Gamma(B)=\Lambda_\Gamma(A)$.
\end{itemize}
\end{definition}

The definition of a general algorithm is more general than the definition of a Turing machine \cite{turing1937computable} or a Blum--Shub--Smale (BSS) machine \cite{BCSS}. A general algorithm has no restrictions on the operations allowed. The only restriction is that it takes a \textit{finite} amount of information, though it is allowed to \emph{adaptively} choose the finite amount of information it reads depending on the input. Condition (iii) ensures that the algorithm consistently reads the information. With a definition of a general algorithm, we can define the concept of towers of algorithms.

\begin{definition}[Tower of Algorithms]\label{tower_funct}
Given a computational problem $\{\Xi,\Omega,\mathcal{M},\Lambda\}$, a \emph{tower of algorithms of height $k$
 for $\{\Xi,\Omega,\mathcal{M},\Lambda\}$} is a family of sequences of functions
 $$\Gamma_{n_k}:\Omega
\rightarrow \mathcal{M},\ \Gamma_{n_k, n_{k-1}}:\Omega
\rightarrow \mathcal{M},\dots,\ \Gamma_{n_k, \hdots, n_1}:\Omega \rightarrow \mathcal{M},
$$
where $n_k,\hdots,n_1 \in \mathbb{N}$ and the functions $\Gamma_{n_k, \hdots, n_1}$ at the lowest level of the tower are general algorithms in the sense of Definition \ref{Gen_alg}. Moreover, for every $A \in \Omega$,
$$
\Xi(A)= \lim_{n_k \rightarrow \infty} \Gamma_{n_k}(A), \quad \Gamma_{n_k, \hdots, n_{j+1}}(A)= \lim_{n_j \rightarrow \infty} \Gamma_{n_k, \hdots, n_j}(A) \quad j=k-1,\dots,1.
$$
\end{definition}

In addition to a general tower of algorithms (defined above), we will focus on arithmetic towers.

\begin{definition}[Arithmetic Tower]\label{arith_tower_def}
Given a computational problem $\{\Xi,\Omega,\mathcal{M},\Lambda\}$, where $\Lambda$ is countable, we define the following: An {arithmetic tower of algorithms} of height $k$
 for $\{\Xi,\Omega,\mathcal{M},\Lambda\}$ is a tower of algorithms where the lowest functions $\Gamma = \Gamma_{n_k, \hdots, n_1} :\Omega \rightarrow \mathcal{M}$ satisfy the following:
 For all $A\in\Omega$ the mapping $(n_k, \hdots, n_1) \mapsto \Gamma_{n_k, \hdots, n_1}(A) = \Gamma_{n_k, \hdots, n_1}(\{A_f\}_{f \in \Lambda})$ is recursive, and $\Gamma_{n_k, \hdots, n_1}(A)$ is a finite string of complex numbers that can be identified with an element in $\mathcal{M}$. For arithmetic towers we let $\alpha = A$ 
\end{definition} 

\begin{remark}By recursive we mean the following. If $f(A) \in \mathbb{Q}$ (or $\mathbb{Q}+i\mathbb{Q}$) for all $f \in \Lambda$, $A \in \Omega$, and $\Lambda$ is countable, then $\Gamma_{n_k, \hdots, n_1}(\{A_f\}_{f \in \Lambda})$ can be executed by a Turing machine \cite{turing1937computable}, that takes $(n_k, \hdots, n_1)$ as input and an oracle input tape consisting of $\{A_f\}_{f \in \Lambda}$. If $f(A) \in \mathbb{R}$ (or $\mathbb{C}$) for all $f \in \Lambda$, then $\Gamma_{n_k, \hdots, n_1}(\{A_f\}_{f \in \Lambda})$ can be executed by a BSS machine \cite{BCSS} that takes $(n_k, \hdots, n_1)$, as input and an oracle that can access any $A_f$ for $f \in \Lambda$. 
\end{remark}

Given the definitions above, we can now define the key concept - the Solvability Complexity Index: 

\begin{definition}[Solvability Complexity Index]\label{complex_ind}
A computational problem $\{\Xi,\Omega,\mathcal{M},\Lambda\}$ is said to have {Solvability Complexity Index $\mathrm{SCI}(\Xi,\Omega,\mathcal{M},\Lambda)_{\alpha} = k$}, with respect to a tower of algorithms of type $\alpha$, if $k$ is the smallest integer for which there exists a tower of algorithms of type $\alpha$ of height $k$. If no such tower exists, $\mathrm{SCI}(\Xi,\Omega,\mathcal{M},\Lambda)_{\alpha} = \infty.$ If there exists a tower $\{\Gamma_n\}_{n\in\mathbb{N}}$ of type $\alpha$ and height one such that $\Xi = \Gamma_{n_1}$ for some $n_1 < \infty$, we define $\mathrm{SCI}(\Xi,\Omega,\mathcal{M},\Lambda)_{\alpha} = 0$. The type $\alpha$ may be General or Arithmetic, denoted by G and A, respectively. We sometimes write $\mathrm{SCI}(\Xi,\Omega)_{\alpha}$ to simplify notation when $\mathcal{M}$ and $\Lambda$ are obvious. 
\end{definition}

We let $\mathrm{SCI}(\Xi,\Omega)_{\mathrm{A}}$ and $\mathrm{SCI}(\Xi,\Omega)_{\mathrm{G}}$ denote the SCI with respect to an arithmetic tower and a general tower, respectively. Note that a general tower means just a tower of algorithms as in Definition \ref{tower_funct}, where there are no restrictions on the mathematical operations. Thus, clearly $\mathrm{SCI}(\Xi,\Omega)_{\mathrm{A}} \geq \mathrm{SCI}(\Xi,\Omega)_{\mathrm{G}}$. The definition of the SCI immediately induces the SCI hierarchy:

\begin{definition}[The Solvability Complexity Index Hierarchy]
\label{1st_SCI}
Consider a collection $\mathcal{C}$ of computational problems and let $\mathcal{T}$ be the collection of all towers of algorithms of type $\alpha$ for the computational problems in $\mathcal{C}$.
Define 
\begin{equation*}
\begin{split}
\Delta^{\alpha}_0 &:= \{\{\Xi,\Omega\} \in \mathcal{C} \ \vert \   \mathrm{SCI}(\Xi,\Omega)_{\alpha} = 0\}\\
\Delta^{\alpha}_{m+1} &:= \{\{\Xi,\Omega\}  \in \mathcal{C} \ \vert \   \mathrm{SCI}(\Xi,\Omega)_{\alpha} \leq m\}, \qquad \quad m \in \mathbb{N},\\
\text{as well as}\quad\quad\quad
\Delta^{\alpha}_{1} &:= \{\{\Xi,\Omega\}  \in \mathcal{C}   \  \vert \ \exists \ \{\Gamma_n\}_{n\in \mathbb{N}} \in \mathcal{T}\text{ s.t. } \forall A \ d(\Gamma_n(A),\Xi(A)) \leq 2^{-n}\}. 
\end{split}
\end{equation*}
\end{definition}

When there is additional structure on the metric space, such as in the spectral case when one considers the Attouch--Wets or the Hausdorff metric, one can extend the SCI hierarchy.

\begin{definition}[The SCI Hierarchy (Attouch--Wets/Hausdorff metric)]
Given the set-up in Definition \ref{1st_SCI}, suppose in addition that $(\mathcal{M},d)$ has the Attouch--Wets or the Hausdorff metric induced by another background metric space $(\mathcal{M}^{\prime},d')$. Define, for $m \in \mathbb{N}$,
\begin{align*}
\Sigma^{\alpha}_0 &= \Pi^{\alpha}_0 = \Delta^{\alpha}_0,\\
\Sigma^{\alpha}_{1} &= \{\{\Xi,\Omega\} \in \Delta_{2}^{\alpha} \ \vert \  \exists \ \{\Gamma_{n}\} \in \mathcal{T}, \ \text{s.t.} \ \forall \ A\in\Omega \ \lim_{n\rightarrow\infty}\Gamma_{n}(A)=\Xi(A) \ \text{and} \ \exists \ \{X_{n}(A)\}\subset\mathcal{M} \ \text{s.t.} \\
& \qquad \qquad \qquad \qquad \Gamma_{n}(A)  \mathop{\subset}_{\mathcal{M}^{\prime}} X_n(A) \ \text{with} \ d(X_{n}(A),\Xi(A))\leq 2^{-n}\}, \\
\Pi^{\alpha}_{1} &= \{\{\Xi,\Omega\} \in \Delta_{2}^{\alpha} \ \vert \  \exists \ \{\Gamma_{n}\} \in \mathcal{T}, \ \text{s.t.} \ \forall \ A\in\Omega \ \lim_{n\rightarrow\infty}\Gamma_{n}(A)=\Xi(A) \ \text{and} \ \exists \ \{X_{n}(A)\}\subset\mathcal{M} \ \text{s.t.} \\
& \qquad \qquad \qquad \qquad \Xi(A)  \mathop{\subset}_{\mathcal{M}^{\prime}} X_n(A) \ \text{with} \ d(X_{n}(A),\Gamma_n(A))\leq 2^{-n}\},
\end{align*}
where $\mathop{\subset}_{\mathcal{M}^{\prime}}$ means inclusion in the background metric space $(\mathcal{M}^{\prime},d')$, and $\{X_{n}(A)\}\subset\mathcal{M}$ is a sequence that may depend on $A$. Moreover, 
\begin{equation*}
\begin{split}
\Sigma^{\alpha}_{m+1} &= \{\{\Xi,\Omega\} \in \Delta_{m+2}^{\alpha} \ \vert \  \exists \ \{\Gamma_{n_{m+1},...,n_1}\} \in \mathcal{T}, \ \text{s.t.} \ \forall \ A\in\Omega \ \lim_{n_{m+1}\rightarrow\infty}\!...\!\lim_{n_{1}\rightarrow\infty}\Gamma_{n_{m+1},...,n_1}(A)=\Xi(A)\\  
& \quad \text{and} \ \exists \ \{X_{n_{m+1}}(A)\}\subset\mathcal{M} \ \text{s.t.} \ \Gamma_{n_{m+1}}(A)  \mathop{\subset}_{\mathcal{M}^{\prime}} X_{n_{m+1}}(A) \ \text{with} \ d(X_{n_{m+1}}(A),\Xi(A))\leq 2^{-n_{m+1}}\}, \\
\Pi^{\alpha}_{m+1} &= \{\{\Xi,\Omega\} \in \Delta_{m+2}^{\alpha} \ \vert \  \exists \ \{\Gamma_{n_{m+1},...,n_1}\} \in \mathcal{T}, \ \text{s.t.} \ \forall \ A\in\Omega \ \lim_{n_{m+1}\rightarrow\infty}\!...\!\lim_{n_{1}\rightarrow\infty}\Gamma_{n_{m+1},...,n_1}(A)=\Xi(A)\\  
& \quad \text{and} \ \exists \ \{X_{n_{m+1}}(A)\}\subset\mathcal{M} \ \text{s.t.} \ \Xi(A)  \mathop{\subset}_{\mathcal{M}^{\prime}} X_{n_{m+1}}(A) \ \text{with} \ d(X_{n_{m+1}}(A),\Gamma_{n_{m+1}}(A))\leq 2^{-n_{m+1}}\}.
\end{split}
\end{equation*}
In all of the above, $d$ can be either $d_{\mathrm{H}}$ or $d_{\mathrm{AW}}$.
\end{definition}

For example, suppose that $(\mathcal{M}',d')$ is the complex plane $\mathbb{C}$ with the usual metric and consider the Hausdorff metric on non-empty compact subsets of $\mathbb{C}$. A computational problem is in $\Sigma_1^{\alpha}$ if there exists a convergent sequence of algorithms $\{\Gamma_n\}$, such that for input $A$, there exists a sequence of non-empty compact subsets $X_n(A)\subset\mathbb{C}$ with $\Gamma_n(A)\subset X_n(A)$ and $d_{\mathrm{H}}(X_n(A),\Xi(A))\leq 2^{-n}$. Note that to build a $\Sigma_1^{\alpha}$ algorithm, it is enough by taking subsequences of $n$ to construct $\Gamma_n(A)$ such that $\Gamma_{n}(A) \subset \Xi(A)+B_{E_n(A)}(0)$ with some computable $E_n(A)$ that converges to zero. The sequence of sets $\Gamma_n(A)$ thus converges to $\Xi(A)$ and is contained in $\Xi(A)$ up to the arbitrarily small tolerance (convergence from below).

Similarly, a computational problem is in $\Pi_1^{\alpha}$ if there exists a convergent sequence of algorithms $\{\Gamma_n\}$, such that for input $A$, there exists a sequence of non-empty compact subsets $X_n(A)\subset\mathbb{C}$ with $\Xi(A)\subset X_n(A)$ and $d_{\mathrm{H}}(X_n(A),\Gamma_n(A))\leq 2^{-n}$. Note that to build a $\Pi_1^{\alpha}$ algorithm, it is enough by taking subsequences of $n$ to construct $\Gamma_n(A)$ such that $\Xi(A) \subset \Gamma_{n}(A)+B_{E_n(A)}(0)$ with some computable $E_n(A)$ that converges to zero. The sequence of sets $\Gamma_n(A)$ thus converges to $\Xi(A)$ and $\Xi(A)$ is contained in $\Gamma_n(A)$ up to the arbitrarily small tolerance (convergence from above).

The classes $\Sigma_m^{\alpha}$ and $\Pi_m^{\alpha}$ for $m>1$ generalise this notion of convergence from below or, respectively, above in the final limit.

The same extension can be applied to the real line with the usual metric, or to $\{0,1\}$ with the discrete metric. For decision problems, we use $\{0,1\}$, where we interpret $1$ as ``Yes'' and $0$ as ``No''.

\begin{definition}[The SCI Hierarchy (totally ordered set)]
Given the set-up in Definition \ref{1st_SCI}, suppose in addition that $\mathcal{M}$ is a totally ordered set. 
Define 
\begin{equation*}
\begin{split}
\Sigma^{\alpha}_0 &= \Pi^{\alpha}_0 = \Delta^{\alpha}_0,\\
\Sigma^{\alpha}_{1} &= \{\{\Xi,\Omega\} \in \Delta_{2}^{\alpha} \ \vert \  \exists \ \{\Gamma_{n}\} \in \mathcal{T} \text{ s.t. } \Gamma_{n}(A) \nearrow \Xi(A) \ \, \forall A \in \Omega\}, 
\\
\Pi^{\alpha}_{1} &= \{\{\Xi,\Omega\} \in \Delta_{2}^{\alpha} \ \vert \  \exists \ \{\Gamma_{n}\} \in \mathcal{T} \text{ s.t. } \Gamma_{n}(A) \searrow \Xi(A) \ \, \forall A \in \Omega\},
\end{split}
\end{equation*}
where $\nearrow$ and $\searrow$ denotes convergence from below and above respectively,
as well as, for $m \in \mathbb{N}$, 
\begin{equation*}
\begin{split}
\Sigma^{\alpha}_{m+1} &= \{\{\Xi,\Omega\} \in \Delta_{m+2}^{\alpha} \ \vert \  \exists \ \{\Gamma_{n_{m+1}, \hdots, n_1}\} \in \mathcal{T} \text{ s.t. }\lim_{n_{m+1}\rightarrow\infty}\!...\!\lim_{n_{1}\rightarrow\infty}\Gamma_{n_{m+1},...,n_1}(A)=\Xi(A) \\
& \hspace{60mm}\text{and }\Gamma_{n_{m+1}}(A) \nearrow \Xi(A) \ \, \forall A \in \Omega\}, \\
\Pi^{\alpha}_{m+1} &= \{\{\Xi,\Omega\} \in \Delta_{m+2}^{\alpha} \ \vert \  \exists \ \{\Gamma_{n_{m+1}, \hdots, n_1}\} \in \mathcal{T} \text{ s.t. }\lim_{n_{m+1}\rightarrow\infty}\!...\!\lim_{n_{1}\rightarrow\infty}\Gamma_{n_{m+1},...,n_1}(A)=\Xi(A) \\
& \hspace{60mm}\text{and }\Gamma_{n_{m+1}}(A) \searrow \Xi(A) \ \, \forall A \in \Omega\}.
\end{split}
\end{equation*}
\end{definition}

\begin{remark}[$\Delta^{\alpha}_1\subsetneq \Sigma^{\alpha}_1 \subsetneq \Delta^{\alpha}_2$]
Note that the inclusions are strict. For example, if $\Omega_K$ consists of the set of compact infinite matrices acting on $l^2(\mathbb{N})$ and $\Xi(A)=\mathrm{Sp}(A)$ (the spectrum of $A$), then $\{\Xi, \Omega_K\} \in \Delta^{\alpha}_2$ but not in $ \Sigma_1^\alpha\cup\Pi_1^\alpha$ for $\alpha$ representing either towers of arithmetical or general type (see \cite{ben2015can} for a proof). Moreover, as was demonstrated in \cite{colb1}, if $\tilde \Omega$ is the set of discrete Schr\"odinger operators on $l^2(\mathbb{Z})$, then $\{\Xi, \Omega\} \in \Sigma^{\alpha}_1$ but not in $\Delta^{\alpha}_1$.
\end{remark}

Suppose we are given a computational problem $\{\Xi, \Omega, \mathcal{M}, \Lambda\}$, and that $\Lambda = \{f_j\}_{j \in \beta}$, where $\beta$ is some index set that can be finite or infinite. Obtaining $f_j$ may be a computational task in its own right, which is exactly the problem in most areas of computational mathematics. For example, given $A \in \Omega$, $f_j(A)$ could be the number $e^{\frac{\pi}{j} i }$. Hence, we cannot access $f_j(A)$, but rather $f_{j,n}(A)$ where $f_{j,n}(A) \rightarrow f_{j}(A)$ as $n \rightarrow \infty$. 
Or, just as for problems that are high up in the SCI hierarchy, it could be that we need several limits. One may need mappings
$f_{j,n_m,\hdots, n_1}: \Omega \rightarrow \mathbb{Q} + i\mathbb{Q}$ such that 
\begin{equation}\label{Lambda_limits}
\lim_{n_m \rightarrow \infty} \hdots \lim_{n_1 \rightarrow \infty} \|\{f_{j,n_m,\hdots, n_1}(A)\}_{j\in\beta} - \{f_j(A)\}_{j\in\beta}\|_{\infty} = 0 \quad  \forall A \in \Omega.
\end{equation}

In particular, we may view the problem of obtaining $f_j(A)$ as a problem in the SCI hierarchy. Thus, $\Delta_1$ classification would correspond to the existence of mappings $f_{j,n}: \Omega \rightarrow \mathbb{Q} + i \mathbb{Q}$
such that 
 \begin{equation}\label{Lambda_limits2}
 \|\{f_{j,n}(A)\}_{j\in\beta} - \{f_j(A)\}_{j\in\beta}\|_{\infty} \leq 2^{-n} \quad \forall A \in \Omega.
 \end{equation}

The following definition formalises these ideas.

\begin{definition}[$\Delta_{m}$-information]\label{definition:Lambda_limits}
Let $\{\Xi, \Omega, \mathcal{M}, \Lambda\}$ be a computational problem. For $m \in \mathbb{N}$ we say that $\Lambda$ has $\Delta_{m+1}$-information if each $f_j \in \Lambda$ is not available, however, there are mappings $f_{j,n_m,\hdots, n_1}: \Omega \rightarrow \mathbb{Q} + i \mathbb{Q}$ such that \eqref{Lambda_limits} holds. Similarly, for $m = 0$, there are mappings $f_{j,n}: \Omega \rightarrow \mathbb{Q} + i \mathbb{Q}$
	such that \eqref{Lambda_limits2} holds. Finally, if $k \in \mathbb{N}$ and $\hat \Lambda$ is a collection of such functions described above, we say that $\hat \Lambda$ provides $\Delta_k$-information for $\Lambda$. We denote the family of all such $\hat \Lambda$ by $\mathcal{L}^k(\Lambda)$. 
\end{definition}

We want to have algorithms that can handle computational problems $\{\Xi,\Omega,\mathcal{M},\hat \Lambda\}$ for any $\hat \Lambda \in  \mathcal{L}^m(\Lambda)$. To formalise this, we define what we mean by a computational problem with $\Delta_m$-information.

\begin{definition}[Computational problem with $\Delta_m$-information]
	Given $m \in \mathbb{N}$, a computational problem where $\Lambda$ has $\Delta_m$-information is denoted by 
	$
	\{\Xi,\Omega,\mathcal{M},\Lambda\}^{\Delta_m} := \{\tilde \Xi,\tilde \Omega,\mathcal{M},\tilde \Lambda\},
	$ 
where 
\[
\tilde \Omega = \left\{ \tilde A = \{f_{j,n_m,\hdots, n_1}(A)\}_{j,n_m,\hdots, n_1 \in \beta \times \mathbb{N}^m} \, \vert \, A \in \Omega, \{f_j\}_{j \in \beta} = \Lambda, f_{j,n_m,\hdots, n_1} \text{ satisfy (*)} \right\},
\]
and (*) denotes \eqref{Lambda_limits} if $m > 1$ and (*) denotes \eqref{Lambda_limits2} if $m = 1$. Moreover, $\tilde \Xi(\tilde A) = \Xi(A)$, and we have  
$\tilde \Lambda = \{\tilde f_{j,n_m,\hdots, n_1}\}_{j,n_m,\hdots, n_1 \in \beta \times \mathbb{N}^m}$, where $\tilde f_{j,n_m,\hdots, n_1}(\tilde A) = f_{j,n_m,\hdots, n_1}(A)$. Note that $\tilde \Xi$ is well-defined by Definition \ref{def:comp_prob} of a computational problem. 
\end{definition}

The SCI and the SCI hierarchy, given $\Delta_m$-information, is then defined in the standard obvious way. We use the notation 
$
\{\Xi,\Omega,\mathcal{M},\Lambda\}^{\Delta_m} \in \Delta_k^{\alpha}
$
to denote that the computational problem is in $\Delta_k^{\alpha}$ given $\Delta_m$-information. When $\mathcal{M}$ and $\Lambda$ are obvious, we write $\{\Xi,\Omega\}^{\Delta_m} \in \Delta_k^{\alpha}$ for short.

\begin{remark}[Classifications in this paper]
\label{rem:needed_rd2}
For the problems considered in this paper, the SCI classifications do not change if we consider arithmetic towers with $\Delta_1$-information. This is easy to see through Church's thesis and an analysis of the stability of our algorithms. For example, when the input is rational we have been careful to restrict all relevant operations to $\mathbb{Q}$ rather than $\mathbb{R}$, and errors incurred from $\Delta_1$-information can be removed in the first limit. Explicitly, for the algorithms based on \texttt{DistSpec} (see Appendix \ref{append_pseudo}), it is possible to carry out an error analysis. We can also bound numerical errors (e.g., using interval arithmetic - see \S \ref{num_test}) and incorporate this uncertainty for the estimation of $\left\|R(z,A)\right\|^{-1}$ to gain the same classification of our problems. Similarly, for other algorithms based on similar functions. In other words, for the results of this paper, it does not matter which model of computation one uses for a definition of `algorithm'. From a classification point of view, they are equivalent for these spectral problems. This leads to rigorous $\Sigma_k^\alpha$ or $\Pi_k^\alpha$ type error control suitable for verifiable numerics. In particular, for $\Sigma_1^{\alpha}$ or $\Pi_1^{\alpha}$ towers of algorithms, this could be useful for computer-assisted proofs.
\end{remark}

\section{Proofs of Theorems on Unbounded Operators on Graphs}
\label{pf_unb_gr}

We now prove the theorems in \S \ref{gen_unbd}, whose proofs will be used in the proofs of the results of \S \ref{dfohjg}. The following argument shows that it is sufficient to consider the $l^2(\mathbb{N})$ case. Given the graph $\mathcal{G}$ and enumeration $v_1,v_2,...$ of the vertices, consider the induced isomorphism $l^2(V(\mathcal{G}))\cong l^2(\mathbb{N})$. This induces a corresponding operator on $l^2(\mathbb{N})$, where the functions $\alpha$ now become matrix values. For the lower bounds, we can consider diagonal operators in $\Omega^{\mathcal{G}}$ (that is, $\alpha(v,w)=0$ if $v\neq w$) with the trivial choice of $S(n)=n$. Hence, lower bounds for $\Omega_D$ translate to lower bounds for $\Omega^{\mathcal{G}}$ and $\Omega_g^{\mathcal{G}}$. For the upper bounds, the construction of algorithms for $l^2(\mathbb{N})$ shows that given the above isomorphism, we can compute a dispersion bounding function $f$ for the induced operator on $l^2(\mathbb{N})$ simply by taking $f(n)=S(n)$. This has $D_{f,n}(A)=0$. Any of the functions in $\Lambda$ for the relevant class of operators on $l^2(\mathbb{N})$ can be computed via the above isomorphism using functions in $\Lambda$ for the relevant class of operators on $l^2(V(\mathcal{G}))$. For instance, to evaluate matrix elements, we use $\alpha(v_i,v_j)$.

There is a useful characterisation of the Attouch--Wets topology. For any closed non-empty sets $C$ and $C_n$, the convergence $d_{\mathrm{AW}}(C_n,C)\rightarrow{0}$ holds if and only if $d_{K}(C_n,C)\rightarrow{0}$ for any compact $K\subset\mathbb{C}$ where
\begin{equation*}
d_{K}(C_1,C_2)=\max\left\{{\underset{a\in{C_1\cap{K}}}{\sup}\mathrm{dist}(a,C_2),\underset{b\in{C_2\cap{K}}}{\sup}\mathrm{dist}(b,C_1)}\right\},
\end{equation*}
with the convention that the supremum over the empty set is $0$. This occurs if and only if for any $\delta>0$ and $K$, there exists $N$ such that if $n>N$ then $C_n\cap K\subset{C+B_{\delta}(0)}$ and $C\cap K\subset{C_n+B_{\delta}(0)}$. Furthermore, it is enough to consider $K$ of the form $B_m(0)$, the closed ball of radius $m$ about the origin, for large $m\in\mathbb{N}$. Throughout this section we take our metric space $(\mathcal{M},d)$ to be $(\mathrm{Cl}(\mathbb{C}),d_{\mathrm{AW}})$.

\begin{remark}[A note on the empty set]
\label{empty_set_subtle}
There is a slight subtlety regarding the empty set. It could be the case that the output of our algorithm is the empty set and hence $\Gamma_n(A)$ does not map to the required metric space. However, the proofs show that for large $n$, $\Gamma_n(A)$ is non-empty, and we gain convergence. By successively computing $\Gamma_n(A)$ and outputting $\Gamma_{m(n)}(A)$, where $m(n)\geq n$ is minimal with $\Gamma_{m(n)}(A)\neq\emptyset$, we see that this does not matter for the classification, but the algorithm in this case is adaptive.
\end{remark}

The following lemma is a useful criterion for determining $\Sigma_1^A$ error control in the Attouch--Wets topology and will be used in the proofs without further comment.

\begin{lemma}
\label{sigma1_unbounded}
Suppose that $\Xi:\Omega\rightarrow(\mathrm{Cl}(\mathbb{C}),d_{\mathrm{AW}})$ is a problem function and $\Gamma_n$ is a sequence of arithmetic algorithms with each output a finite set such that
$$
\lim_{n\rightarrow\infty}d_{\mathrm{AW}}(\Gamma_n(A),\Xi(A))=0,\quad\forall A\in\Omega.
$$
Suppose also that there is a function $E_n$ provided by $\Gamma_n$ (and defined over the output of $\Gamma_n$), such that$$\lim_{n\rightarrow\infty}\sup_{z\in\Gamma_n(A)\cap B_m(0)}E_n(z)=0$$for all $m\in\mathbb{N}$ and such that
$$
\mathrm{dist}(z,\Xi(A))\leq E_n(z),\quad\forall z\in\Gamma_n(A).
$$ 
Then:
\begin{enumerate}
	\item For each $m\in\mathbb{N}$ and given $\Gamma_n(A)$, we can compute in finitely many arithmetic operations and comparisons a sequence of non-negative numbers $a_n^m\rightarrow 0$ (as $n\rightarrow\infty$) such that
\begin{equation*}
\Gamma_n(A)\cap B_{m}(0)\subset\Xi(A)+B_{a_n^m}(0).
\end{equation*}
	\item Given $\Gamma_n(A)$, we can compute in finitely many arithmetic operations and comparisons a sequence of non-negative numbers $b_n\rightarrow 0$ such that
$
\Gamma_n(A)\subset A_n
$ for some $A_n\in\mathrm{Cl}(\mathbb{C})$ with $d_{\mathrm{AW}}(A_n,\Xi(A))\leq b_n$.
\end{enumerate}
Hence we can convert $\Gamma_n$ to a $\Sigma_1^A$ tower using the sequence $\{b_n\}$ by taking subsequences if necessary.
\end{lemma}

\begin{proof}
For the proof of (1), we may take $a_n^m=\sup{\{E_n(z):z\in\Gamma_n(A)\cap B_m(0)\}}$ and the result follows. Note that we need $\Gamma_n(A)$ to be finite to compute this number with finitely many arithmetic operations and comparisons. We next show (2) by defining
$$
A_n^m=\big((\Xi(A)+B_{a^m_n}(0))\cap B_m(0)\big)\cup(\Gamma_n(A)\cap\{z:\left|z\right|\geq m\}).
$$
It is clear that $\Gamma_n(A)\subset A_n^m$ and given $\Gamma_n(A)$ we can easily compute a lower bound $m_1$ such that $\Xi(A)\cap B_{m_1}(0)\neq\emptyset$. Compute this from $\Gamma_1(A)$ and then fix it. Suppose that $m\geq 4m_1$, and suppose that $\left|z\right|<\left\lfloor m/4\right\rfloor$. Then the points in $A_n^m$ and $\Xi(A)$ nearest to $z$ must lie in $B_m(0)$ and hence
$$
\mathrm{dist}(z,A_n^m)\leq \mathrm{dist}(z,\Xi(A)),\quad \mathrm{dist}(z,\Xi(A))\leq \mathrm{dist}(z,A_n^m)+a^m_n.
$$
It follows that $$d_{\mathrm{AW}}(A_n^m,\Xi(A))\leq a^m_n + 2^{-\left\lfloor m/4\right\rfloor}.$$
We now choose a sequence $m(n)$ such that setting $A_n=A_n^{m(n)}$ and $b_n=a^{m(n)}_n + 2^{-\left\lfloor m(n)/4\right\rfloor}$ proves the result. Clearly it is enough to ensure that $b_n$ converges to zero. If $n<4m_1$, set $m(n)=4m_1$, otherwise consider $4m_1\leq k\leq n$. If such a $k$ exists with $a_n^k\leq 2^{-k}$, let $m(n)$ be the maximal such $k$. If no such $k$ exists, set $m(n)=4m_1$. For a fixed $m$, $a^m_n\rightarrow 0$ as $n\rightarrow\infty$. It follows that for large $n$, $a_n^{m(n)}\leq 2^{-m(n)}$ and that $m(n)\rightarrow\infty$.
\end{proof}

\begin{remark}
We will only consider algorithms where the output of $\Gamma_n(A)$ is at most finite for each $n$. Hence the above restriction does not matter in what follows.
\end{remark}

To build our algorithms, we characterise the reciprocal of resolvent norm in terms of the injection modulus. For $A\in\mathcal{C}(l^2(\mathbb{N}))$, we define the injection modulus as
\begin{equation}
\label{inj_mod_def}
\sigma_{\mathrm{inf}}(A)=\inf\{\left\|Ax\right\|:x\in\mathcal{D}(A),\left\|x\right\|=1\},
\end{equation}
and define the function
$$
\gamma(z,A)=\min\{\sigma_{\mathrm{inf}}(A-zI),\sigma_{\mathrm{inf}}(A^*-\bar{z}I)\}.
$$
The following shows that if $z\notin \mathrm{Sp}(A)$, $\gamma(z,A)=\sigma_{\mathrm{inf}}(A-zI)=\sigma_{\mathrm{inf}}(A^*-\bar{z}I)=\left\|R(z,A)\right\|^{-1}$. For $z\in\mathrm{Sp}(A)$, it can occur that $\sigma_{\mathrm{inf}}(A-zI)\neq\sigma_{\mathrm{inf}}(A^*-\bar{z}I)$, but we still must have $\gamma(z,A)=0$.

\begin{lemma}
\label{char_sigma}
For $A\in\mathcal{C}(l^2(\mathbb{N}))$, $\gamma(z,A)=1/\left\|R(z,A)\right\|$, where $R(z,A)$ denotes the resolvent $(A-zI)^{-1}$ and we adopt the convention that $1/\left\|R(z,A)\right\|=0$ if $z\in\mathrm{Sp}(A)$.
\end{lemma}
\begin{proof}
We deal with the case $z\notin\mathrm{Sp}(A)$ first, where we prove that $\sigma_{\mathrm{inf}}(A-zI)=\sigma_{\mathrm{inf}}(A^*-\bar{z}I)=1/\left\|R(z,A)\right\|$. We show this for $\sigma_{\mathrm{inf}}(A-zI)$ and the other case is similar using the fact that $R(z,A)^*=R(\overline{z},A^*)$ and $\left\|R(z,A)\right\|=\left\|R(z,A)^*\right\|$. Let $x\in\mathcal{D}(A)$ with $\left\|x\right\|=1$, then
$$
1=\left\|R(z,A)(A-zI)x\right\|\leq\left\|R(z,A)\right\|\left\|(A-zI)x\right\|.
$$
Hence upon taking the infinum over such $x$, $\sigma_{\mathrm{inf}}(A-zI)\geq 1/\left\|R(z,A)\right\|$. Conversely, let $x_n\in l^2(\mathbb{N})$ such that $\left\|x_n\right\|=1$ and $\left\|R(z,A)x_n\right\|\rightarrow \left\|R(z,A)\right\|$. It follows that
$$
1=\left\|(A-zI)R(z,A)x_n\right\|\geq\sigma_{\mathrm{inf}}(A-zI)\left\|R(z,A)x_n\right\|.
$$
Letting $n\rightarrow\infty$ we get $\sigma_{\mathrm{inf}}(A-zI)\leq 1/\left\|R(z,A)\right\|$.

Now suppose that $z\in\mathrm{Sp}(A)$. If at least one of $A-zI$ or $A^*-\bar{z}I$ is not injective on their respective domain then we are done, so assume both are one to one. Suppose also that $\sigma_{\mathrm{inf}}(A-zI),\sigma_{\mathrm{inf}}(A^*-\bar{z}I)>0$ otherwise we are done. It follows that $\mathcal{R}(A-zI)$ is dense in $l^2(\mathbb{N})$ by injectivity of $A^*-\bar{z}I$ since $\mathcal{R}(A-zI)^{\perp}=N(A^*-\bar{z}I)$. It follows that we can define $(A-zI)^{-1}$, bounded on the dense set $\mathcal{R}(A-zI)$. We can extend this inverse to a bounded operator on the whole of $l^2(\mathbb{N})$. Closedness of $A$ now implies that $(A-zI)(A-zI)^{-1}=I$. Clearly $(A-zI)^{-1}(A-zI)x=x$ for all $x\in\mathcal{D}(A)$. Hence, $(A-zI)^{-1}=R(z,A)\in\mathcal{B}(l^2(\mathbb{N}))$ so that $z\notin\mathrm{Sp}(A)$, a contradiction. 
\end{proof}

Suppose we have a sequence of functions $\gamma_n(z,A)$ that converge uniformly to $\gamma(z,A)$ on compact subsets of $\mathbb{C}$. Define the grid
\begin{equation}
\label{grid_def_un}
\texttt{Grid}(n)=\frac{1}{n}(\mathbb{Z}+i\mathbb{Z})\cap B_n(0).
\end{equation}
For an strictly increasing continuous function $g:\mathbb{R}_{\geq0}\rightarrow\mathbb{R}_{\geq0}$, with $g(0)=0$ and diverging at infinity, for $n\in\mathbb{N}$ and $y\in\mathbb{R}_{\geq0}$ define
\begin{equation}
\label{comp_inv_g_un}
\texttt{CompInvg}(n,y,g)=\min\{k/n:k\in\mathbb{N},g(k/n)>y\}.
\end{equation}
$\texttt{CompInvg}(n,y,g)$ can be computed from finitely many evaluations of the function $g$. We now build the algorithm converging to the spectrum using the functions in (\ref{comp_ball_gs}). 
For each $z\in\texttt{Grid}(n)$, let 
$$\Upsilon_{n,z}=B_{\texttt{CompInvg}(n,\gamma_n(z,A),g_{\left\lceil \left|z\right|\right\rceil})}(z)\cap{}\texttt{Grid}(n).$$
If $\gamma_n(z,A)>\left(\left|z\right|^2+1\right)^{-1}$, set $M_z=\emptyset$, otherwise set
$$
M_z=\{w\in \Upsilon_{n,z}:\gamma_n(w,A)=\min_{v\in \Upsilon_{n,z}}\gamma_n(v,A) \}.
$$
Finally, define $\Gamma_n(A)=\cup_{z\in \texttt{Grid}(n)}M_z$. It is clear that if $\gamma_n(z,A)$ can be computed in finitely many arithmetic operations and comparisons from the relevant functions in $\Lambda$ for each problem, then this procedure defines an arithmetic algorithm $\Gamma_n$. If $A\in\mathcal{C}(l^2(\mathbb{N}))$ with non-empty spectrum, there exists $z\in B_m(0)$ with $\gamma(z,A)\leq(m^2+1)^{-1}/2$ and, for large $n$, $z_n\in\texttt{Grid}(n)$ sufficiently close to $z$ with $\gamma(z_n,A)\leq(\left|z_n\right|^2+1)^{-1}$. Hence, by computing successive $\Gamma_n(A)$, we can assume that $\Gamma_n(A)\neq\emptyset$ without loss of generality (see Remark \ref{empty_set_subtle}).

\begin{proposition}
\label{conv_alg_unbounded}
Suppose $A\in\mathcal{C}(l^2(\mathbb{N}))$ with non-empty spectrum and we have a function $\gamma_n(z,A)$ that converges uniformly to $\gamma(z,A)$ on compact subsets of $\mathbb{C}$. Suppose also that (\ref{comp_ball_gs}) holds, namely
$$
g_m(\mathrm{dist}(z,\mathrm{Sp}(A)))\leq\left\|R(z,A)\right\|^{-1},\quad \forall z\in B_m(0).
$$
Then $\Gamma_n(A)$ converges in the Attouch--Wets topology to $\mathrm{Sp}(A)$ (assuming $\Gamma_n(A)\neq\emptyset$ without loss of generality).
\end{proposition}
\begin{proof}
We use the characterisation of the Attouch--Wets topology. Suppose that $m\in\mathbb{N}$ is large such that $B_m(0)\cap\mathrm{Sp}(A)\neq\emptyset$. We must show that given $\delta>0$, there exists $N$ such that if $n>N$ then $\Gamma_n(A)\cap B_m(0)\subset{\mathrm{Sp}(A)+B_{\delta}(0)}$ and $\mathrm{Sp}(A)\cap B_m(0)\subset{\Gamma_n(A)+B_{\delta}(0)}$. Throughout the rest of the proof we fix such an $m$. Let $\epsilon_n=\left\|\gamma_n(\cdot,A)-\gamma(\cdot,A)\right\|_{\infty,B_{m+1}(0)}$, where the notation means the supremum norm over the set $B_{m+1}(0)$. 

We deal with the second inclusion first. Suppose that $z\in\mathrm{Sp}(A)\cap B_m(0)$, then there exists some $w\in\texttt{Grid}(n)$ such that $\left|w-z\right|\leq 1/n$. It follows that
$$
\gamma_n(w,A)\leq \gamma(w,A)+\epsilon_n\leq \mathrm{dist}(w,\mathrm{Sp}(A))+\epsilon_n\leq \epsilon_n+1/n.
$$
By choosing $n$ large, we can ensure that $\epsilon_n<(2m^2+2)^{-1}$ and that $1/n\leq (2m^2+2)^{-1}$ so that $\gamma_n(w,A)<(\left|w\right|^2+1)^{-1}$. It follows that $M_w$ is non-empty. If $y\in M_w$,
$$
\left|y-z\right|\leq \left|w-z\right|+\left|y-w\right|\leq 1/n + 1/n + g^{-1}_{\left\lceil w\right\rceil}(\gamma_n(w,A)).
$$
But the $g_k$'s are non-increasing in $k$, strictly increasing continuous functions with $g_k(0)=0$. Since $\gamma_n(w,A)\leq \epsilon_n+1/n$, it follows that
\begin{equation}
\label{spec_conv1}
\left|y-z\right|\leq 2/n+g^{-1}_{m+1}(\epsilon_n+1/n). 
\end{equation}
There exists $N_1$ such that if $n\geq N_1$ then (\ref{spec_conv1}) holds and $2/n+g^{-1}_{m+1}(\epsilon_n+1/n)\leq\delta$. This gives the second inclusion $\mathrm{Sp}(A)\cap B_m(0)\subset{\Gamma_n(A)+B_{\delta}(0)}$.

For the first inclusion, suppose for a contradiction that this is false. Then there exists $n_j\rightarrow\infty$, $\delta>0$ and $z_{n_j}\in\Gamma_{n_j}(A)\cap B_m(0)$ such that $\mathrm{dist}(z_{n_j},\mathrm{Sp}(A))\geq \delta$. Then $z_{n_j}\in M_{w_{n_j}}$ for some $w_{n_j}\in \texttt{Grid}(n_j)$. Let 
\[
I(j)=B_{\texttt{CompInvg}(n_j,\gamma_{n_j}(w_{n_j},A),g_{\lceil |w_{n_j}|\rceil})}(w_{n_j})\cap{}\texttt{Grid}(n_j),
\]
the set that we compute minima of $\gamma_{n_j}$ over. Let $y_{n_j}\in\mathrm{Sp}(A)$ be of minimal distance to $w_{n_j}$ (such a $y_{n_j}$ exists since the spectrum restricted to any compact ball is compact). It follows that $\left|y_{n_j}-w_{n_j}\right|\leq g_{\lceil |w_{n_j}|\rceil}^{-1}(\gamma(w_{n_j},A))$. A simple geometrical argument (which also works when we restrict everything to the real line for self-adjoint operators), shows that there must be a $v_{n_j}$ in $I(j)$ so that
$$
\left|v_{n_j}-y_{n_j}\right|\leq \frac{4}{n_j}+g^{-1}_{\left\lceil \left|w_{n_j}\right|\right\rceil}(\gamma(w_{n_j},A))-g^{-1}_{\left\lceil \left|w_{n_j}\right|\right\rceil}(\gamma_{n_j}(w_{n_j},A)).
$$
Since $z_{n_j}$ minimises $\gamma_{n_j}$ over $I(j)$ and $M_{w_{n_j}}$ is non-empty, it follows that
$$
\gamma(z_{n_j},A)\leq\gamma_{n_j}(z_{n_j},A)+\epsilon_{n_j}\leq \min\left\{\frac{1}{\left|w_{n_j}\right|^2+1},\gamma_{n_{j}}(v_{n_j},A)\right\}+\epsilon_{n_j}.
$$
This implies that
\begin{equation}
\label{contradiction_wanted}
\delta\leq\mathrm{dist}(z_{n_j},\mathrm{Sp}(A))\leq g_{m}^{-1}\left(\min\left\{\frac{1}{\left|w_{n_j}\right|^2+1},\gamma_{n_{j}}(v_{n_j},A)\right\}+\epsilon_{n_j}\right),
\end{equation}
where we recall that $g_m^{-1}$ is continuous. It follows that the $w_{n_j}$ must be bounded and hence so are the $v_{n_j}$. Due to the local uniform convergence of $\gamma_n$ to $\gamma$, it follows that
$$
 \frac{4}{n_j}+g^{-1}_{\left\lceil \left|w_{n_j}\right|\right\rceil}(\gamma(w_{n_j},A))-g^{-1}_{\left\lceil \left|w_{n_j}\right|\right\rceil}(\gamma_{n_j}(w_{n_j},A))\rightarrow 0,\quad \text{as }n_j\rightarrow\infty.
$$
But then
$$
\gamma(v_{n_j},A)\leq \mathrm{dist}(v_{n_j},\mathrm{Sp}(A))\leq\left|v_{n_j}-y_{n_j}\right|\rightarrow 0.
$$
The local uniform convergence implies that $\gamma_{n_{j}}(v_{n_j},A)\rightarrow 0$, which contradicts (\ref{contradiction_wanted}) and completes the proof.
\end{proof}

Next, given such a sequence $\gamma_n$, we would like to provide an algorithm for computing the pseudospectrum. However, care must be taken in the unbounded case since the resolvent norm can be constant on open subsets of $\mathbb{C}$ \cite{shargorodsky2008level}. Simply taking
$
\texttt{Grid}(n)\cap\{z:\gamma_n(z,A)\leq \epsilon\}
$
is not guaranteed to converge, as can be seen in the case that $\gamma_n$ is identically $\gamma$ and $A$ is such that the level set $\{\|R(z,A)\|^{-1}=\epsilon\}$ has non-empty interior. To get around this, we need an extra assumption on the functions $\gamma_n$.

\begin{lemma}
\label{pseudospec_unbounded}
Suppose $A\in\mathcal{C}(l^2(\mathbb{N}))$ with non-empty spectrum and let $\epsilon>0$. Suppose we have a sequence of functions $\gamma_n(z,A)$ that converge uniformly to $\left\|R(z,A)\right\|^{-1}$ on compact subsets of $\mathbb{C}$. Set
$$
{\Gamma}_n^{\epsilon}(A)=\texttt{\textup{Grid}}(n)\cap\{z:\gamma_n(z,A)<\epsilon\}.
$$
Then for large $n$, $\Gamma_n^{\epsilon}(A)\neq\emptyset$ (so we can assume this without loss of generality). Suppose also that there exists $N\in\mathbb{N}$ (possibly dependent on $A$ but independent of $z$) such that if $n\geq N$ then $\gamma_n(z,A)\geq \left\|R(z,A)\right\|^{-1}$. Then $d_{\mathrm{AW}}({\Gamma}_n^{\epsilon}(A),\mathrm{Sp}_{\epsilon}(A))\rightarrow 0$ as $n\rightarrow\infty$.
\end{lemma}

\begin{proof}
Since the pseudospectrum is non-empty, $\Gamma_n^{\epsilon}(A)\neq\emptyset$ for large $n$. It follows from our usual argument of computing successive $\Gamma_n^{\epsilon}(A)$ (see Remark \ref{empty_set_subtle}) that we may assume $\Gamma_n^{\epsilon}(A)\neq\emptyset$ for all $n$ without loss of generality. We use the characterisation of the Attouch--Wets topology. Suppose that $m$ is large such that $B_m(0)\cap\mathrm{Sp}_{\epsilon}(A)\neq\emptyset$. $\exists N\in\mathbb{N}$ such that if $n\geq N$, $\gamma_n(z,A)\geq \left\|R(z,A)\right\|^{-1}$ and hence $\Gamma_n^{\epsilon}(A)\cap B_m(0)\subset\mathrm{Sp}_{\epsilon}(A)$. Hence we must show that given $\delta>0$, there exists $N_1$ such that if $n>N_1$ then $\mathrm{Sp}_{\epsilon}(A)\cap B_m(0)\subset{\Gamma_n^{\epsilon}(A)+B_{\delta}(0)}$. Suppose for a contradiction that this were false. Then there exists $z_{n_j}\in\mathrm{Sp}_{\epsilon}(A)\cap B_m(0)$, $\delta>0$ and $n_j\rightarrow\infty$ such that $\mathrm{dist}(z_{n_j},\Gamma_{n_j}^{\epsilon}(A))\geq \delta$. Without loss of generality, we can assume that $z_{n_j}\rightarrow z\in\mathrm{Sp}_{\epsilon}(A)\cap B_m(0)$. There exists some $w$ with $\left\|R(w,A)\right\|^{-1}<\epsilon$ and $\left|z-w\right|\leq \delta/2$. Assuming $n_j>m+\delta$, there exists $y_{n_j}\in \texttt{Grid}(n_j)$ with $\left|y_{n_j}-w\right|\leq 1/{n_j}$. It follows that
$$
\gamma_{n_j}(y_{n_j},A)\leq \left|\gamma_{n_j}(y_{n_j},A)-\gamma(y_{n_j},A)\right|+\left|\gamma(w,A)-\gamma(y_{n_j},A)\right|+\left\|R(w,A)\right\|^{-1}.
$$
But $\gamma$ is continuous and $\gamma_{n_j}$ converges uniformly to $\gamma$ on compact subsets. Hence for large $n_j$, it follows that $\gamma_{n_j}(y_{n_j},A)<\epsilon$ so that $y_{n_j}\in\Gamma_{n_j}^{\epsilon}(A)$. But
$
\left|y_{n_j}-z\right|\leq \left|z-w\right|+\left|y_{n_j}-w\right|\leq \delta/2 + 1/{n_j},
$
which is smaller than $\delta$ for large $n_j$. This inequality gives the required contradiction.
\end{proof}

Now suppose that $A\in\hat\Omega$ (recall that $\hat\Omega$ is the class of all $A\in\mathcal{C}(l^2(\mathbb{N}))$ with non-empty spectrum such that (1) and (2) from \S \ref{gen_unbd} hold) and let $D_{f,n}(A)\leq c_n$. The following shows that we can construct the required sequence $\gamma_n(z,A)$. Each function output requires only finitely many arithmetic operations and comparisons of the corresponding input information.

\begin{theorem}
\label{unif_conv_gamma}
Let $A\in\hat\Omega$ and define the function
$$
\tilde{\gamma}_n(z,A)=\min\{\sigma_{\mathrm{inf}}(P_{f(n)}(A-zI)|_{P_n(l^2(\mathbb{N}))}),\sigma_{\mathrm{inf}}(P_{f(n)}(A^*-\bar{z}I)|_{P_n(l^2(\mathbb{N}))})\}.
$$
We can compute $\tilde{\gamma}_n$ up to precision $1/n$ using finitely many arithmetic operations and comparisons. We call this approximation $\hat{\gamma}_n$ and set
$$
\gamma_n(z,A)=\hat{\gamma}_n(z,A)+c_n+1/n.
$$
Then $\gamma_n(z,A)$ converges uniformly to $\gamma(z,A)$ on compact subsets of $\mathbb{C}$ and $\gamma_n(z,A)\geq \gamma(z,A)$. 
\end{theorem}
\begin{proof}
We will first prove that $
\sigma_{\mathrm{inf}}((A-zI)|_{P_n(l^2(\mathbb{N}))})\downarrow\sigma_{\mathrm{inf}}(A-zI)
$
as $n\rightarrow\infty$. It is trivial that $\sigma_{\mathrm{inf}}((A-zI)|_{P_n(l^2(\mathbb{N}))})\geq\sigma_{\mathrm{inf}}(A-zI)$ and that $\sigma_{\mathrm{inf}}((A-zI)|_{P_n(l^2(\mathbb{N}))})$ is non-increasing in $n$. Using Lemma \ref{char_sigma}, let $\epsilon>0$ and $x\in\mathcal{D}(A)$ such that $\left\|x\right\|=1$ and $\left\|(A-zI)x\right\|\leq \sigma_{\mathrm{inf}}(A-zI)+\epsilon$. Since $\mathrm{span}\{e_n:n\in\mathbb{N}\}$ forms a core of $A$, $AP_{n_j}x_{n_j}\rightarrow Ax$ and $P_{n_j}x_{n_j}\rightarrow x$ for some $n_j\rightarrow\infty$ and some sequence of vectors $x_{n_j}$  of norm $1$. It follows that for large $n_j$
$$
\sigma_{\mathrm{inf}}((A-zI)|_{P_{n_j}(l^2(\mathbb{N}))})\leq \frac{\left\|(A-zI)P_{n_j}x_{n_j}\right\|}{\left\|P_{n_j}x_{n_j}\right\|}\rightarrow\left\|(A-zI)x\right\|\leq\sigma_{\mathrm{inf}}(A-zI)+\epsilon.
$$
Since $\epsilon>0$ was arbitrary, this shows the convergence of $\sigma_{\mathrm{inf}}((A-zI)|_{P_n(l^2(\mathbb{N}))})$. The fact that $\mathrm{span}\{e_n:n\in\mathbb{N}\}$ forms a core of $A^*$ can also be used to show that $\sigma_{\mathrm{inf}}((A-zI)^*|_{P_n(l^2(\mathbb{N}))})\downarrow\sigma_{\mathrm{inf}}(A^*-\overline{z}I)$.

Next we use the assumption of bounded dispersion from assumption (2) of \S \ref{gen_unbd} for $A\in\hat\Omega$. For any bounded operators $B,C$, it holds that
$
\left|\sigma_{\mathrm{inf}}(B)-\sigma_{\mathrm{inf}}(C)\right|\leq \left\|B-C\right\|.
$
The definition of bounded dispersion now implies that
$$
\left|\tilde{\gamma}_n(z,A)-\min\{\sigma_{\mathrm{inf}}((A-zI)|_{P_n(l^2(\mathbb{N}))}),\sigma_{\mathrm{inf}}((A-zI)^*|_{P_n(l^2(\mathbb{N}))})\}\right|\leq c_n.
$$
The monotone convergence of $\min\{\sigma_{\mathrm{inf}}((A-zI)|_{P_n(l^2(\mathbb{N}))}),\sigma_{\mathrm{inf}}((A-zI)^*|_{P_n(l^2(\mathbb{N}))})\}$, together with Dini's theorem, imply that $\tilde{\gamma}_n(z,A)$ converges uniformly to the continuous function $\gamma(z,A)$ on compact subsets of $\mathbb{C}$ with $\tilde{\gamma}_n(z,A)+c_n\geq \gamma(z,A)$.

The proof will be complete if we can show that we can compute $\tilde{\gamma}_n(z,A)$ to precision $1/n$ using finitely many arithmetic operations and comparisons. To do this, consider the matrices
$$
B_n(z)=P_n(A-zI)^*P_{f(n)}(A-zI)P_n,\quad C_n(z)=P_n(A-zI)P_{f(n)}(A-zI)^*P_n.
$$
By an interval search routine and Lemma \ref{num_neg_evals} below, we can determine the smallest $l\in\mathbb{N}$ such that at least one of $B_n(z)-(l/n)^{2}I$ or $C_n(z)-(l/n)^{2}I$ has a negative eigenvalue. We then output $l/n$ to get the $1/n$ bound.
\end{proof}

Every finite Hermitian matrix $B$ (not necessarily positive semidefinite) has a decomposition
\begin{equation*}
PBP^T=LDL^*,
\end{equation*}
where $L$ is lower triangular with $1$'s along its diagonal, $D$ is block diagonal with block sizes $1$ or $2$ and $P$ is a permutation matrix \cite{bunch1971direct,bunch1977some} \cite[\S 4.4]{golub1996matrix}. The permutation matrix $P$ arises from pivoting strategies for stability. The above decomposition can be computed with finitely many arithmetic operations and comparisons.

\begin{lemma}
\label{num_neg_evals}
Let $B\in\mathbb{C}^n$ be self-adjoint (Hermitian), then we can determine the number of negative eigenvalues of $B$ in finitely many arithmetic operations and comparisons (assuming no round-off errors) on the matrix entries of $B$.
\end{lemma}
\begin{proof}
The matrix $\tilde B=PBP^T$ has the same eigenvalues as $B$. Hence, without loss of generality, we can consider $\tilde B$. We can compute the decomposition $\tilde B=LDL^*$ in finitely many arithmetical operations and comparisons. By Sylvester's law of inertia, $D$ has the same number of negative eigenvalues as $\tilde B$. It is then clear that we only need to deal with $2\times2$ matrices corresponding to the maximum block size of $D$. Let $\lambda_1,\lambda_2$ be the two eigenvalues of such a matrix. Then we can determine their sign pattern from the trace and determinant of the matrix. 
\end{proof}

This lemma has a corollary that is used in \S \ref{Sec:proof_spec_gap}.

\begin{corollary}
\label{eigs_fin_mat}
Let $B\in\mathbb{C}^n$ be self-adjoint (Hermitian) and list its eigenvalues in increasing order, including multiplicity, as $\lambda_1\leq\lambda_2\leq...\leq \lambda_n$. In exact arithmetic, given $\epsilon>0$, we can compute $\lambda_1,\lambda_2,...\lambda_n$ to precision $\epsilon$ using only finitely many arithmetic operations and comparisons.
\end{corollary}

\begin{proof}
Consider $A(\lambda)=B-\lambda I$. We will apply Lemma \ref{num_neg_evals} to $A(\lambda)$ for various $\lambda$. First, by considering the sequences $-1,-2,...$ and $1,2,...$ we can find $m_1\in\mathbb{N}$ such that $\mathrm{Sp}(B)\subset(-m_1,m_1)$. Now let $m_2\in\mathbb{N}$ such that $1/m_2<\epsilon$ and let $a_j$ be the output of Lemma \ref{num_neg_evals} applied to $A(j/m_2)$ for $-m_1m_2\leq j\leq m_1m_2$. Set
$$
\tilde{\lambda}_k=\min\{j:-m_1m_2\leq j\leq m_1m_2,a_j\geq k\}, \quad k=1,...,n.
$$
If $\lambda_k\in[j/m_2,(j+1)/m_2)$ then $\tilde{\lambda}_k=(j+1)/m_2$ and hence $\left|\tilde{\lambda}_k-\lambda_k\right|\leq 1/m_2<\epsilon$.
\end{proof}

\begin{remark}
Of course, in practice, there are much more computationally efficient ways to numerically compute eigenvalues or singular values. The above is purely used to show this can be done to any precision with finitely many arithmetic operations. Computing the eigenvalues and eigenvectors of finite-dimensional matrices dates back to Wilkinson \cite{MR0184422}, with guaranteed convergence for self-adjoint matrices via Wilkinson shifts, see \cite{parlett1998symmetric}. It is not completely straightforward to deduce Corollary \ref{eigs_fin_mat} via the QR algorithm with Wilkinson shifts, as one has to deal with halting criteria to achieve the correct precision. Moreover, one must approximate roots to extract the approximate eigenvalues from a potential $2 \times 2$ matrix block. One could make the cost of the method in Corollary \ref{eigs_fin_mat} logarithmic in the desired accuracy by using interval bisections. This is beyond the scope of the present paper (and not its purpose). For a polynomial-time (but impractical) algorithm for eigenvalues and eigenvectors based on Newton's method, see \cite{armentano2018stable}.
\end{remark}

By taking successive minima,
$
\upsilon_n(z,A)=\min_{1\leq j \leq n} \gamma_n(z,A),
$
we can obtain a sequence of functions $\upsilon_n$ that converge uniformly on compact subsets of $\mathbb{C}$ to $\gamma(z,A)$ monotonically from above. Hence without loss of generality, we will always assume that $\gamma_n$ have this property. We can now prove our main result. 

\begin{proof}[Proof of Theorem \ref{unbounded_theorem}]
By considering bounded diagonal operators, it is straightforward to see that none of the problems lie in $\Delta_1^G$. We first deal with the convergence of height one arithmetical towers. For the spectrum, we use the function $\gamma_n$ described in Theorem \ref{unif_conv_gamma} together with Proposition \ref{conv_alg_unbounded} and its described algorithm. For the pseudospectrum, we use the same function $\gamma_n$ described in Theorem \ref{unif_conv_gamma} and convergence follows from using the algorithm in Proposition \ref{pseudospec_unbounded}.

We are left with proving that our algorithms have $\Sigma^A_1$ error control. For any $A\in\hat\Omega$, the output of the algorithm in Proposition \ref{pseudospec_unbounded} is contained in the true pseudospectrum since $\gamma_n(z,A)\geq\gamma(z,A)=\left\|R(z,A)\right\|^{-1}$. Hence we need only show that the algorithm in Proposition \ref{conv_alg_unbounded} provides $\Sigma_1^A$ error control for input $A\in\Omega_g$. Denote the algorithm by $\Gamma_n$ and set 
$$
E_n(z)=\texttt{CompInvg}(n,\gamma_n(z,A),g_{\left\lceil \left|z\right|\right\rceil}^{-1})
$$
on $\Gamma_n(A)$ and zero on $\mathbb{C}\backslash\Gamma_n(A)$. Since $\gamma_n(z,A)\geq\left\|R(z,A)\right\|^{-1}$, the assumptions on $\{g_m\}$ imply that
$$
\mathrm{dist}(z,\mathrm{Sp}(A))\leq E_n(z),\quad \forall z\in \Gamma_n(A).
$$
Suppose for a contradiction that $E_n$ does not converge uniformly to zero on compact subsets of $\mathbb{C}$. Then there exists some compact set $K$, some $\epsilon>0$, a sequence $n_j\rightarrow\infty$ and $z_{n_j}\in K$ such that $E_{n_j}(z_{n_j})\geq \epsilon$. It follows that $z_{n_j}\in\Gamma_{n_j}(A)$. Without loss of generality, $z_{n_j}\rightarrow z$. By convergence of $\Gamma_{n_j}(A)$, $z\in\mathrm{Sp}(A)$ and hence $\gamma_{n_j}(z_{n_j},A)\rightarrow\gamma(z,A)=0$. Now choose $M$ large such that $K\subset B_{M}(0)$. But then
$$
E_{n_j}(z_{n_j})\leq g_{M}^{-1}(\gamma_{n_j}(z_{n_j},A))+\frac{1}{n_j}\rightarrow 0,
$$
the required contradiction.
\end{proof}

\begin{remark}
\label{E_conv2}
The above shows that $E_n(z)$ converges uniformly to the function $g_{\left\lceil \left|z\right|\right\rceil}^{-1}(\gamma(z,A))$ as $n\rightarrow\infty$ on compact subsets of $\mathbb{C}$.
\end{remark}

Finally, we consider the decision problems $\Xi_3$ and $\Xi_4$.

\begin{proof}[Proof of Theorem \ref{unbounded_test_theorem}]
It is clearly enough to prove the lower bounds for $\Omega_D\times\mathcal{K}(\mathbb{C})$ and the existence of towers for $\hat\Omega\times\mathcal{K}(\mathbb{C})$. The proof of lower bounds for $\Omega_D\times\mathcal{K}(\mathbb{C})$ can also be trivially adapted to the more restrictive versions of the problem than described in the theorem.

\textbf{Step 1}: $\{\Xi_3,\Omega_D\times\mathcal{K}(\mathbb{C})\}\not\in\Delta_2^G$. Suppose this were false, and $\Gamma_n$ is a height one tower solving the problem. For every $A$ and $n$ there exists a finite number $N(A,n)\in\mathbb{N}$ such that the evaluations from $\Lambda_{\Gamma_n}(A)$ only take the matrix entries $A_{ij} = \left\langle Ae_j , e_i\right\rangle$ with $i,j\leq N(A,n)$ into account. Without loss of generality (by shifting and rotating our argument), we assume that $K\cap[0,1]=\{0\}$. We will consider the operators $A_m=\mathrm{diag}\{1,1/2,...,1/m\}\in\mathbb{C}^{m\times m}$, $B_m=\mathrm{diag}\{1,1,...,1\}\in\mathbb{C}^{m\times m}$ and $C=\mathrm{diag}\{1,1,...\}$. Set $A=\bigoplus_{m=1}^{\infty}(B_{k_m}\oplus A_{k_m})$ where we choose an increasing sequence $k_m$ inductively as follows.

Set $k_1=1$ and suppose that $k_1,...,k_m$ have been chosen. $\mathrm{Sp}(B_{k_1}\oplus A_{k_1}\oplus...\oplus B_{k_m}\oplus A_{k_m}\oplus C)=\{1,1/2,...,1/m\}$ and hence
$$\Xi_3(B_{k_1}\oplus A_{k_1}\oplus...\oplus B_{k_m}\oplus A_{k_m}\oplus C)=\text{``No''}.$$
So there exists some $n_m\geq m$ such that if $n\geq n_m$ then
\begin{equation*}
\Gamma_n(B_{k_1}\oplus A_{k_1}\oplus...\oplus B_{k_m}\oplus A_{k_m}\oplus C)=\text{``No''}.
\end{equation*}
Now let $k_{m+1}\geq \max\{N(B_{k_1}\oplus A_{k_1}\oplus...\oplus B_{k_m}\oplus A_{k_m}\oplus C,n_m),k_m+1\}$. By assumption (iii) in Definition, \ref{Gen_alg} it follows that $\Lambda_{\Gamma_{n_m}}(B_{k_1}\oplus A_{k_1}\oplus...\oplus B_{k_m}\oplus A_{k_m}\oplus C)=\Lambda_{\Gamma_{n_m}}(A)$ and hence by assumption (ii) in the same definition that $\Gamma_{n_m}(A)=\Gamma_{n_m}(B_{k_1}\oplus A_{k_1}\oplus...\oplus B_{k_m}\oplus A_{k_m}\oplus C)=\text{``No''}$. But $0\in\mathrm{Sp}(A)$ and so must have $\lim_{n\rightarrow\infty}(\Gamma_n(A))=\text{``Yes''}$, a contradiction.

\textbf{Step 2}: $\{\Xi_4,\Omega_D\}\not\in\Delta_2^G$. The same proof as step 1, but replacing $A$ by $A+\epsilon I$ works in this case.

\textbf{Step 3}: $\{\Xi_3,\hat\Omega\times\mathcal{K}(\mathbb{C})\}\in\Pi^A_2$. Recall that we can compute, with finitely many arithmetic operations and comparisons, a function $\gamma_n$ that converges monotonically down to $\left\|R(z,A)\right\|^{-1}$ uniformly on compact subsets of $\mathbb{C}$. Set
$$
\Gamma_{n_2,n_1}(A)=\text{``Does there exist some $z\in K_{n_2}$ such that $\gamma_{n_1}(z,A)<1/2^{n_2}$?''}.
$$
This is an arithmetic algorithm since each $K_n$ is finite. Moreover,
$$
\lim_{n_1\rightarrow\infty}\Gamma_{n_2,n_1}(A)=\text{``Does there exist some $z\in K_{n_2}$ such that $\left\|R(z,A)\right\|^{-1}<1/2^{n_2}$?''}=:\Gamma_{n_2}(A).
$$
If $K\cap\mathrm{Sp}(A)=\emptyset$, $\left\|R(z,A)\right\|^{-1}$ is bounded below on the compact set $K$ and hence for large $n_2$, $\Gamma_{n_2}(A)=\text{``No''}$. However, if $z\in\mathrm{Sp}(A)\cap K$, let $z_{n_2}\in K_{n_2}$ minimise the distance to $z$. Then
$$\left\|R(z_{n_2},A)\right\|^{-1}\leq\mathrm{dist}(z_{n_2},\mathrm{Sp}(A))<1/2^{n_2}$$
and hence $\Gamma_{n_2}(A)=\text{``Yes''}$ for all $n_2$. This also shows the $\Pi^A_2$ classification.

\textbf{Step 4}: $\{\Xi_4,\hat\Omega\times\mathcal{K}(\mathbb{C})\}\in\Pi^A_2$. Set
$$
\Gamma_{n_2,n_1}(A)=\text{``Does there exist some $z\in K_{n_2}$ such that $\gamma_{n_1}(z,A)<1/2^{n_2}+\epsilon$?''},
$$
then the same argument used in step 3 works in this case.
\end{proof}

\subsection{Examples of $f$ used in the computational examples}

We end with some examples for the graph case $l^2(V(\mathcal{G}))$. Recall that we consider operators in (\ref{kthNeigh}), where for any $v\in V$ the set of $w\in V$ with $\alpha(v,w)\neq 0$ is finite. As noted at the start of \S \ref{pf_unb_gr}, we can equate our operators with operators on $l^2(\mathbb{N})$ with bounded dispersion as in (\ref{bd_disp}) with $D_{f,n}(A)=0$.

Suppose our enumeration of the vertices obeys the following pattern. The neighbours of $v_1$ (including itself) are $S_1=\{v_1,v_2,...,v_{q_1}\}$ for some finite $q_1$. The set of neighbours of these vertices is $S_2=\{v_1,...,v_{q_2}\}$ for some finite $q_2$, where we continue the enumeration of $S_1$. This process is continued to inductively enumerate each $S_m$. 
\begin{example}\label{findf2}
Suppose that the bounded operator $A$ can be written as
\begin{equation}
A=\sum_{v \sim_k w}\alpha(v,w)\left|v\right\rangle\left\langle w\right|
\label{kthNeigh2}
\end{equation}
for some $k\in\mathbb{N}$. Here we write $v \sim_k w$ for two vertices $v,w\in V$ if there is a path of at most $k$ edges connecting $v$ and $w$. Hence $A$ only involves $k$th nearest neighbour interactions. This is a graph operator version of $A$ being banded (though of course the representation matrix acting on $l^2(\mathbb{N})$ need not be banded in the usual sense). Suppose also that the vertex degree of $\mathcal{G}$ is bounded by $M$. It holds that $v_n\in S_n$ and $\{w\in V: v \sim_k w\}\subset S_{n+k}$. Inductively $\left|S_m\right|\leq(M+1)^m$ and hence we may take the upper bound
$$
S(n)=(M+1)^{n+k}.
$$
\end{example}

\begin{example}\label{findf3}
Consider a nearest neighbour operator ($k=1$ in (\ref{kthNeigh2})) on $l^2(\mathbb{Z}^d)$. It holds that $\left|S_{m}\right|\sim\mathcal{O}(m^{d})$ whilst $\left|S_{m+1}-S_m\right|\sim\mathcal{O}(m^{d-1})$ (by considering radial spheres). It is easy to see that we can choose a suitable $S$ such that
\begin{equation*}
S(n)-n\sim \mathcal{O}(n^{\frac{d-1}{d}}),
\end{equation*}
that is, $S$ grows at most linearly.
\end{example}

\section{Proofs of Theorems on Differential Operators on Unbounded Domains}
\label{proof_for_PDES}

Here we prove Theorems \ref{PDE1} and \ref{PDE2}. The constructed algorithms involve technical error estimates with parameters depending on these estimates. In the construction of the algorithms, our strategy will be to reduce the problem to one handled by the proofs in \S \ref{pf_unb_gr}. To do so, we must first select a suitable basis and then compute matrix values. Recall that we aim to compute the spectrum and pseudospectrum from the information given to us regarding the coefficient functions $a_k$ and $\tilde{a}_k$, with the information we can evaluate made precise by the mappings $\Xi_j^1$, $\Xi_j^2$, $\Xi_j^3$ and $\Xi_j^4$. We start by constructing the algorithms used for the positive results in Theorems \ref{PDE1} and \ref{PDE2}, and then prove the lower bounds.

\subsection{Construction of algorithms}
We begin with the description for $d=1$ and then comment on how this can easily be extended to arbitrary dimensions. As an orthonormal basis of $L^2(\mathbb{R})$ we choose the Hermite functions
$$
\psi_m(x)=(2^{m}m!\sqrt{\pi})^{-1/2}e^{-x^2/2}H_{m}(x),\quad m\in\mathbb{Z}_{\geq 0},
$$
where $H_m$ denotes the $m$th (physicists') Hermite polynomial defined by
$$
H_m(x)=(-1)^m\exp(x^2)\frac{d^m}{dx^m}\exp(-x^2).
$$
The Hermite functions obey the recurrence relations
\begin{align*}
\psi_m'(x)&=\sqrt{\frac{m}{2}}\psi_{m-1}(x)-\sqrt{\frac{m+1}{2}}\psi_{m+1}(x)\\
x\psi_m(x)&=\sqrt{\frac{m}{2}}\psi_{m-1}(x)+\sqrt{\frac{m+1}{2}}\psi_{m+1}(x).
\end{align*}
We let $C_H(\mathbb{R})=\mathrm{span}\{\psi_m:m\in\mathbb{Z}_{\geq 0}\}$. Note that since the Hermite functions decay like $e^{-x^2/2}$ (up to polynomials) and the functions $a_k$ and $\tilde{a}_k$ can only grow polynomially, the formal differential operator $T$ and its formal adjoint $T^*$ make sense as operators from $C_H(\mathbb{R})$ to $L^2(\mathbb{R})$. The next proposition says that we can use the chosen basis.

\begin{proposition}
\label{basis_core}
Consider an operator $T\in\Omega$. Then $C_H(\mathbb{R})$ forms a core of both $T$ and $T^*$.
\end{proposition}

\begin{proof}We argue for $T$, and the case of $T^*$ is analogous. Let $f\in C_H(\mathbb{R})$ and choose $\phi\in C_0^\infty(\mathbb{R})$ (the space of compactly supported smooth functions) bounded by $1$ such that $\phi(x)=1$ for all $\left|x\right|\leq1$. It is straightforward, using the fact that the $a_k$'s are polynomially bounded, to show that
$$
\lim_{n\rightarrow\infty}\phi(x/n)f(x)=f(x),\quad\lim_{n\rightarrow\infty}T\phi(x/n)f(x)=(Tf)(x)
$$
in $L^2(\mathbb{R})$, where $Tf$ is the formal differential operator applied to $f$. The fact that $T$ is closed implies that $f\in\mathcal{D}(T)$ and that the formula in (\ref{eq:diff_op}) holds for $u=f$.

Let $g\in C_0^\infty(\mathbb{R})$ and in the $L^2$ sense write
$$
g=\sum_{m\geq 0}b_m\psi_m.
$$
Define $g_n=\sum_{m=0}^nb_m\psi_m$. We show that $Tg_n$ converges as $n\rightarrow\infty$. Since $T$ is closed, $C_0^\infty(\mathbb{R})$ is a core for $T$ and $C_H(\mathbb{R})\ni g_n\rightarrow g$, proving this will prove the desired result.

Let $H$ denote the closure of the operator $-{d^2}/{dx^2} +x^2$ with initial domain $C_0^\infty(\mathbb{R})$. Then $H\psi_m=(2m+1)\psi_m$ and $H$ is self-adjoint. Note also that $g\in\mathcal{D}(H^n)$ for any $n\in\mathbb{N}$. But
$
\langle Hg,\psi_m\rangle=(2m+1)\langle g,\psi_m\rangle=(2m+1)b_m,
$
so $\{(2m+1)\left|b_m\right|\}$ is square summable. We can repeat this argument any number of times to see that the coefficients $b_m$ decay faster than any inverse polynomial. To prove the required convergence, it is enough to consider one of the terms $a_k(x)\partial^k$ that defines ${T}$ acting on $C_H(\mathbb{R})$. The coefficient $a_k(x)$ is polynomially bounded almost everywhere, and for some $A_k$ and $B_k$
$$
\langle a_k\partial^k \psi_m, a_k\partial^k \psi_m\rangle \leq A_k^2 \int_{\mathbb{R}}(1+\left|x\right|^{2B_k})^2\partial^k \psi_m(x)\partial^k \psi_m(x) dx.
$$
We can use the recurrence relations for the derivatives of the Hermite functions and orthogonality to bound the right-hand side by a polynomial in $m$. The convergence now follows since ${T}g_n$ is a Cauchy sequence due to the rapid decay of the $\{b_m\}$.
\end{proof}

Clearly, all of the above analysis holds in higher dimensions by considering tensor products
$$
C_H(\mathbb{R}^d):=\mathrm{span}\{\psi_{m_1}\otimes ...\otimes \psi_{m_d} \, \vert \, m_1,...,m_d \in\mathbb{Z}_{\geq0}\}
$$
of Hermite functions. We abuse notation and write $\psi_m = \psi_{m_1}\otimes ...\otimes \psi_{m_d} $. It will be clear from the context when we are dealing with the multidimensional case. To build the required algorithms with $\Sigma_1^A$ error control, we need to select an enumeration of $\mathbb{Z}_{\geq0}^d$ in order to represent $T$ as an operator acting on $l^2(\mathbb{N})$. A simple way to do this is to consider successive half spheres $S_n=\{m\in\mathbb{Z}_{\geq0}^d:\left|m\right|\leq n\}$. We list $S_1$ as $\{e_1,...,e_{r_1}\}$ and given an enumeration $\{e_1,...,e_{r_n}\}$ of $S_n$, we list $S_{n+1}\backslash S_n$ as $\{e_{r_n+1},...,e_{r_{n+1}}\}$. We then list our basis functions as $e_1,e_2,...$ with $\psi_{m}=e_{h(m)}$. In practice, it is often more efficient (especially for large $d$) to consider other orderings such as the hyperbolic cross \cite{lubich_qm_book}. Now that we have a suitable basis, the next question to ask is how to recover the matrix elements of $T$. In \S \ref{pf_unb_gr} the key construction is a function, that can be computed from the information given to us, $\gamma_n(z,T)$, that also converges uniformly from above to $\left\|R(z,T)\right\|^{-1}$ on compact subsets of $\mathbb{C}$. Such a sequence of functions is given by
$$
\Psi_n(z,T):=\min\{\sigma_{\mathrm{inf}}((T-zI)|_{P_n(l^2(\mathbb{N}))}),\sigma_{\mathrm{inf}}((T^*-\bar{z}I)|_{P_n(l^2(\mathbb{N}))})\}
$$
as long as the linear span of the basis forms a core of $T$ and $T^*$. In \S \ref{pf_unb_gr} we used the notion of bounded dispersion to approximate this function. Here we have no such notion, but we can use the information given to us to replace this. It turns out that to approximate $\gamma_n(z,T)$, it suffices to use the following.

\begin{lemma}
\label{approx_mat_vals}
Let $\epsilon>0$ and $n\in\mathbb{N}$. Suppose that we can compute, with finitely many arithmetic operations and comparisons, the matrices
\begin{align*}
\{W_n(z)\}_{ij}&=\langle (T-zI) e_j,(T-zI) e_i\rangle+E^{n,1}_{ij}(z),\\
\{V_n(z)\}_{ij}&=\langle (T-zI)^* e_j,(T-zI)^* e_i\rangle+E^{n,2}_{ij}(z),
\end{align*}
for $1\leq i,j\leq n$, where the entrywise errors $E^{n,1}_{i,j}$ and $E^{n,2}_{i,j}$ have magnitude at most $\epsilon$. Then
$$
\left|\Psi_n(z,T)^2-\min\{{\sigma_{\mathrm{inf}}(W_n)},{\sigma_{\mathrm{inf}}(V_n)}\}\right|\leq n\epsilon.
$$
It follows that if $\epsilon$ is known, we can compute $\Psi_n(z,T)^2$ to within $2n\epsilon$. If $\epsilon$ is unknown, then for any $\delta>0$, we can compute $\Psi_n(z,T)^2$ to within $n\epsilon+\delta$. (In each case with finitely many arithmetic operations and comparisons.)
\end{lemma}
\begin{proof}
Given $\{W_n(z)\}_{ij}$, $(\{W_n(z)\}_{ij}+\overline{\{W_n(z)\}_{ji}})/2$ still has an entrywise absolute error bounded by $\epsilon$. Hence without loss of generality we can assume that the approximations $W_n(z)$ and $V_n(z)$ are self-adjoint. Let $\tilde{W}_n(z)$ and $\tilde{V}_n(z)$ denote the corresponding matrices with no errors. Then
$$
\min\{\sigma_{\mathrm{inf}}((T-zI)|_{P_n(l^2(\mathbb{N}))}),\sigma_{\mathrm{inf}}((T^*-\bar{z}I)|_{P_n(l^2(\mathbb{N}))})\}^2=\min\{{\sigma_{\mathrm{inf}}(\tilde{W}_n)},{\sigma_{\mathrm{inf}}(\tilde{V}_n)}\},
$$
and
\begin{equation}
\label{Frob}
\left|\min\{\sigma_{\mathrm{inf}}(\tilde{W}_n),\sigma_{\mathrm{inf}}(\tilde{V}_n)\}-\min\{\sigma_{\mathrm{inf}}(W_n),\sigma_{\mathrm{inf}}(V_n)\}\right|\leq\max\left\{\left\|W_n-\tilde{W}_n\right\|,\left\|V_n-\tilde{V}_n\right\|\right\}.
\end{equation}
For a finite matrix $M$, we can bound $\left\|M\right\|$ by its Frobenius norm $\sqrt{\sum\left|M_{ij}\right|^2}$. Hence the right hand side of (\ref{Frob}) is at most $n\epsilon$. Given a self-adjoint positive semi-definite matrix $M$, we can compute $\sigma_{\mathrm{inf}}(M)$ to arbitrary precision using finitely many arithmetic operations and comparisons via the argument in the proof of Theorem \ref{unif_conv_gamma}. The lemma now follows.
\end{proof}

Finally, we need some results concerning quasi-Monte Carlo numerical integration, which we use to build the algorithm. Note that with either no prior information concerning the coefficients or for large $d$, this is the type of approach one would use in practice. We start with some definitions and theorems, which we include here for completeness. An excellent reference for these results is \cite{niederreiter1992random}.

\begin{definition} Let $\{t_1,...,t_j\}$ be a sequence in $[0,1]^d$ and let $\mathcal{K}$ denote all subsets of $[0,1]^d$ of the form $\prod_{k=1}^d[0,y_k)$ for $y_k\in(0,1]$. Then we define the star discrepancy of $\{t_1,...,t_j\}$ to be
$$
D_j^*(\{t_1,...,t_j\})=\sup_{K\in\mathcal{K}}\left|\frac{1}{j}\sum_{k=1}^j\chi_K(t_j)-\left|K\right|\right|,
$$
where $\chi_K$ denotes the characteristic function of $K$.
\end{definition}

\begin{theorem}
If $\{t_k\}_{k\in\mathbb{N}}$ is the Halton sequence in $[0,1]^d$ in the pairwise relatively prime bases $q_1,...,q_d$, then
$$
D_j^*(\{t_1,...,t_j\})<\frac{d}{j}+\frac{1}{j}\prod_{k=1}^d\left(\frac{q_k-1}{2\log(q_k)}\log(j)+\frac{q_k+1}{2}\right).
$$
\end{theorem}

Note that given $d$ (and suitable $q_1$,..., $q_d$), we can easily compute in finitely many arithmetic operations and comparisons a constant $C(d)$ such that the above implies
\begin{equation}
\label{star_bound}
D_j^*(\{t_1,...,t_j\})<C(d){(\log(j)+1)^d}/{j}.
\end{equation}
The following theorem says why this is useful.

\begin{theorem}[Koksma--Hlawka inequality \cite{niederreiter1992random}]
\label{int_err_tot_var}
If $f$ has bounded variation $\mathrm{TV}_{[0,1]^d}(f)$ on the hypercube $[0,1]^d$ then for any $t_1,...,t_j$ in $[0,1]^d$
$$
\left|\frac{1}{j}\sum_{k=1}^j f(t_k)-\int_{[0,1]^d}f(x)dx\right|\leq \mathrm{TV}_{[0,1]^d}(f)D_j^*(\{t_1,...,t_j\}).
$$
By re-scaling, if $f$ has bounded variation $\mathrm{TV}_{[-r,r]^d}(f)$ and $s_k=2rt_k-(r,r,...,r)^T$, we obtain
$$
\left|\frac{(2r)^d}{j}\sum_{k=1}^j f(s_k)-\int_{[-r,r]^d}f(x)dx\right|\leq (2r)^d\cdot\mathrm{TV}_{[-r,r]^d}(f)D_j^*(\{t_1,...,t_j\}).
$$
\end{theorem}

Finally, to deal with our choice of basis, we need the following.

\begin{lemma}
\label{Herm_tot_var}
Let $\psi_m(x):=\psi_{m_1}(x_1)\psi_{m_2}(x_2) \cdots\psi_{m_d}(x_d)$ in $d$ dimensions and let $r>0$. Then
\begin{equation}
\mathrm{TV}_{[-r,r]^d}(\psi_m)\leq \left(1+2r\sqrt{2(\left|m\right|+1)}\right)^d-1.
\end{equation}
\end{lemma}
\begin{proof}
We use an alternative form of the total variation which holds for sufficiently smooth functions and can be found in \cite{niederreiter1992random}:
$$
\mathrm{TV}_{[-r,r]^d}(\psi_m)=\sum_{k=1}^d\sum_{1\leq i_1<...<i_k\leq d}\int_{-r}^r...\int_{-r}^r \left|\frac{\partial^k \psi_m}{\partial x_{i_1}...\partial x_{i_k}}(\tilde{x})\right|dx_{i_1}...dx_{i_k},
$$
where $\tilde{x}$ has $\tilde{x}_{j}=x_{j}$ for $j=i_1,...,i_k$ and $\tilde{x}_j=r$ otherwise. We can use the recurrence relations for Hermite functions as well as Cram{\'e}r's inequality (bounds on Hermite functions \cite{cramer_inequality}) to gain the bound
$$
\int_{-r}^r...\int_{-r}^r \left|\frac{\partial^k \psi_m}{\partial x_{i_1}...\partial x_{i_k}}(\tilde{x})\right|dx_{i_1}...dx_{i_k}\leq \left(2r\sqrt{2(\left|m\right|+1)}\right)^{k}.
$$
It follows that
\begin{align*}
\mathrm{TV}_{[-r,r]^d}(\psi_m)&\leq \sum_{k=1}^d\left(2r\sqrt{2(\left|m\right|+1)}\right)^{k}\sum_{1\leq i_1<...<i_k\leq d} 1\\
&=\sum_{k=1}^d\left(2r\sqrt{2(\left|m\right|+1)}\right)^{k} {\binom{d}{k}}=\left(1+2r\sqrt{2(\left|m\right|+1)}\right)^d-1.
\end{align*}\end{proof}

\begin{proposition}
\label{mat_approx}
Given $T\in\Omega^1_{\mathrm{TV}}$ or $T\in\Omega^1_{\mathrm{AN}}$ and $\epsilon>0$, we can approximate the matrix values
$$
\langle (T-zI)\psi_m,(T-zI) \psi_n\rangle\quad\text{and}\quad\langle (T-zI)^*\psi_m,(T-zI)^* \psi_n\rangle
$$
to within $\epsilon$ using finitely many arithmetical operations and comparisons of the relevant information given to us in each class (captured by $\Xi_j^1$ and $\Xi_j^3$ in \S \ref{dfohjg}).
\end{proposition}

\begin{proof}
Let $T\in\Omega^1_{\mathrm{TV}}$ or $T\in\Omega^1_{\mathrm{AN}}$, and $\epsilon>0$. Recall that
$$
T=\sum_{\left|k\right|\leq N}a_k(x)\partial^k,\quad T^*=\sum_{\left|k\right|\leq N}\tilde{a}_k(x)\partial^k.
$$
By expanding out the inner products and also considering the case $a_k=1$, it is sufficient to approximate
$$
\langle a_k\partial^k \psi_m,a_j\partial^j \psi_n\rangle\quad\text{and}\quad\langle \tilde{a}_k\partial^k \psi_m,\tilde{a}_j\partial^j \psi_n\rangle
$$
for all relevant $k,j,m$ and $n$. Due to the symmetry in the assumptions on $T$ and $T^*$, we only need to show that we can compute the first inner product. The proof for the second inner product is analogous. For the choice of the basis functions $\psi_m$, $\partial^k \psi_m$ can be written as a finite linear combination of tensor products of Hermite functions using the recurrence relations. The coefficients in these linear combinations are recursively defined as a function of $k$. Hence, in the inner product, we can assume that there are no partial derivatives. In doing this, we have assumed that we can compute square roots of integers (which occur in the coefficients) to arbitrary precision (recall we want an arithmetic tower). A simple interval bisection routine can achieve this computation. It follows that we only need to consider approximations of inner products of the form
$
\langle a_k\psi_m,a_j \psi_n\rangle.
$

To do so, let $R>1$. By H\"older's inequality and the assumption of polynomially bounded growth on the coefficients $a_k$, we have
\begin{equation*}
\begin{split}
& \int_{\left|x_i\right|\geq R} \left|a_k\overline{a_j}\right|\left|\psi_m\psi_n\right| dx\\
& \leq A_kA_j\left(\int_{\left|x_i\right|\geq R} \left(1+\left|x\right|^{2B_k}\right)^2\left(1+\left|x\right|^{2B_j}\right)^2\psi_m(x)^2dx\right)^{1/2}\left(\int_{\left|x_i\right|\geq R} \psi_n(x)^2dx\right)^{1/2}.
\end{split}
\end{equation*}
The first integral on the right-hand side can be bounded by
\[
16\int_{\mathbb{R}^d} \left|x\right|^{2B}\psi_m(x)^2dx\leq 16\int_{\mathbb{R}^d} \left(x_1^2+...+x_d^2\right)^{B}\psi_m(x)^2dx,
\]
for $B=4(B_k+B_j)$, since we restrict to $\left|x_i\right|\geq R$ with $R>1$ and $\left|x\right|\leq \left\|x\right\|_{2}$. $B$ is even so we can expand out the product $(x_1^2+...+x_d^2)^{B/2}\psi_m$ using the recurrence relations for the Hermite functions. In one dimension this gives
\begin{align*}
&x\psi_m(x)=\sqrt{\frac{m}{2}}\psi_{m-1}(x)+\sqrt{\frac{m+1}{2}}\psi_{m+1}(x),\\
&x^2\psi_m(x)=\sqrt{\frac{m}{2}}x\psi_{m-1}(x)+\sqrt{\frac{m+1}{2}}x\psi_{m+1}(x),\\
&=\sqrt{\frac{m}{2}}\left(\sqrt{\frac{m-1}{2}}\psi_{m-2}(x)+\sqrt{\frac{m}{2}}\psi_{m}(x)\right)+\sqrt{\frac{m+1}{2}}\left(\sqrt{\frac{m+1}{2}}\psi_{m}(x)+\sqrt{\frac{m+2}{2}}\psi_{m+2}(x)\right),
\end{align*}
and so on. We can do the same for tensor products of Hermite functions. In particular, multiplying a tensor product of Hermite functions, $\psi_m$, by $(x_1^2+...+x_d^2)$ induces a linear combination of at most $4d$ such tensor products, each with a coefficient of magnitude at most $(\left|m\right|+2)^2$ and index with $l^\infty$ norm bounded by $\left|m\right|+2$ (allowing repetitions). It follows that $(x_1^2+...+x_d^2)^{B/2}\psi_m$ can be written as a linear combination of at most $(4d)^{B/2}$ such tensor products, each with a coefficient of magnitude at most $(\left|m\right|+B)^B$. Squaring this and integrating, the orthogonality and normalisation of the tensor product of Hermite functions implies that
$$
16\int_{\mathbb{R}^d} (x_1^2+...+x_d^2)^{B}\psi_m(x)^2dx\leq 16(4d)^{B/2}(\left|m\right|+B)^{2B}=:p_1(\left|m\right|).
$$
For the other integral, define $p_2(\left|n\right|):=4d(\left|n\right|+2)^4.$ We then have
\begin{align*}
\int_{\left|x_i\right|\geq R} \psi_n^2dx\leq \frac{1}{R^{4}}\int_{\mathbb{R}^d} \left|x\right|^{4}\psi_n^2dx\leq \frac{p_2(\left|n\right|)}{R^{4}},
\end{align*}
by using the same argument as above but with $B=2$.

So given $\delta>0$ and $n,m,B,A_k,A_j$, (and $d$) we can choose $r\in\mathbb{N}$ large such that
$$
\int_{\left|x_i\right|\geq r} \left|a_k\overline{a_j}\right|\left|\psi_m\psi_n\right| dx\leq A_kA_j\frac{p_1(\left|m\right|)^{1/2}p_2(\left|n\right|)^{1/2}}{r^2}\leq\delta.
$$
We now have to consider the cases $T\in\Omega^1_{\mathrm{TV}}$ or $T\in\Omega^1_{\mathrm{AN}}$ separately, noting that it is sufficient to approximate the integral $\int_{\left|x_i\right|\leq r} a_k\overline{a_j}\psi_m\psi_n dx$ to any given precision. For notational convenience, let
$$
L_r(m)=\left[1+(3^d+1)\left(\left(1+2r\sqrt{2(\left|m\right|+1)}\right)^d-1\right)\right]
$$
so that with the definition of $\left\|\cdot\right\|_{\mathcal{A}_r}$, we have via Lemma \ref{Herm_tot_var} that $\left\|\psi_m\right\|_{\mathcal{A}_r}\leq  L_r(m).$

\textbf{Case 1: $T\in\Omega^1_{\mathrm{TV}}$.} Given $k,j,m,n,\delta$ and $r\in\mathbb{N}$ as above, choose $M$ large such that
\begin{equation}
(2r)^d\cdot\frac{C(d)\big(\log(M)+1\big)^d}{M}\cdot c_r^2\cdot L_r(m)\cdot L_r(n)\leq \delta/2,
\end{equation}
where $C(d)$ is as (\ref{star_bound}) and $c_r$ controls the total variation as in (\ref{tot_var_bound}). Again, note that such an $M$ can be chosen in finitely many arithmetic operations and comparisons with the given data, assuming that logarithms and square roots can be computed to arbitrary precision (say by a power series representation and bound on the remainder). Using the fact that $\mathcal{A}_r$ is a Banach algebra (so that we can bound the norms of products of functions by the product of their norms) and Theorem \ref{int_err_tot_var}, it follows that
$$
\left|\frac{(2r)^d}{M}\sum_{l=1}^M a_k(s_l)\overline{a_j}(s_l)\psi_m(s_l)\psi_n(s_l)-\int_{\left|x_i\right|\leq r} a_k\overline{a_j}\psi_m\psi_n dx\right|\leq \delta/2,
$$
where $s_l=2rt_l-(r,r,...,r)^T$ are the rescaled Halton points. Hence it is enough to show that each product $a_k(s_l)\overline{a_j}(s_l)\psi_m(s_l)\psi_n(s_l)$ can be computed to a given accuracy using finitely many arithmetic operations and comparisons. Since each $s_l\in\mathbb{Q}^d$, we can evaluate $a_k(s_l)\overline{a_j}(s_l)$. Note that we can compute $\exp(-x^2/2)$ to arbitrary precision with finitely many arithmetic operations and comparisons (e.g., by a power series representation and bound on the remainder) and that we can compute the coefficients of the polynomials $Q_m$ with
$
\psi_m(x)=Q_m(x)\exp(-x^2/2)
$
using the recursion formulae to any given precision. It follows that we can compute $\psi_m(s_l)\psi_n(s_l)$ to a given accuracy using finitely many arithmetic operations and comparisons. Using the bounds on the $a_k$ and $\overline{a_j}$ and Cram{\'e}r's inequality, we can bound the error in the product, and hence the result follows.

\textbf{Case 2: $T\in\Omega^1_{\mathrm{AN}}$.}
On the compact cube $\left|x_i\right|\leq r$ the double series
$$
a_k(x)\overline{a_j(x)}=\sum_{t\in(\mathbb{Z}_{\geq 0})^d}\sum_{s\in(\mathbb{Z}_{\geq 0})^d}a_k^t\overline{a_j^s}x^{t+s}
$$
converges uniformly (recall that $\{a_k^t\}_{t\in(\mathbb{Z}_{\geq 0})^d}$ are the power series coefficients for $a_k$). It follows that we can exchange this series and integration to write
\begin{equation}
\label{series_tail}
\int_{\left|x_i\right|\leq r} a_k\overline{a_j}\psi_m\psi_n dx=\sum_{t,s\in(\mathbb{Z}_{\geq 0})^d}a_k^t\overline{a_j^s}\int_{\left|x_i\right|\leq r} x^{s+t}\psi_m(x)\psi_n(x) dx.
\end{equation}
But $\left|\int_{\left|x_i\right|\leq r} x^{s+t}\psi_m(x)\psi_n(x) dx\right|$ is bounded by
$
r^{\left|t\right|+\left|s\right|}\int_{x\in\mathbb{R}^d} \left|\psi_m\right|\left|\psi_n\right| dx\leq r^{\left|t\right|+\left|s\right|},
$
by H\"older's inequality. Let $\tau=r/(r+1)$. Since we know $d_r$ in (\ref{bound_tail}), we can bound the tail of the series in (\ref{series_tail}) by
$$
d_r^2\sum_{\left|t\right|,\left|s\right|>M}\tau^{\left|t\right|+\left|s\right|}\leq d_r^2\left(\sum_{\left|t\right|>M}\tau^{\frac{\left|t_1\right|}{d}+...+\frac{\left|t_d\right|}{d}}\right)^2,
$$
using the fact that $\left|x\right|\leq(\left|x_1\right|+...+\left|x_d\right|)/d$. We can explicitly sum this series (as the difference of geometric series) to gain the bound
\begin{align*}
d_r^2\left[\frac{1-(1-\tau^{(M+1)/d})^d}{\left(1-\tau^{1/d}\right)^{d}}\right]^2.
\end{align*}
Given $r$ and $d_r$ (and $d$), we can keep increasing $M$ and evaluating the bound (strictly speaking an upper bound accurate to $1/M$ say), to choose $M$ large such that the tail is smaller than $\delta/2$ for any given $\delta>0$. It follows that it is enough to estimate integrals of the form
$
\int_{\left|x_i\right|\leq r} x^{s+t}\psi_m(x)\psi_n(x) dx.
$
Using the recurrence relations for Hermite functions and writing $\psi_m(x)=Q_m(x)\exp(-x^2/2)$, it is enough to split the multidimensional integral up as products and sums of one-dimensional integrals of the form
$
\int_{-r}^r x^a \exp(-x^2) dx,
$
for $a\in\mathbb{Z}_{\geq0}$. Again, we have assumed that we can compute the coefficients of the $Q_m$ to any given accuracy using finitely many arithmetic operations and comparisons, and using this we can bound the total error of the expression by $\delta/2$. The above integral vanishes unless $a$ is even, so integration by parts (again assuming we can evaluate $\exp(-x^2)$ to any desired accuracy) reduces this to estimating
$
\int_{-r}^r \exp(-x^2) dx.
$
Consider the Taylor series for $\exp(-x^2)$. The tail can be bounded by
$$
\sum_{k>N}\frac{r^{2k}}{k!}\leq \frac{r^{2N}}{N!}\exp(-r^2).
$$
Integrating this estimate over the interval $[-r,r]$, we can bound this by any given $\eta>0$ by choosing $N$ large enough. We can then explicitly compute $\int_{-r}^r \sum_{k\leq N}{x^{2k}}/{k!} dx$. Keeping track of all the errors is elementary. Hence $\int_{\left|x_i\right|\leq r} a_k\overline{a_j}\psi_m\psi_n dx$ can be approximated with finitely many arithmetic operations and comparisons, as required.
\end{proof}

In some cases, we can also directly compute matrix elements without the cut-off argument used in the above proof. For instance, if each $a_k(x)$ (and hence $\tilde{a}_k(x)$) is a polynomial then we can simply integrate the power series to compute $\langle a_k(x)\psi_m,a_j(x) \psi_n\rangle$ and use the recurrence relations for Hermite functions. If we know a bound on the degree of the polynomials, then clearly we can compute
\begin{equation}
\label{mat_vals_analytic}
\langle (T-zI)\psi_m,(T-zI) \psi_n\rangle\quad\text{and}\quad\langle (T-zI)^*\psi_m,(T-zI)^* \psi_n\rangle
\end{equation}
to within $\epsilon$ using finitely many arithmetical operations and comparisons directly.

We can now prove the positive parts of Theorems \ref{PDE1} and \ref{PDE2}.

\begin{proof}[Proof of inclusions in Theorems \ref{PDE1} and \ref{PDE2}]

\textbf{Step 1}: $\{\Xi_1^1,\Omega^1_{\mathrm{TV}}\},\{\Xi_1^3,\Omega^1_{\mathrm{AN}}\}\in\Sigma_A^1$. The proof of this simply strings together the above results. The linear span of $\{e_1,e_2,...\}$ (the reordered Hermite functions) is a core of $T$ and $T^*$ by Proposition \ref{basis_core}. By Proposition \ref{mat_approx}, we can compute the inner products
$
\langle (T-zI)e_j,(T-zI) e_i\rangle
$
and
$
\langle (T-zI)^*e_j,(T-zI)^* e_i\rangle
$
up to arbitrary precision with finitely many arithmetic operations and comparisons. Using Lemma \ref{approx_mat_vals}, given $z\in\mathbb{C}$, we can compute some approximation $\upsilon_n(z,T)$ in finitely many arithmetic operations and comparisons such that
$$
\left|\upsilon_n(z,T)^2-\min\{\sigma_{\mathrm{inf}}((T-zI)|_{P_n(l^2(\mathbb{N}))}),\sigma_{\mathrm{inf}}((T^*-\bar{z}I)|_{P_n(l^2(\mathbb{N}))})\}^2\right|\leq \frac{1}{n^2}.
$$
We now set
\begin{equation}
\label{key_fun_PDE}
\gamma_n(z,T)=\upsilon_n(z,T)+1/n.
\end{equation}
Then $\gamma_n$ satisfies the hypotheses of Proposition \ref{conv_alg_unbounded}. The proof of Theorem \ref{unbounded_theorem} also makes clear that we have error control since $\gamma_n(z,T)\geq \left\|R(z,T)\right\|^{-1}$.

\textbf{Step 2}: $\{\Xi_2^1,\Omega^1_{\mathrm{TV}}\},\{\Xi_2^3,\Omega^1_{\mathrm{AN}}\}\in\Sigma_A^1$. Consider the sequence of functions $\gamma_n$ defined by equation (\ref{key_fun_PDE}). These converge uniformly to $\left\|R(z,T)\right\|^{-1}$ on compact subsets of $\mathbb{C}$ and satisfy $\gamma_n(z,T)\geq \left\|R(z,T)\right\|^{-1}$. We can now apply Proposition \ref{pseudospec_unbounded}.

\textbf{Step 3:} $\{\Xi_1^2,\Omega^2_{\mathrm{TV}}\},\{\Xi_2^2,\Omega^2_{\mathrm{TV}}\}\in\Delta^A_2$. Let $T\in\Omega^2_{\mathrm{TV}}$. Our strategy will be to compute the inner products
$
\langle (T-zI)e_j,(T-zI) e_i\rangle\quad\text{and}\quad\langle (T-zI)^*e_j,(T-zI)^* e_i\rangle
$
to an error that decays rapidly enough as we increase the cut-off parameter $r$. We follow the proof of Proposition \ref{mat_approx} closely. Recall that given $n,m$, we can choose $r\in\mathbb{N}$ large such that
$$
\int_{\left|x_i\right|\geq r} \left|a_k\overline{a_j}\right|\left|\psi_m\psi_n\right| dx\leq A_kA_j\frac{p_1(\left|m\right|)^{1/2}p_2(\left|n\right|)^{1/2}}{r^2},
$$
with the crucial difference that now we do not assume we can compute $A_k,A_j,p_1$ or $p_2$. It follows that there exists some polynomial $p_3$, with coefficients not necessarily computable from the given information, such that
$$
\int_{\left|x_i\right|\geq r} \left|a_k\overline{a_j}\right|\left|\psi_m\psi_n\right| dx\leq \frac{p_3(\left|m\right|,\left|n\right|)}{r^2},
$$
for all $\left|j\right|,\left|k\right|\leq N$. Now we use the sequence $b_r$ to bound the error in the integral over the compact cube asymptotically. We assume without loss of generality that $b_r$ increases monotonically to infinity as $r\rightarrow\infty$. Using Halton sequences and the same argument in the proof of Proposition \ref{mat_approx}, we can approximate 
$
\int_{\left|x_i\right|\leq r} a_k\overline{a_j}\psi_m\psi_n dx,
$
with an error that, asymptotically up to some unknown constant, is bounded by
\begin{equation}
\label{asym_err_TV}
r^d\cdot \frac{\big(\log(M)+1\big)^d}{M}\cdot b_r^2\cdot L_r(m)\cdot L_r(n),
\end{equation}
where $M$ is the number of Halton points. We can let $M$ depend on $r,n$ and $m$ such that (\ref{asym_err_TV}) is bounded by a constant times $1/r^2$. It follows that we can bound the total error in approximating
$
\langle a_k\psi_m,a_j \psi_n\rangle
$
for any $j,k$ by $p_3(\left|m\right|,\left|n\right|)/{r^2}$, by making the coefficients of $p_3$ larger if necessary. We argue similarly for the adjoint and note that
$
\langle (T-zI)\psi_m,(T-zI) \psi_n\rangle
$
and 
$\langle (T-zI)^*\psi_m,(T-zI)^* \psi_n$
are both approximated to within
$$
(1+\left|z\right|^2)\frac{P(\left|m\right|,\left|n\right|)}{r^2},
$$
for some unknown polynomial $P$.
Hence we can apply Lemma \ref{approx_mat_vals} (the form where we do not know the error in inner product estimates), changing the polynomial $P$ to take into account the basis mapping from $\mathbb{Z}_{\geq0}^d$ to $\mathbb{N}$ to some polynomial $Q$, to gain some approximation $\upsilon_n(z,T)$ in finitely many arithmetic operations and comparisons such that
\begin{equation}
\label{key_fun_PDE2}
\left|\upsilon_n(z,T)^2-\min\{\sigma_{\mathrm{inf}}((T-zI)|_{P_n(l^2(\mathbb{N}))}),\sigma_{\mathrm{inf}}((T^*-\bar{z}I)|_{P_n(l^2(\mathbb{N}))})\}^2\right|\leq \frac{n(1+\left|z\right|^2)Q(n)}{r(n,z)^2}+\frac{1}{n^3}.
\end{equation}
We now choose $r(z,n)$ larger if necessary such that $r(z,n)\geq (1+\left|z\right|^2)\exp(n)$. We now set
$
\gamma_n(z,T)=\upsilon_n(z,T)+1/n.
$
Then $\gamma_n$ satisfies the hypotheses of Proposition \ref{conv_alg_unbounded} and Proposition \ref{pseudospec_unbounded} since the error in (\ref{key_fun_PDE2}) decays faster than $1/n^2$. We can use these propositions to build the required arithmetical algorithm.

\textbf{Step 4:} $\{\Xi_1^4,\Omega^2_{\mathrm{AN}}\},\{\Xi_2^4,\Omega^2_{\mathrm{AN}}\}\in\Delta^A_2$.
We argue as in step 3. To control the error in the approximation of the integral over a compact hypercube, choose the cut-off $M(r)$ such that
$$
\sum_{\left|t\right|,\left|s\right|>M(r)}\left(\frac{r}{r+1}\right)^{\left|t\right|+\left|s\right|}\leq \frac{1}{b_r^2r^2}.
$$
It follows that there exists some (unknown) constant $B$ such that we can bound the error in approximating
$
\int_{\left|x_i\right|\leq r} a_k\overline{a_j}\psi_m\psi_n dx
$
by $B/r^2$. Here we have absorbed the arbitrarily small error that comes from approximating the integral of the truncated power series using finitely many arithmetic operations and comparisons. The rest of the argument is the same as in step 3.
\end{proof}

\subsection{Proofs of impossibility results in Theorems \ref{PDE1} and \ref{PDE2}}

We first deal with Theorem \ref{PDE1}. Recall the maps
\begin{align*}
&\Xi_1^k: \Omega_{\mathrm{TV}}^k\ni T\mapsto \mathrm{Sp}(T)\in \mathrm{Cl}(\mathbb{C}),\quad \text{for }k=1,2,\\
&\Xi_2^k: \Omega_{\mathrm{TV}}^k\ni T\mapsto \mathrm{Sp}_{\epsilon}(T) \in \mathrm{Cl}(\mathbb{C}),\quad \text{for } k=1,2.
\end{align*}
We split up the arguments to deal first with $\Omega^1_{\mathrm{TV}}$, and then $\Omega^2_{\mathrm{TV}}$.

\begin{proof}[Proof that $\{\Xi_j^1, \Omega_{\mathrm{TV}}^1\}\notin\Delta^G_1$]

Suppose first for a contradiction that a height one tower, $\Gamma_n$, exists for the problem $\{\Xi_1^1,\Omega_{\mathrm{TV}}^1\}$ such that $d_{\mathrm{AW}}(\Gamma_n(T),\Xi_1^1(T))\leq 2^{-n}$. We deal with the one-dimensional case since higher dimensions are similar. Let $\rho(x)$ be any smooth bump function with maximum value $1$, minimum value $0$ and support $[0,1]$. Let $\rho_n$ denote the translation of $\rho$ to have support $[n,n+1]$. We consider the two (self-adjoint and bounded) operators
$$
(T_0u)(x)=0, \qquad(T_mu)(x)=\rho_m(x)u(x),
$$
which have spectra $\{0\}$ and $[0,1]$ respectively. For these operators we can take the polynomial bound (the $\{A_k\}$ and $\{B_k\}$) to be $1$ and the total variation bound to be $c_r=1+(3^d+1)\mathrm{TV}_{[0,1]}(\rho)$. When we compute $\Gamma_2(T_0)$, we only use finitely many evaluations of the coefficient function $a_0(x)=0$ (as well as the other given information). We can then choose $m$ large such that the support of $\rho_m$ does not intersect the points of evaluation. By assumptions (ii) and (iii) in Definition \ref{Gen_alg}, $\Gamma_2(T_m)=\Gamma_2(T_0)$. But this contradicts the triangle inequality since $d_{\mathrm{AW}}(\{0\},[0,1])\geq 1$

To argue for the pseudospectrum let $\epsilon>0$ and note that $2\epsilon\notin\mathrm{Sp}_\epsilon(T_0)$ but $2\epsilon\in\mathrm{Sp}_\epsilon(\epsilon T_m)$. We now alter the given $c_r$ to $\epsilon(1+(3^d+1)\mathrm{TV}_{[0,1]}(\rho))$ and the polynomial bound to $\epsilon$. The argument is now exactly as before. Namely, we choose $n$ large such that 
\[
d_{\mathrm{AW}}(\Gamma_n(T_0),[-\epsilon,2\epsilon])>2^{-n},
\]
and then choose $m$ large such that 
$
\Gamma_n(T_0)=\Gamma_n(\epsilon T_m).
$
\end{proof}

\begin{proof}[Proof that $\{\Xi_j^2, \Omega_{\mathrm{TV}}^2\}\notin\Sigma^G_1\cup\Pi^G_1 $]

Suppose first of all that a $\Sigma^G_1$ tower, $\Gamma_n$, exists for $\{\Xi_1^2,\Omega_{\mathrm{TV}}^2\}$. We deal with the one-dimensional case since higher dimensions are similar. Consider the operators
$$
(T_0u)(x)=0, \qquad(T_1u)(x)=f(x)u(x),
$$
where we define $f$ in terms of $\Gamma_n$ as follows. We will ensure that $f(x)=1$ except for finitely many values of $x$, where it takes the value $0$. Hence $T_0$ and $T_1$ have spectra $\{0\}$ and $\{1\}$, respectively, and are both self-adjoint. Note that once the zeros of $f$ are fixed, this choice ensures that $f$ has total variation bounded by a constant on any hypercube and hence we may take $b_r=1$ for all $r\in\mathbb{N}$. There exists some $n$ such that $\Gamma_n(T_0)$ contains $z_n\in B_{1/8}(0)$ with a guaranteed error estimate of $\mathrm{dist}(z_n,\mathrm{Sp}(T_0))\leq 1/4$. But $\Gamma_{n}(T_0)$ can only depend on finitely many evaluations of $0$ (as well as $b_r=1$ and the trivial choice of $g_j(x)=x$). We choose $f$ to be zero at precisely these evaluation points. By assumptions (ii) and (iii) in Definition \ref{Gen_alg}, $\Gamma_{n}(T_1)=\Gamma_{n}(T_0)$, including the given error estimates, which is the required contradiction.

For $\{\Xi_2^2, \Omega_{\mathrm{TV}}^2\}\notin\Sigma^G_1$, given $\epsilon>0$ we replace $f$ by $3\epsilon f$ in the above argument and keep all other inputs the same. Hence $T_0$ and $T_1$ have $\epsilon$-pseudospectra $[-\epsilon,\epsilon]$ and $[2\epsilon,4\epsilon]$ respectively. We note that again there exists some $n$ such that $\Gamma_n(T_0)$ contains $z_n\in B_{\epsilon/8}(0)$ with a guaranteed error bound of $\mathrm{dist}(z_n,\mathrm{Sp}_{\epsilon}(T_0))\leq \epsilon/4$. But $\Gamma_{n}(T_0)$ can only depend on finitely many evaluations of $0$ (as well as $b_r=1$ and the trivial choice of $g_j(x)=x$). We choose $f$ to be zero at precisely these evaluation points. By assumptions (ii) and (iii) in Definition \ref{Gen_alg}, $\Gamma_{n}(T_1)=\Gamma_{n}(T_0)$, including the given error bounds, which is the required contradiction.

To argue that neither problem lies in $\Pi^G_1$, we can use the same arguments in the proof that $\{\Xi_j^1, \Omega_{\mathrm{TV}}^1\}\notin\Delta^G_1$. The only change now is that the algorithm, $\Gamma_n$, used to derive the contradiction provides $\Pi_1^G$ information rather than $\Delta_1^G$. For the spectrum, we consider the operators
$$
(T_0u)(x)=0\quad \text{and}\quad(T_mu)(x)=\rho_m(x)u(x),
$$
and choose $n$ large such that $\Gamma_n(T_0)$ produces the guarantee $\mathrm{Sp}(T_0)\cap B_{1/4}(0)^c=\emptyset$. For $m$ sufficiently large, we argue as before to get $\Gamma_n(T_m)=\Gamma_n(T_0)$, including the guarantee, the required contradiction. Again a similar argument works for the pseudospectrum by rescaling $T_m$ to $2\epsilon T_m$.
\end{proof}

We now deal with the impossibility results in Theorem \ref{PDE2}, where
\begin{align*}
&\Xi_1^{k+2}: \Omega_{\mathrm{AN}}^k\ni T\mapsto \mathrm{Sp}(T)\in \mathrm{Cl}(\mathbb{C}),\quad \text{for }k=1,2,\\
&\Xi_2^{k+2}: \Omega_{\mathrm{AN}}^k\ni T\mapsto \mathrm{Sp}_{\epsilon}(T) \in \mathrm{Cl}(\mathbb{C}),\quad \text{for } k=1,2.
\end{align*}

\begin{proof}[Proof that $\{\Xi_j^3, \Omega_{\mathrm{AN}}^1\}\notin\Delta^G_1$]

Suppose for a contradiction that a height one tower, $\Gamma_n$, exists for $\{\Xi_1^3,\Omega_{\mathrm{AN}}^1\}$ such that $d_{\mathrm{AW}}(\Gamma_n(T),\Xi_1^3(A))\leq 2^{-n}$. Now consider the two (self-adjoint and bounded) operators
$$
(T_1u)(x)=0\quad \text{and}\quad(T_2u)(x)=x^k\exp(-x^2)u(x)/s_k,
$$
where $k$ is even and will be chosen later. We choose $s_k$ such that the range of the function is $x^k\exp(-x^2)/s_k$ is $[0,1]$ and hence $T_2$ has spectrum $[0,1]$. We can take the polynomial bounding function to be the constant $1$ for both operators and must show that we can use the same $d_r$ for both operators in (\ref{bound_tail}), independent of $k$. Simple calculus yields that $s_k=(k/(2e))^{k/2}$. It follows that such a $d_r$ must satisfy
\begin{equation}
\label{bounded?}
\left(\frac{2e}{k}\right)^{k/2}\frac{(r+1)^{2m+k}}{m!}\leq d_r,\quad\forall k\in2\mathbb{N},m\in\mathbb{Z}_{\geq0}.
\end{equation}
Hence it suffices to show that the function on the left-hand side of (\ref{bounded?}) is bounded (as a function of $m,k$ for all $r\in\mathbb{N}$). Using Stirling's approximation (explicitly the bounds on $m!$) this will follow if we can show
$$
\frac{r^{2m+k}}{k^{k/2}m^{m+1/2}}\leq \left(\frac{r}{\sqrt{k}}\right)^{k}\left(\frac{r}{\sqrt{m}}\right)^{2m}
$$
is bounded for all $r\in\mathbb{N}$ (now with $m>1)$. But this is obvious.

We can now choose $k$ (which depends on the algorithm $\Gamma_n$) to gain a contradiction. Since $\mathrm{Sp}(T_1)=\{0\}$ and $1\in\mathrm{Sp}(T_2)$ for all even $k$, there exists $n$ such that $\mathrm{dist}(1,\Gamma_n(T_1))>1/4$ but $\mathrm{dist}(1,\Gamma_n(T_2))<1/4$. However, $\Gamma_n(T)$ can only depend on finitely many of the coefficients $\{c_j\}$, say $c_1,...,c_{\tilde{N}(T,n)}$, of $T$ (as well as the other given information). By assumption (iii) in Definition \ref{Gen_alg}, we can choose $k$ such that the coefficient corresponding to $x^k$, call it $c_{l_k}$, has $l_k>\tilde{N}(T_1,n)$ and get $\Gamma_n(T_1)=\Gamma_n(T_2)$, the required contradiction.

To show $\{\Xi_2^3,\Omega_{\mathrm{AN}}^1\}\notin\Delta^G_1$, we use exactly the same argument as above. To gain the necessary separation $3\epsilon\notin\mathrm{Sp}_\epsilon(T_1)$ but $3\epsilon\in\mathrm{Sp}_\epsilon(T_2)$, we rescale $T_2$ to $3\epsilon T_2$. Then there exists $n$ such that $\mathrm{dist}(3\epsilon,\Gamma_n(T_1))>\epsilon/2$ but $\mathrm{dist}(3\epsilon,\Gamma_n(T_2))<\epsilon/2$. The rest of the contradiction follows.
\end{proof}

\begin{proof}[Proof that $\{\Xi_j^4, \Omega_{\mathrm{AN}}^2\},\{\Xi_j^4, \Omega_p\}\notin\Sigma^G_1\cup\Pi^G_1$.]
Since $\Omega_p\subset\Omega_{\mathrm{AN}}^2$, it is enough to prove the results for $\Omega_p$.

Suppose for a contradiction that there exists a $\Sigma^G_1$ algorithm, $\Gamma_n$, for $\{\Xi_1^4,\Omega_{p}\}$. Consider
$$
(T_1u)(x)=xu(x)\quad \text{and}\quad(T_2u)(x)=(x-x^k)u(x),
$$
where $k$ is even and chosen later. $(T_j\pm iI)C_0^\infty(\mathbb{R})$ are dense in $L^2(\mathbb{R})$ with $T_j$ initially defined on $C_0^\infty(\mathbb{R})$ symmetric. It follows that the closure of $T_j|_{C_0^\infty(\mathbb{R})}$ is self-adjoint and hence that $T_j\in\Omega_p$. Note that $\mathrm{Sp}(T_1)=\mathbb{R}$ but $\mathrm{Sp}(T_2)\subset(-\infty,1]$. Now choose $n$ such that $\Gamma_n(T_1)$ contains a point $z_n\in B_{1/4}(2)$ with a guaranteed error estimate of $\mathrm{dist}(z_n,\mathrm{Sp}(T_1))\leq 1/4$. However, $\Gamma_n(T)$ can only depend on the first $\tilde{N}(T,n)$ coefficients, $c_1,...,c_{\tilde{N}(T,n)}$, of $T$ (as well as the trivial choice $g_j(x)=x$ and the numbers $b_n=n!$). By assumption (iii) in Definition \ref{Gen_alg}, we can choose $k$ such that the coefficient corresponding to $x^k$, call it $c_{r_k}$, has $r_k>\tilde{N}(T_1,n)$ and get $\Gamma_n(T_1)=\Gamma_n(T_2)$, the required contradiction. Similarly by rescaling as above, we get $\{\Xi_2^4,\Omega_{p}\}\notin\Sigma^G_1$.

To show $\{\Xi_1^4,\Omega_{p}\}\notin\Pi^G_1$ we argue the same way, but now set $(T_1u)(x)=0$ and $(T_2u)(x)=x^ku(x)$. As before, $T_j\in\Omega_p$, but now $\mathrm{Sp}(T_1)=\{0\}$ and $1\in\mathrm{Sp}(T_2)$. Choose $n$ such that $\Gamma_n(T_1)$ produces the guarantee $\mathrm{Sp}(T_1)\cap B_{1/4}(0)^c=\emptyset$. Again, choose $k$ such that $c_{r_k}$ has $r_k>\tilde{N}(T_1,n)$ and we get $\Gamma_n(T_1)=\Gamma_n(T_2)$, the required contradiction. Rescaling and using the same argument shows $\{\Xi_2^4,\Omega_{p}\}\notin\Pi^G_1$.
\end{proof}

\section{Proofs of Theorems on Discrete Spectra}
\label{pf_disc}

We first need some results on finite section approximations to the discrete spectrum of a Hermitian operator below the essential spectrum. There are two cases to consider. Either there are infinitely many eigenvalues below the essential spectrum, or there are only finitely many. The following two lemmas are well-known and follow from the `min-max' theorem characterising eigenvalues.

\begin{lemma}
\label{case1_fs}
Let $B\in\mathcal{B}(l^2(\mathbb{N}))$ be self-adjoint with eigenvalues $\lambda_1\leq\lambda_2\leq...$ (infinitely many, counted according to multiplicity) below the essential spectrum. Consider the finite section approximates $B_n=P_nBP_n\in\mathbb{C}^n$ and list the eigenvalues of $B_n$ as $\mu^{(n)}_1\leq\mu^{(n)}_2\leq...\leq\mu^{(n)}_n$. Then the following hold:
\begin{enumerate}
	\item $\lambda_j\leq\mu^{(n)}_j$ for $j=1,...,n$,
	\item for any $j\in\mathbb{N}$, $\mu_j^{(n)}\downarrow\lambda_j$ as $n\rightarrow\infty$ ($j\leq n$, so that $\mu_j^{(n)}$ makes sense). 
\end{enumerate}
\end{lemma}

\begin{lemma}
\label{case2_fs}
Let $B\in\mathcal{B}(l^2(\mathbb{N}))$ be self-adjoint with finitely many eigenvalues $\lambda_1\leq\lambda_2\leq...\leq\lambda_m$ (counted according to multiplicity) below the essential spectrum and let $a=\inf\{x:x\in\mathrm{Sp}_{\mathrm{ess}}(B)\}$. For $j>m$ we set $\lambda_j=a$. Consider the finite section approximates $B_n=P_nBP_n\in\mathbb{C}^n$ and list the eigenvalues of $B_n$ as $\mu^{(n)}_1\leq\mu^{(n)}_2\leq...\leq\mu^{(n)}_n$. Then the following hold:
\begin{enumerate}
	\item $\lambda_j\leq\mu^{(n)}_j$ for $j=1,...,n$,
	\item for any $j\leq m$, $\mu_j^{(n)}\downarrow\lambda_j$ as $n\rightarrow\infty$ ($j \leq n$ so that $\mu_j^{(n)}$ makes sense),
	\item given $\epsilon>0$ and $k\in\mathbb{N}$, there exists $N$ such that for all $n\geq N$, $\mu^{(n)}_{k}\leq a+\epsilon$.
\end{enumerate}
\end{lemma}

\begin{proof}[Proof of Theorem \ref{discreteojoiojo} for $\Xi_1^d$]
\textbf{Step 1}: $\{\Xi_1^d,\Omega_{\mathrm{D}}^d\}\notin\Delta_2^G$. Suppose this were false and that there exists some height one tower $\Gamma_n$ solving the problem. Consider the matrix operators $A_m=\mathrm{diag}\{0,0,...,0,2\}\in\mathbb{C}^{m\times m}$ and $C=\mathrm{diag}\{0,0,...\}$ and set
$$
A=\mathrm{diag}\{1,2\}\oplus\bigoplus_{m=1}^{\infty}A_{k_m},
$$
where we choose an increasing sequence $k_m$ inductively as follows. Set $k_1=1$ and suppose that $k_1,...,k_m$ have been chosen. $\mathrm{Sp}_{d}(\mathrm{diag}\{1,2\}\oplus A_{k_1}\oplus A_{k_2}\oplus...\oplus A_{k_m}\oplus C)=\{1,2\}$ is closed and so there exists some $n_m\geq m$ such that if $n\geq n_m$, then
\begin{equation}
\label{near2}
\mathrm{dist}(2,\Gamma_n(\mathrm{diag}\{1,2\}\oplus A_{k_1}\oplus...\oplus A_{k_m}\oplus C)\leq \frac{1}{4}.
\end{equation}
Now let $k_{m+1}\geq \max\{N(\mathrm{diag}\{1,2\}\oplus A_{k_1}\oplus...\oplus A_{k_m}\oplus C,n_m),k_m+1\}$. Arguing as in the proof of Theorem \ref{unbounded_test_theorem}, it follows that $\Gamma_{n_m}(A)=\Gamma_{n_m}(\mathrm{diag}\{1,2\}\oplus A_{k_1}\oplus...\oplus A_{k_m}\oplus C)$. But $\Gamma_{n_m}(A)$ converges to $\mathrm{Sp}_d(A)=\{1\}$, contradicting (\ref{near2}). 

\textbf{Step 2}: $\{\Xi_1^d,\Omega_{\mathrm{N}}^d\}\in\Sigma_2^A$. We now construct an arithmetic height two tower for $\Xi_1^d$ and the class $\Omega_{\mathrm{N}}^d$. To do this, we recall that a height two tower $\tilde{\Gamma}_{n_2,n_1}$ for the essential spectrum of operators in $\Omega_{\mathrm{N}}^d$ was constructed in \cite{ben2015can}. For completeness, we write out the algorithm here.\footnote{The actual algorithm is slightly more complicated to avoid the empty set, but its listed properties still hold.} Let $P_n$ be the usual projection onto the first $n$ basis elements and set $Q_n=I-P_n$. Define
\begin{align*}
\mu_{m,n}(A)&:=\min\{\sigma_{\mathrm{inf}}(P_{f(n)}(A-zI)|_{Q_mP_n(l^2(\mathbb{N}))}),\sigma_{\mathrm{inf}}(P_{f(n)}(A-zI)^*|_{Q_mP_n(l^2(\mathbb{N}))})\},\\
G_n&:=\min\left\{\frac{s+it}{2^n}:s,t\in\{-2^{2n},...,2^{2n}\}\right\},\\
\Upsilon_m(z)&:=z+\{w\in\mathbb{C}:\left|\mathrm{Re}(w)\right|,\left|\mathrm{Im}(w)\right|\leq2^{-(m+1)}\}.
\end{align*}
We then define the following sets for $n> m$:
\begin{align*}
S_{m,n}(z)&:=\{j=m+1,...,n:\exists w\in\Upsilon_m(z)\cap G_j\text{ with }\mu_{m,i}(w)\leq1/m\},\\
T_{m,n}(z)&:=\{j=m+1,...,n:\exists w\in\Upsilon_m(z)\cap G_j\text{ with }\mu_{m,i}(w)\leq1/(m+1)\},\\
E_{m,n}(z)&:=\left|S_{m,n}(z)\right|+\left|T_{m,n}(z)\right|-n,\\
I_{m,n}&:=\left\{z\in\left\{\frac{s+it}{2^m}:s,t\in\mathbb{Z}\right\}:E_{m,n}(z)>0\right\}.
\end{align*}
Finally, we define for $n_1>n_2$
$$
\tilde{\Gamma}_{n_2,n_1}(A)=\bigcup_{z\in I_{n_2,n_1}}\Upsilon_{n_2}(z),
$$
and set $\tilde{\Gamma}_{n_2,n_1}(A)=\{1\}$ if $n_1\leq n_2$.
Furthermore, the tower has the following desirable properties:
\begin{enumerate}
	\item For fixed $n_2$, the sequence $\tilde{\Gamma}_{n_2,n_1}(A)$ is eventually constant as we increase $n_1$,
	\item The sets $\lim_{n_1\rightarrow\infty}\tilde{\Gamma}_{n_2,n_1}(A)=:\tilde{\Gamma}_{n_2}(A)$ are nested, converging down to $\mathrm{Sp}_{\mathrm{ess}}(A)$. 
\end{enumerate}
We also need the height one tower, $\hat{\Gamma}_n$, for the spectrum of operators in $\Omega_{\mathrm{N}}^d$ discussed in \S \ref{gen_unbd} and \S \ref{pf_unb_gr}. Note that $\hat{\Gamma}_n(A)$ is a finite set for all $n$. For $z\in\hat{\Gamma}_n(z)$, this also outputs an error control $E(n,z)$ such that $\mathrm{dist}(z,\mathrm{Sp}(A))\leq E(n,z)$ and such that $E(n,z)$ converges to the true distance to the spectrum uniformly on compact subsets of $\mathbb{C}$ (with the choice of $g(x)=x$ since the operator is normal). We now fit the pieces together and initially define
$$
\zeta_{n_2,n_1}(A)=\{z\in\hat{\Gamma}_{n_1}(A):E({n_1},z)<\mathrm{dist}(z,\tilde{\Gamma}_{n_2,n_1}(A)+B_{1/n_2}(0))\}.
$$

We must show that this defines an arithmetic tower in the sense of Definitions \ref{Gen_alg} and \ref{arith_tower_def}. Given $z\in\hat{\Gamma}_{n_1}(A)$ and using Pythagoras' theorem, along with the fact that $\tilde{\Gamma}_{n_2,n_1}(A)$ consists of finitely many squares in the complex plane aligned with the real and imaginary axes, we can compute $\mathrm{dist}(z,\tilde{\Gamma}_{n_2,n_1}(A))^2$ in finitely many arithmetic operations and comparisons. We can compute $(E(n_1,z)+1/{n_2})^2$ and check if this is less than $\mathrm{dist}(z,\tilde{\Gamma}_{n_2,n_1}(A))^2$. Hence $\zeta_{n_2,n_1}(A)$ can be computed with finitely many arithmetic operations and comparisons. There are now two cases to consider:

\vspace{2mm}

\textbf{Case 1:} $\mathrm{Sp}_d(A)\cap (\tilde{\Gamma}_{n_2}(A)+B_{1/n_2}(0))^c=\emptyset$. For large $n_1$, $\tilde{\Gamma}_{n_2}(A)=\tilde{\Gamma}_{n_2,n_1}(A)$ and this set contains the essential spectrum. It follows, for large $n_1$, since $E(n_1,z)\geq \mathrm{dist}(z,\tilde{\Gamma}_{n_2,n_1}(A))$ for all $z\in\hat{\Gamma}_{n_1}(A)$, that $\zeta_{n_2,n_1}(A)=\emptyset$.

\textbf{Case 2:} $\mathrm{Sp}_d(A)\cap (\tilde{\Gamma}_{n_2}(A)+B_{1/n_2}(0))^c\neq\emptyset$. In this case, this set is a finite subset of $\mathrm{Sp}_d(A)$, $\{\hat{z}_1,...,\hat{z}_{m(n_2)}\}$, separated from the closed set $\tilde{\Gamma}_{n_2}(A)+B_{1/n_2}(0)$ (we need the $+B_{1/n_2}(0)$ for this to be true to avoid accumulation points of the discrete spectrum). There exists some $\delta_{n_2}>0$ such that the balls $B_{2\delta_{n_2}}(\hat{z}_j)$ for $j=1,...,m(n_2)$ are pairwise disjoint and such that their union does note intersect $\tilde{\Gamma}_{n_2}(A)+B_{1/n_2}(0)$. Using the convergence of $\hat{\Gamma}_{n_1}(A)$ to $\mathrm{Sp}(A)$ and $E(n,z)\geq \mathrm{dist}(z,\mathrm{Sp}(A))$, it follows that for large $n_1$
\begin{equation}
\label{inside_balls}
\zeta_{n_2,n_1}(A)\subset \bigcup_{j=1}^{m(n_2)}B_{\delta_{n_2}}(\hat{z}_j),
\end{equation}
is non-empty and that $\zeta_{n_2,n_1}(A)$ converges to $\mathrm{Sp}_d(A)\cap (\tilde{\Gamma}_{n_2}(A)+B_{1/n_2}(0))^c\neq\emptyset$ in the Hausdorff metric.

\vspace{2mm}

Suppose that $\zeta_{n_2,n_1}(A)$ is non-empty. Recall that we only want one output per eigenvalue in the discrete spectrum. To do this, we partition the finite set $\zeta_{n_2,n_1}(A)$ into equivalence classes as follows. For $z,w\in\zeta_{n_2,n_1}(A)$, we say that $z\sim_{n_1}w$ if there exists a finite sequence $z=z_1,z_2,...,z_n=w\in\zeta_{n_2,n_1}(A)$ such that $B_{E({n_1},z_j)}(z_j)$ and $B_{E({n_1},z_{j+1})}(z_{j+1})$ intersect. The idea is that equivalence classes correspond to clusters of points in $\zeta_{n_2,n_1}(A)$. Given any $z\in\zeta_{n_2,n_1}(A)$ we can compute its equivalence class using finitely many arithmetic operations and comparisons. Let $S_0$ be the set $\{z\}$ and given $S_n$, let $S_{n+1}$ be the union of any $w\in \zeta_{n_2,n_1}(A)$ such that $B_{E({n_1},w)}(w)$ and $B_{E({n_1},v)}(v)$ intersect for some $v\in S_n$. Given $S_n$, we can compute $S_{n+1}$ using finitely many arithmetic operations and comparisons. The equivalence class is any $S_n$ where $S_n=S_{n+1}$, which must happen since $\zeta_{n_2,n_1}(A)$ is finite. We let $\Phi_{n_2,n_1}$ consist of one element of each equivalence class that minimises $E(n_1,\cdot)$ over its respective equivalence class. By the above comments it is clear that $\Phi_{n_2,n_1}$ can be computed in finitely many arithmetic operations and comparisons from the given data. Furthermore, due to (\ref{inside_balls}) which holds for large $n_1$, the separation of the $B_{2\delta_{n_2}}(\hat{z}_j)$ and the fact that $E(n_1,\cdot)$ converges uniformly on compact subsets to the distance to $\mathrm{Sp}(A)$, it follows that for large $n_1$ there is exactly one point in each intersection $B_{2\delta_{n_2}}(\hat{z}_j)\cap\Phi_{n_2,n_1}(A)$. But we can shrink $\delta_{n_2}$ and apply the same argument to see that $\Phi_{n_2,n_1}(A)$ converges to $\mathrm{Sp}_d(A)\cap (\tilde{\Gamma}_{n_2}(A)+B_{1/n_2}(0))^c\neq\emptyset$ in the Hausdorff metric.

Now suppose that $\zeta_{n_2,n_1}(A)$ is non-empty, $z_1,z_2\in\Phi_{n_2,n_1}(A)$, and both lie in $B_{\epsilon}(z)$ for some $z\in\mathrm{Sp}_d(A)$ and $\epsilon>0$ with $\mathrm{Sp}(A)\cap B_{2\epsilon}(z)=\{z\}$. It follows that $z$ minimises the distance to the spectrum from both $z_1$ and $z_2$. Hence, $B_{E({n_1},z_1)}(z_1)$ and $B_{E({n_1},z_2)}(z_2)$ both contain the point $z$ so that $z_1\sim_{n_1} z_2$. But then at most one of $z_1,z_2$ can lie in $\Phi_{n_2,n_1}(A)$, and hence $z_1=z_2$.

To finish, we must alter $\Phi_{n_2,n_1}(A)$ to take care of the case when $\zeta_{n_2,n_1}(A)=\emptyset$ and to produce a $\Sigma_2^A$ algorithm. In the case that $\zeta_{n_2,n_1}(A)=\emptyset$, set $\Phi_{n_2,n_1}(A)=\emptyset$. Let $N(A)\in\mathbb{N}$ be minimal such that $\mathrm{Sp}_d(A)\cap (\tilde{\Gamma}_{N}(A)+B_{1/N}(0))^c\neq\emptyset$ (recall the discrete spectrum is non-empty for our class of operators). If $n_2>n_1$, set $\Gamma_{n_2,n_1}(A)=\{0\}$, otherwise consider $\Phi_{k,n_1}(A)$ for $n_2\leq k\leq n_1$. If all of these are empty, set $\Gamma_{n_2,n_1}(A)=\{0\}$, otherwise choose minimal $k$ with $\Phi_{k,n_1}(A)\neq \emptyset$ and let $\Gamma_{n_2,n_1}(A)=\Phi_{k,n_1}(A)$. Note that this defines an arithmetic tower of algorithms, with $\Gamma_{n_2,n_1}(A)$ non-empty. By the above case analysis, for large $n_1$ it holds that
$$
\Gamma_{n_2,n_1}(A)=\Phi_{n_2\vee N(A),n_1}(A)
$$
and it follows that
\begin{equation}
\label{dfgviwh}
\lim_{n_1\rightarrow\infty}\Gamma_{n_2,n_1}(A)=:\Gamma_{n_2}(A)=\mathrm{Sp}_d(A)\cap (\tilde{\Gamma}_{n_2\vee N(A)}(A)+B_{1/{n_2\vee N(A)}}(0))^c.
\end{equation}
Hence $\Gamma_{n_2}(A)\subset\mathrm{Sp}_d(A)$ and $\Gamma_{n_2}(A)$ converges up to $\mathrm{cl}(\mathrm{Sp}_d(A))$ in the Hausdorff metric.

\textbf{Step 3}: \textbf{Multiplicities.} Suppose that $z_{n_2,n_1}\in \Gamma_{n_2,n_1}(A)$ converges as $n_1\rightarrow\infty$ to some $z_{n_2}=z\in\Gamma_{n_2}(A)\subset\mathrm{Sp}_d(A)$, where $\Gamma_{n_2}$ is the first limit of the height two tower constructed in step 2. Consider the following operator, viewed as a finite matrix acting on $\mathbb{C}^{n}$,
$
A_n=P_n(A-zI)^*(A-zI)P_n.
$
This is a truncation of the operator $(A-zI)^*(A-zI)$. The key observation is that $0$ lies in the discrete spectrum of $(A-zI)^*(A-zI)$ with $h((A-zI)^*(A-zI),0)=h(A,z)$, the multiplicity of the eigenvalue $z$. To see this, note that $\mathrm{ker}(A-zI)=\mathrm{ker}((A-zI)^*(A-zI))$ and that if $\left\|x\right\|=1$ then
$$
\left\|(A-zI)x\right\|\leq\sqrt{\left\|(A-zI)^*(A-zI)x\right\|}.
$$
Since $(A-zI)$ is bounded below on $\mathrm{ker}(A-zI)^{\perp}$, the same must be true for $(A-zI)^*(A-zI)$. Now set
$$
h_{n_2,n_1}(A,z_{n_2,n_1})=\min\{n_2,\left|\{w\in\mathrm{Sp}(P_{n_1}(A-z_{n_2,n_1}I)^*P_{f(n_1)}(A-z_{n_2,n_1}I)P_{n_1}):\left|w\right|<1/{n_2}-d_{n_1}\}\right|\},
$$
where $d_{n_1}$ is some non-negative sequence converging to $0$ that we define below. As usual, we consider the relevant operator as a matrix acting on $\mathbb{C}^{n_1}$ and we count eigenvalues according to their multiplicity. Via shifting by $(1/{n_2}-d_{n_1})I$ and assuming that $d_{n_1}$ can be computed with finitely many arithmetic operations and comparisons, Lemma \ref{num_neg_evals} shows that $h_{n_2,n_1}$ is a general algorithm and can be computed with finitely many arithmetic operations and comparisons. Consider the similar function (that we cannot necessarily compute since we do not know $z$),
$$
q_{n_2,n_1}(A,z)=\min\{n_2,\left|\{w\in\mathrm{Sp}(A_{n_1}):\left|w\right|<1/{n_2}\}\right|\},
$$
where
$$
A_{n_1}=P_{n_1}(A-zI)^*(A-zI)P_{n_1}.
$$
We set $B=(A-zI)^*(A-zI)$ and list $\lambda_1\leq\lambda_2\leq...$ as in Lemmas \ref{case1_fs} and \ref{case2_fs}, then
$$
\lim_{n_1\rightarrow\infty}q_{n_2,n_1}(A,z)=\min\{n_2,\left|\lambda_j:\lambda_j<1/n_2\right|\}.
$$
It is then clear from the same lemmas that
$$
\lim_{n_2\rightarrow\infty}\lim_{n_1\rightarrow\infty}q_{n_2,n_1}(A,z)=h((A-zI)^*(A-zI),0)=h(A,z).
$$
We will have completed the proof if we can choose $d_{n_1}$ such that
$$
\lim_{n_1\rightarrow\infty}\left|h_{n_2,n_1}(A,z_{n_2,n_1})-q_{n_2,n_1}(A,z)\right|=0.
$$
It is straightforward to show that
\begin{equation*}
\begin{split}
&\left\|A_{n_1}-P_{n_1}(A-z_{n_2,n_1}I)^*P_{f(n_1)}(A-z_{n_2,n_1}I)P_{n_1}\right\| \\
& \qquad \leq\big(\left|z-z_{n_2,n_1}\right|+c_{n_1}\big)\big(\left\|AP_{n_1}\right\|+\left|z-z_{n_2,n_1}\right|+\left|z_{n_2,n_1}\right|+\left\|P_{f({n_1})}(A-z_{n_2,n_1}I)P_{n_1}\right\|\big)\\
& \qquad \leq\big(\left|z-z_{n_2,n_1}\right|+c_{n_1}\big)\big(2\left\|P_{f({n_1})}AP_{n_1}\right\|+\left|z-z_{n_2,n_1}\right|+2\left|z_{n_2,n_1}\right|+c_{n_1}\big),
\end{split}
\end{equation*}
where $D_{f,m}(A) \leq c_m$ is the dispersion bound. Choose
$$
d_{n_1}=\big(E({n_1},z_{n_2,n_1})+c_{n_1}\big)\big(E({n_1},z_{n_2,n_1})+2\left|z_{n_2,n_1}\right|+2k_{n_1}+c_{n_1}\big),
$$
where $k_{n_1}$ overestimates $\left\|P_{f({n_1})}AP_{n_1}\right\|$ by at most $1$. $k_{n_1}$ can be computed using a similar positive definiteness test as in \texttt{DistSpec} (see Appendix \ref{append_pseudo}). Since $z_{n_2,n_1}$ converges to $z\in\mathrm{Sp}_d(A)$, it is clear that
$$
\left\|A_{n_1}-P_{n_1}(A-z_{n_2,n_1}I)^*P_{f(n_1)}(A-z_{n_2,n_1}I)P_{n_1}\right\|\leq d_{n_1}
$$
eventually and that $d_{n_1}$ converges to 0. Weyl's inequality for eigenvalue perturbations of Hermitian matrices implies the needed convergence. 
\end{proof}

\begin{proof}[Proof of Theorem \ref{discreteojoiojo} for $\Xi_2^d$]
Since $\Omega_{\mathrm{D}}\subset\Omega_{\mathrm{N}}^f$, its suffices to show that $\{\Xi_2^d,\Omega_{\mathrm{N}}^f\}\in\Sigma_2^A$ and $\{\Xi_2^d,\Omega_{\mathrm{D}}\}\notin\Delta_2^G$.\\
\textbf{Step 1}: $\{\Xi_2^d,\Omega_{\mathrm{D}}\}\notin\Delta_2^G$. The proof is almost identical to step 1 in the proof of Theorem \ref{discreteojoiojo} for $\Xi_1^d$. Suppose there exists some height one tower $\Gamma_n$ solving the problem. Consider the matrix operators $A_m=\mathrm{diag}\{0,0,...,0,2\}\in\mathbb{C}^{m\times m}$ and $C=\mathrm{diag}\{0,0,...\}$ and set
$$
A=\bigoplus_{m=1}^{\infty}A_{k_m},
$$
where we choose an increasing sequence $k_m$ inductively as follows. Set $k_1=1$ and suppose that $k_1,...,k_m$ have been chosen. $\mathrm{Sp}_{d}(A_{k_1}\oplus A_{k_2}\oplus...\oplus A_{k_m}\oplus C)=\{2\}$ so there exists some $n_m\geq m$ such that if $n\geq n_m$ then
$$
\Gamma_n(A_{k_1}\oplus...\oplus A_{k_m}\oplus C)=1.
$$
Now let $k_{m+1}\geq \max\{N(\mathrm{diag}\{1,2\}\oplus A_{k_1}\oplus...\oplus A_{k_m}\oplus C,n_m),k_m+1\}$. Arguing as in the proof of Theorem \ref{unbounded_test_theorem}, it follows that $\Gamma_{n_m}(A)=\Gamma_{n_m}(A_{k_1}\oplus...\oplus A_{k_m}\oplus C)$. But $\Gamma_{n_m}(A)$ converges to $0$ as $A$ has no discrete spectrum and this contradiction finishes this step.\\
\textbf{Step 2}: $\{\Xi_2^d,\Omega_{\mathrm{N}}^f\}\in\Sigma_2^A$. Consider the height two tower, $\zeta_{n_2,n_1}$, defined in step 2 of the proof of Theorem \ref{discreteojoiojo} for $\Xi_1^d$. Let $A\in\Omega_{\mathrm{N}}^f$ and if $\zeta_{n_2,n_1}(A)=\emptyset$, define $\rho_{n_2,n_1}(A)=0$, otherwise define $\rho_{n_2,n_1}(A)=1$. The discussion in the proof of Theorem \ref{discreteojoiojo} for $\Xi_1^d$ shows that
$$
\lim_{n_1\rightarrow\infty}\rho_{n_2,n_1}(A)=:\rho_{n_2}(A)=\begin{cases}
0, \quad\text{ if }\mathrm{Sp}_d(A)\cap (\tilde{\Gamma}_{n_2}(A)+B_{1/n_2}(0))^c=\emptyset\\
1, \quad\text{ otherwise.}
\end{cases}
$$
Since $\mathrm{Sp}_d(A)\cap (\tilde{\Gamma}_{n_2}(A)+B_{1/n_2}(0))^c$ increases to $\mathrm{cl}(\mathrm{Sp}_d(A))$, it follows that $\lim_{n_2\rightarrow\infty}\rho_{n_2}(A)=\Xi_2^d(A)$ and that if $\rho_{n_2}(A)=1$, $\Xi_2^d(A)=1$. Hence, $\rho_{n_2,n_1}$ provides a $\Sigma_2^A$ tower for $\{\Xi_2^d,\Omega_{\mathrm{N}}^f\}$.
\end{proof}

\begin{proof}[Proof of Theorem \ref{evector_theorem}]
Let $\epsilon=(E({n_1},z_{n_1})+\delta)^2$ and consider the matrix
\begin{equation*}
B=P_{n_1}(A-z_{n_1}I)^*P_{f(n_1)}(A-z_{n_1}I)P_{n_1}-\epsilon I_n\in\mathbb{C}^{n\times n},
\end{equation*}
where $I_n$ is the $n\times n$ identity matrix. $B$ is a Hermitian matrix and is not positive semi-definite. It follows that $B$ can be put into the form
$$
PBP^T=LDL^*,
$$
where $L$ is lower triangular with $1$'s along its diagonal, $D$ is block diagonal with block sizes $1\times 1$ or $2\times 2$ and $P$ is a permutation matrix. This can be computed in finitely many arithmetic operations and comparisons. Let $x$ be an eigenvector of $B$ with negative eigenvalue and set $y=L^*Px$. Such an $x$ exists by assumption. Note that
\begin{equation*}
\langle y,Dy\rangle=\langle L^*Px,DL^*Px\rangle=\langle x,Bx\rangle<0.
\end{equation*}
It follows that there exists a unit vector $y_{n_1}$ with $\langle y_{n_1},Dy_{n_1}\rangle<0$. Such a vector is easy to spot if a value in one of the $1\times 1$ blocks of $D$ is negative. If not, we need to consider $2\times 2$ blocks. Using the argument in the proof of Lemma \ref{num_neg_evals}, we can find a $2\times 2$ block with a negative eigenvalue by computing the trace and determinant. Without loss of generality we assume that this block is the upper $2\times 2$ portion of $D$. It follows that there exist \textit{real} numbers $a,b$, not both equal to $0$, such that $y_{n_1}=(a,b,0,...,0)^{T}$ has $\langle y_{n_1},Dy_{n_1}\rangle<0$. If $D_{2,2}<0$, we can take $a=0,b=1$. Otherwise, $a\neq 0$ so set $a=1$. We then note that there is an open interval $J$ such that if $b\in J$ then $y_{n_1}=(a,b,0,...,0)^{T}$ has $\langle y_{n_1},Dy_{n_1}\rangle<0$. We can now perform a search routine on $\mathbb{R}$ with finer and finer spacing to find such a $b$.

Since $L^*$ is invertible and upper triangular, we can efficiently solve for $\tilde{x}_{n_1}=P^T(L^*)^{-1}y_{n_1}$ using finitely many arithmetic operations and comparisons. We then approximately normalise $\tilde{x}_{n_1}$ by computing
$\left\|\tilde{x}_{n_1}\right\|\approx t_{n_1}(\rho)>0$ to precision $\rho>0$ using arithmetic operations and comparisons. If we set $x_{n_1}=\tilde{x}_{n_1}/t_{n_1}(\rho)$ then
$$
1-\frac{\rho}{t_{n_1}(\rho)}=\frac{t_{n_1}(\rho)-\rho}{t_{n_1}(\rho)}\leq\left\|x_{n_1}\right\|\leq \frac{t_{n_1}(\rho)+\rho}{t_{n_1}(\rho)}=1+\frac{\rho}{t_{n_1}(\rho)}.
$$
So we successively choose $\rho$ smaller until we reach $\rho_{n_1}$ such that $\rho_{n_1}/t_{n_1}(\rho_{n_1})<\delta$. This is always possible since $\lim_{\rho\downarrow0}t_{n_1}(\rho)=\left\|\tilde{x}_{n_1}\right\|>0$. Let $t_{n_1}=t_{n_1}(\rho_{n_1})$, then
$$
\langle x_{n_1},Bx_{n_1}\rangle=t_{n_1}^{-2}\langle L^*P\tilde{x}_{n_1},DL^*P\tilde{x}_{n_1}\rangle=t_{n_1}^{-2}\langle y_{n_1},Dy_{n_1}\rangle<0.
$$
Note that
\begin{align*}
\left\|P_{f({n_1})}(A-z_{n_1}I)x_{n_1}\right\|^2&=\langle x_{n_1},Bx_{n_1}\rangle+\left\|x_{n_1}\right\|^2\epsilon<\left\|x_{n_1}\right\|^2\epsilon.
\end{align*}
Taking square roots and recalling that $D_{f,n_1}(A)\leq c_{n_1}$ and the definition of $D_{f,n_1}$ finishes the proof.
\end{proof}

Note that even in the finite-dimensional case, this type of error control is the best possible owing to numerical errors due to round-off and finite precision. This method is also efficient.

\begin{proof}[Proof of Theorem \ref{discrete_dont_now_f}]
{\bf Step 1:} $\{\Xi_1^d,\Omega_1^d\}\notin\Delta_3^G$. Suppose for a contradiction that $\Gamma_{n_2,n_1}$ is a height two tower solving this problem. For this proof we shall use one of the decision problems in \cite{colbrook4} that were proven to have $\mathrm{SCI}_G=3$. Let $(\mathcal{M},d)$ be the discrete space $\{0,1\}$, let $\Omega'$ denote the collection of all infinite matrices $\{a_{i,j}\}_{i,j\in\mathbb{N}}$ with entries $a_{i,j}\in\{0,1\}$ and consider the problem function
\begin{equation*}
\Xi'(\{a_{i,j}\}):\text{ ``Does }\{a_{i,j}\}\text{ have only finitely many columns containing only finitely may non-zero entries?''}
\end{equation*}
We will gain a contradiction by using the supposed height two tower for $\{\Xi_1^d,\Omega_1^d\}$, $\Gamma_{n_2,n_1}$, to solve $\{\Xi',\Omega'\}$.

Without loss of generality, identify $\mathcal{B}(l^2(\mathbb{N}))$ with $\mathcal{B}(X)$ where $X=\mathbb{C}^2\oplus\bigoplus_{j=1}^{\infty}X_j$ in the $l^2$-sense with $X_j=l^2(\mathbb{N})$. Now let $\{a_{i,j}\}\in\Omega'$ and for the $j$th column define $B_j\in\mathcal{B}(X_j)$ with the following matrix representation:
$$
B_j=\bigoplus_{r=1}^{M_j} A_{l_r^j},\quad A_{m}:=\begin{pmatrix}
1& & & &1\\
 &0& & & \\
& &\ddots& & \\
 & & &0& \\
1& & & &1\\
\end{pmatrix}
\in\mathbb{C}^{m\times m},
$$
where if $M_j$ is finite then $l_{M_j}^j=\infty$ with $A_{\infty}=\mathrm{diag}(1,0,0,...)$. The $l_r^j$ are defined such that
\begin{equation}
\label{crazy_indices}
\sum_{r=1}^{\sum_{i=1}^m a_{i,j}} l_r^j=m+\sum_{i=1}^m a_{i,j}.
\end{equation}
Define the self-adjoint operator
$$
A=\mathrm{diag}\{3,1\}\oplus\bigoplus_{j=1}^{\infty}B_j.
$$
Note that no matter what the choices of $l_r^j$ are, $3\in\mathrm{Sp}_d(A)$ and hence $A\in\Omega_1^d$. Note also that the spectrum of $A$ is contained in $\{0,1,2,3\}$. If $\Xi'(\{a_{i,j}\})=1$, $1$ is an isolated eigenvalue of finite multiplicity and hence in $\mathrm{Sp}_d(A)$. But if $\Xi'(\{a_{i,j}\})=0$, then $1$ is an isolated eigenvalue of infinite multiplicity so does not lie in the discrete spectrum and hence $\mathrm{Sp}_d(A)\subset\{0,2,3\}$.

Consider the intervals
$
J_1=[0,1/2],
$ and 
$
J_2=[3/4,\infty).
$
Set $\alpha_{n_2,n_1}=\mathrm{dist}(1,\Gamma_{n_2,n_1}(A))$. Let $k(n_2,n_1)\leq n_1$ be maximal such that $\alpha_{n_2,k}(A)\in J_1\cup J_2$. If no such $k$ exists or $\alpha_{n_2,k}(A)\in J_1$, set $\tilde{\Gamma}_{n_2,n_1}(\{a_{i,j}\})=1$. Otherwise set $\tilde{\Gamma}_{n_2,n_1}(\{a_{i,j}\})=0$. It is clear from (\ref{crazy_indices}) that this defines a generalised algorithm. In particular, given $N$ we can evaluate $\{A_{k,l}:k,l\leq N\}$ using only finitely many evaluations of $\{a_{i,j}\}$, where we can use a suitable bijection between bases of $l^2(\mathbb{N})$ and $\mathbb{C}^2\oplus\bigoplus_{j=1}^{\infty}X_j$ to view $A$ as acting on $l^2(\mathbb{N})$. The point of the intervals $J_1,J_2$ is that we can show $\lim_{n_1\rightarrow\infty}\tilde{\Gamma}_{n_2,n_1}(\{a_{i,j}\})=\tilde{\Gamma}_{n_2}(\{a_{i,j}\})$ exists. If $\Xi'(\{a_{i,j}\})=1$, then, for large $n_2$, $\lim_{n_1\rightarrow\infty}\alpha_{n_2,k}(A)<1/2$ and hence $\lim_{n_2\rightarrow\infty}\tilde{\Gamma}_{n_2}(\{a_{i,j}\})=1$. Similarly, if $\Xi'(\{a_{i,j}\})=0$, then, for large $n_2$, $\lim_{n_1\rightarrow\infty}\alpha_{n_2,k}(A)>3/4$ and hence it follows that $\lim_{n_2\rightarrow\infty}\tilde{\Gamma}_{n_2}(\{a_{i,j}\})=0$. Hence $\tilde{\Gamma}_{n_2,n_1}$ is a height two tower of general algorithms solving $\{\Xi',\Omega'\}$, a contradiction.

{\bf Step 2:} $\{\Xi_2^d,\Omega_2^d\}\notin\Delta_3^G$. To prove this we can use a slight alteration of the argument in step 1. Replace $X$ by $X=l^2(\mathbb{N})\oplus\bigoplus_{j=1}^{\infty}X_j$ and $A$ by
$$
A=\mathrm{diag}\{1,0,2,0,2,...\}\oplus\bigoplus_{j=1}^{\infty}B_j.
$$
It is then clear that $\Xi_2^d(A)=1$ if and only if $\Xi'(\{a_{i,j}\})=1$.

{\bf Step 3:} $\{\Xi_1^d,\Omega_1^d\}\in\Sigma_3^A$. For this we argue similarly to the proof of Theorem\ref{discreteojoiojo} for $\Xi_1^d$ step 2. It was shown in \cite{ben2015can} that there exists a height three arithmetic tower $\tilde{\Gamma}_{n_3,n_2,n_1}$ for the essential spectrum of operators in $\Omega_1^d$ such that
\begin{itemize}
	\item Each $\tilde\Gamma_{n_3,n_2,n_1}(A)$ consists of a finite collection of points in the complex plane.
	\item For large $n_1$, $\tilde\Gamma_{n_3,n_2,n_1}(A)$ is eventually constant and equal to $\tilde\Gamma_{n_3,n_2}(A)$.
	\item $\tilde\Gamma_{n_3,n_2}(A)$ is increasing with $n_2$ with limit $\tilde\Gamma_{n_3}(A)$ containing the essential spectrum. The limit $\tilde\Gamma_{n_3}(A)$ is also decreasing with $n_3$.
\end{itemize}
Furthermore, it was proven in \cite{ben2015can} that for operators in $\Omega_1^d$, there exists a height two arithmetic tower $\hat\Gamma_{n_2,n_1}$ for computing the spectrum such that
\begin{itemize}
	\item $\hat\Gamma_{n_2,n_1}(A)$ is constant for large $n_1$.
	\item For any $z\in\hat\Gamma_{n_2}(A)$, $\mathrm{dist}(z,\mathrm{Sp}(A))\leq 2^{-n_2}$.
\end{itemize}

Using these, we initially define
$$
\zeta_{n_3,n_2,n_1}(A)=\{z\in\hat{\Gamma}_{n_2,n_1}(A):2^{-n_3}-2^{-n_2}\leq\mathrm{dist}(z,\tilde{\Gamma}_{n_3,n_2,n_1}(A))\}.
$$
The arguments in the proof of Theorem \ref{discreteojoiojo} for $\Xi_1^d$ show that this can be computed in finitely many arithmetic operations and comparisons using the relevant evaluation functions. Note that for large $n_1$
$$
\zeta_{n_3,n_2,n_1}(A)=\{z\in\hat{\Gamma}_{n_2}(A):2^{-n_3}-2^{-n_2}\leq\mathrm{dist}(z,\tilde{\Gamma}_{n_3,n_2}(A))\}=:\zeta_{n_3,n_2}(A).
$$

There are now two cases to consider (we use $D_{\eta}(z)$ to denote the open ball of radius $\eta$ about a point $z$):

\textbf{Case 1:} $\mathrm{Sp}_d(A)\cap (\tilde{\Gamma}_{n_3}(A)+D_{2^{-n_3}}(0))^c=\emptyset$. Suppose, for a contradiction, in this case that there exists $z_{m_j}\in\zeta_{n_3,m_j}(A)$ with ${m_j}\rightarrow\infty$. Then, without loss of generality, $z_{m_j}\rightarrow z\in\mathrm{Sp}(A)$. We also have that
$$
\mathrm{dist}(z_{m_j},\tilde{\Gamma}_{n_3,m_j}(A))\geq 2^{-n_3}-2^{-m_j},
$$
which implies that $\mathrm{dist}(z,\tilde{\Gamma}_{n_3}(A))\geq 2^{-n_3}$ and hence $z\in\mathrm{Sp}_d(A)\cap (\tilde{\Gamma}_{n_3}(A)+D_{2^{-n_3}}(0))^c$, the required contradiction. It follows that $\zeta_{n_3,n_2}(A)$ is empty for large $n_2$.

\textbf{Case 2:} $\mathrm{Sp}_d(A)\cap (\tilde{\Gamma}_{n_3}(A)+D_{2^{-n_3}}(0))^c\neq\emptyset$. In this case, this set is a finite subset of $\mathrm{Sp}_d(A)$, $\{\hat{z}_1,...,\hat{z}_{m(n_3)}\}$. Each of these points is an isolated point of the spectrum. It follows that there exists $z_{n_2}\in\hat{\Gamma}_{n_2}(A)$ with $z_{n_2}\rightarrow\hat{z}_1$ and $\left|z_{n_2}-\hat{z}_1\right|\leq 2^{-n_2}$ for large $n_2$. Since the $\tilde \Gamma_{n_3,n_2}(A)$ are increasing, this implies that
\begin{align*}
\mathrm{dist}(z_{n_2},\tilde\Gamma_{n_3,n_2}(A))&\geq \mathrm{dist}(z_{n_2},\tilde\Gamma_{n_3}(A))\\
&\geq \mathrm{dist}(\hat{z}_1,\tilde\Gamma_{n_3}(A))-2^{-n_2}\geq 2^{-n_3}-2^{-n_2},
\end{align*}
so that $z_{n_2}\in\zeta_{n_3,n_2}(A)$. The same argument holds for points converging to all of $\{\hat{z}_1,...,\hat{z}_{m(n_3)}\}$. On the other hand, the argument used in Case 1 shows that any limit points of $\zeta_{n_3,n_2}(A)$ as $n_2\rightarrow\infty$ are contained in $\mathrm{Sp}_d(A)\cap (\tilde{\Gamma}_{n_3}(A)+D_{2^{-n_3}}(0))^c$. It follows that, in this case, $\zeta_{n_3,n_2}(A)$ converges to $\mathrm{Sp}_d(A)\cap (\tilde{\Gamma}_{n_3}(A)+B_{1/n_3}(0))^c\neq\emptyset$ in the Hausdorff metric as $n_2\rightarrow\infty$.

\vspace{2mm}

Let $N(A)\in\mathbb{N}$ be minimal such that $\mathrm{Sp}_d(A)\cap (\tilde{\Gamma}_{N}(A)+D_{2^{-N}}(0))^c\neq\emptyset$ (recall the discrete spectrum is non-empty for our class of operators). If $n_3>n_2$, set $\Gamma_{n_3,n_2,n_1}(A)=\{0\}$, otherwise consider $\zeta_{k,n_2,n_1}(A)$ for $n_3\leq k\leq n_2$. If all of these are empty, set $\Gamma_{n_3,n_2,n_1}(A)=\{0\}$, otherwise choose minimal $k$ with $\zeta_{k,n_2,n_1}(A)\neq \emptyset$ and let $\Gamma_{n_3,n_2,n_1}(A)=\zeta_{k,n_2,n_1}(A)$. Note that this defines an arithmetic tower of algorithms, with $\Gamma_{n_3,n_2,n_1}(A)$ non-empty. Since we consider finitely many of the sets $\zeta_{k,n_2,n_1}(A)$, and these are constant for large $n_1$, it follows that $\Gamma_{n_3,n_2,n_1}(A)$ is constant for large $n_1$ and constructed in the same manner with replacing $\zeta_{k,n_2,n_1}(A)$ by $\zeta_{k,n_2}(A)$. Call this limit $\Gamma_{n_3,n_2}(A)$.

For large $n_2$,
$$
\Gamma_{n_3,n_2}(A)=\zeta_{n_3\vee N(A),n_2}(A).
$$
It follows that
$$
\lim_{n_2\rightarrow\infty}\Gamma_{n_3,n_2}(A)=:\Gamma_{n_3}(A)=\mathrm{Sp}_d(A)\cap (\tilde{\Gamma}_{n_3\vee N(A)}(A)+D_{2^{-n_3\vee N(A)}}(0))^c.
$$
Hence $\Gamma_{n_3}(A)\subset\mathrm{Sp}_d(A)$ and $\Gamma_{n_3}(A)$ converges up to $\mathrm{cl}(\mathrm{Sp}_d(A))$ in the Hausdorff metric.

{\bf Step 4:} $\{\Xi_2^d,\Omega_2^d\}\in\Sigma_3^A$.
Consider the height three tower, $\zeta_{n_3,n_2,n_1}$, defined in step 3. Let $A\in\Omega_2^d$ and if $\zeta_{n_3,n_2,n_1}(A)=\emptyset$, define $\rho_{n_3,n_2,n_1}(A)=0$, otherwise define $\rho_{n_3,n_2,n_1}(A)=1$. The discussion in step 3 shows that
$$
\lim_{n_2\rightarrow\infty}\lim_{n_1\rightarrow\infty}\rho_{n_3,n_2,n_1}(A)=:\rho_{n_3}(A)=\begin{cases}
0, \quad\text{ if }\mathrm{Sp}_d(A)\cap (\tilde{\Gamma}_{n_3}(A)+D_{2^{-n_3}}(0))^c=\emptyset\\
1, \quad\text{ otherwise.}
\end{cases}
$$
Since $\mathrm{Sp}_d(A)\cap (\tilde{\Gamma}_{n_3}(A)+D_{2^{-n_3}}(0))^c$ increases to $\mathrm{cl}(\mathrm{Sp}_d(A))$, it follows that $\lim_{n_3\rightarrow\infty}\rho_{n_3}(A)=\Xi_2^d(A)$ and that if $\rho_{n_3}(A)=1$, then $\Xi_2^d(A)=1$. Hence, $\rho_{n_3,n_2,n_1}$ provides a $\Sigma_3^A$ tower for $\{\Xi_2^d,\Omega_2^d\}$.
\end{proof}

\section{Proof of Theorem on the Spectral Gap and Spectral Classification }
\label{Sec:proof_spec_gap}

\begin{proof}[Proof of Theorem \ref{class_thdfjlwdjfkl} for $\Xi_{\mathrm{gap}}$]
\textbf{Step 1}: $\{\Xi_{\mathrm{gap}},\widehat{\Omega}_{\mathrm{SA}}\}\in\Sigma^A_2$. Let $A\in\widehat{\Omega}_{\mathrm{SA}}$. Using Corollary \ref{eigs_fin_mat} we can compute all $n$ eigenvalues of $P_nAP_n$ to arbitrary precision in finitely many arithmetic operations and comparisons. In the notation of Lemmas \ref{case1_fs} and \ref{case2_fs} (whose analogous results also hold for the possibly unbounded $A\in\widehat{\Omega}_{\mathrm{SA}}$), consider an approximation
\begin{equation*}
0\leq l_n:=\mu_2^{(n)}-\mu_1^{(n)}+\epsilon_n,\quad n\geq 2,
\end{equation*}
where we have computed $\mu_2^{(n)}-\mu_1^{(n)}$ to accuracy $\left|\epsilon_n\right|\leq1/n$ using Corollary \ref{eigs_fin_mat} with $B=P_nAP_n$. Using Lemmas \ref{case1_fs} and \ref{case2_fs}, we see that $l_n$ converges to zero if and only if $\Xi_{\mathrm{gap}}(A)=0$, otherwise it converges to some positive number. If $n_1=1$, set $\Gamma_{n_2,n_1}(A)=1$, otherwise consider the following.

Let $J_{n_2}^1=[0,1/(2n_2)]$ and $J_{n_2}^2=(1/n_2,\infty)$. Given $n_1\in\mathbb{N}$, consider $l_k$ for $k\leq n_1$. If no such $k$ exists with $l_k\in J_{n_2}^1\cup J_{n_2}^2$, set $\Gamma_{n_2,n_1}(A)=0$. Otherwise, consider $k$ maximal with $l_k\in J_{n_2}^1\cup J_{n_2}^2$ and set $\Gamma_{n_2,n_1}(A)=0$ if $l_k\in J_{n_2}^1$ and $\Gamma_{n_2,n_1}(A)=1$ if $l_k \in J_{n_2}^2$. The sequence $l_{n_1}\rightarrow c\geq 0$ for some number $c$. The separation of the intervals $J_{n_2}^1$ and $J_{n_2}^2$, ensures that $l_{n_1}$ cannot be in both intervals infinitely often as $n_1\rightarrow\infty$ and hence the first limit $\Gamma_{n_2}(A):=\lim_{n_1\rightarrow\infty}\Gamma_{n_2,n_1}(A)$ exists. If $c=0$, $\Gamma_{n_2}(A)=0$, but if $c>0$, there exists $n_2$ with $1/n_2<c$ and hence for large $n_1$, $l_{n_1}\in J_{n_2}^2$. It follows in this case that $\Gamma_{n_2}(A)=1$ and we also see that if $\Gamma_{n_2}(A)=1$ then $\Xi_{\mathrm{gap}}(A)=1$. Hence $\Gamma_{n_2,n_1}$ provides a $\Sigma_2^A$ tower.

\textbf{Step 2}: $\{\Xi_{\mathrm{gap}},\widehat{\Omega}_{\mathrm{D}}\}\notin\Delta^G_2$. We argue by contradiction and assume the existence of a height one tower, $\Gamma_n$ converging to $\Xi_{\mathrm{gap}}$. The method of proof follows the same lines as before. For every $A$ and $n$ there exists a finite number $N(A,n)\in\mathbb{N}$ such that the evaluations from $\Lambda_{\Gamma_n}(A)$ only take the matrix entries $A_{ij} = \left\langle Ae_j , e_i\right\rangle$ with $i,j\leq N(A,n)$ into account. List the rationals in $(0,1)$ without repetition as $d_1,d_2,...$. We consider the operators $A_m=\mathrm{diag}\{d_1,d_2,...,d_m\}\in\mathbb{C}^{m\times m}$, $B_m=\mathrm{diag}\{1,1,...,1\}\in\mathbb{C}^{m\times m}$ and $C=\mathrm{diag}\{1,1,...\}$. Let 
\[
A=\bigoplus_{m=1}^{\infty}(B_{k_m}\oplus A_{k_m}),
\] where we choose an increasing sequence $k_m$ inductively as follows. In what follows, all operators considered are easily seen to be in $\widehat{\Omega}_{\mathrm{D}}$.

Set $k_1=1$ and suppose that $k_1,...,k_m$ have been chosen. Define $\zeta_p:=\min\{d_r:1\leq r\leq k_p\}$. $\mathrm{Sp}(B_{k_1}\oplus A_{k_1}\oplus...\oplus B_{k_m}\oplus A_{k_m}\oplus C)=\{d_1,d_2,...,d_m,1\}$ has $\zeta_m$ the minimum of its spectrum and an isolated eigenvalue of multiplicity $1$, hence
$$\Xi(B_{k_1}\oplus A_{k_1}\oplus...\oplus B_{k_m}\oplus A_{k_m}\oplus C)=\text{``Yes''}.$$
It follows that there exists some $n_m\geq m$ such that if $n\geq n_m$, then
\begin{equation*}
\Gamma_n(B_{k_1}\oplus A_{k_1}\oplus...\oplus B_{k_m}\oplus A_{k_m}\oplus C)=\text{``Yes''}.
\end{equation*}
Now let $k_{m+1}\geq \max\{N(B_{k_1}\oplus A_{k_1}\oplus...\oplus B_{k_m}\oplus A_{k_m}\oplus C,n_m),k_m+1\}$. The same argument used in the proof of Theorem \ref{unbounded_test_theorem} shows that $\Gamma_{n_m}(A)=\Gamma_{n_m}(B_{k_1}\oplus A_{k_1}\oplus...\oplus B_{k_m}\oplus A_{k_m}\oplus C)=\text{``Yes''}$. But $\mathrm{Sp}(A)=[0,1]$ is gappless, and so $\lim_{n\rightarrow\infty}(\Gamma_n(A))=\text{``No''}$, a contradiction.
\end{proof}

\begin{proof}[Proof of Theorem \ref{class_thdfjlwdjfkl} for $\Xi_{\mathrm{class}}$] By composing with the map
$$
\rho:\{1,2,3,4\}\rightarrow\{0,1\},
$$
$\rho(1)=1$, $\rho(2)=\rho(3)=\rho(4)=0$, it is clear that the result for $\Xi_{\mathrm{gap}}$ implies $\{\Xi_{\mathrm{class}}, \widehat{\Omega}_{\mathrm{SA}}^f\},\{\Xi_{\mathrm{class}}, \widehat{\Omega}_{\mathrm{D}}\}\notin\Delta^G_2$. Since $\widehat{\Omega}_{\mathrm{D}}\subset\widehat{\Omega}_{\mathrm{SA}}^f$, we need only construct a $\Pi^A_2$ tower for $\{\Xi_{\mathrm{class}}, \widehat{\Omega}_{\mathrm{SA}}^f\}$.

Let $A\in \widehat{\Omega}_{\mathrm{SA}}^f$. For a given $n$, set $B_{n}=P_{n}AP_{n}$ and in the notation of Lemmas \ref{case2_fs} and \ref{case1_fs}, let
$$
0\leq l_{n}^j:=\mu_{j+1}^{(n)}-\mu_{1}^{(n)}+\epsilon_{n}^j, \text{ for }j<n.
$$
where we again have computed $\mu_{j+1}^{(n)}-\mu_{1}^{(n)}$ to accuracy $\left|\epsilon_{n}^j\right|\leq1/n$ using only finitely many arithmetic operations and comparisons by Corollary \ref{eigs_fin_mat}. $\Xi_{\mathrm{class}}(A)=1$ if and only if $l_{n}^1$ converges to a positive constant as $n\rightarrow\infty$. $\Xi_{\mathrm{class}}(A)=2$ if and only if $l_{n}^1$ converges to zero as $n\rightarrow\infty$ and there exists $j$ with $l_{n}^j$ convergent to a positive constant. 

Note that we can use the algorithm presented in \S \ref{pf_unb_gr}, denoted $\hat\Gamma_{n}$, to compute the spectrum, with error function denoted by $E(n,\cdot)$ converging uniformly on compact subsets of $\mathbb{C}$ to the true error from above (again with the choice of $g(x)=x$ since the operator is normal). Setting
$$
a_n(A)=\min_{x\in\hat\Gamma_{n}(A)}\{x+E(n,x)\},
$$
we see that $a_n(A)\geq a(A):=\inf_{x\in\mathrm{Sp}(A)}\{x\}$ and that $a_n(A)\rightarrow a(A)$. Now consider
$$
b_{n_2,n_1}(A)=\min\{E(k,a_k(A)+1/n_2)+1/k:1\leq k\leq n_1\},
$$
then $b_{n_2,n_1}(A)$ is positive and decreasing in $n_1$, so converges to some limit $b_{n_2}(A)$ as $n_1\rightarrow\infty$.

\begin{lemma}
\label{cs_and_bs}
Let $A\in\widehat{\Omega}_{\mathrm{SA}}^f$ and $c_{n_2,n_1}(A)=E(n_1,a_{n_1}(A)+1/n_2)+1/n_1$, then
$$
\lim_{n_1\rightarrow\infty}c_{n_2,n_1}(A)=:c_{n_2}(A)= \mathrm{dist}(a+1/n_2,\mathrm{Sp}(A)).
$$
Furthermore, if $\Xi_{\mathrm{class}}(A)\neq 4$ then for large $n_2$ it follows that $c_{n_2}(A)=b_{n_2}(A)=1/n_2$.
\end{lemma}
\begin{proof}[Proof of Lemma \ref{cs_and_bs}]
We know that $a_{n_1}(A)+1/n_2$ converges to $a(A)+1/n_2$ as $n_1\rightarrow\infty$. Furthermore, $\mathrm{dist}(z,\mathrm{Sp}(A))$ is continuous in $z$ and $E(n_1,z)$ converges uniformly to $\mathrm{dist}(z,\mathrm{Sp}(A))$ on compact subsets of $\mathbb{C}$. Hence, the limit $c_{n_2}(A)$ exists and is equal to $\mathrm{dist}(a(A)+1/n_2,\mathrm{Sp}(A))$. It is clear that $b_{n_2}(A)\leq c_{n_2}(A)$. Suppose now that $\Xi_{\mathrm{class}}(A)\neq 4$, then for large $n_1$, say bigger than some $N$, and for large enough $n_2$,
\begin{align*}
E(n_1,a_{n_1}(A)+1/n_2)&\geq \mathrm{dist}(a_{n_1}(A)+1/n_2,\mathrm{Sp}(A))\\ &=\left|a_{n_1}(A)+1/n_2-a(A)\right|\\&\geq 1/n_2=\mathrm{dist}(a(A)+1/{n_2},\mathrm{Sp}(A))=c_{n_2}(A).
\end{align*}
Now choose $n_2$ large such that the above inequality holds and $1/n_2\leq 1/N$. Then $b_{n_2,n_1}(A)\geq 1/n_2$. Taking limits finishes the proof.
\end{proof}

If $n_2\geq n_1$, set $\Gamma_{n_2,n_1}(A)=1$. Otherwise, for $1\leq j \leq n_2$, let $k^j_{n_2,n_1}$ be maximal with $1\leq k_{n_2,n_1}^j<n_1$ such that $l^j_{k^j_{n_2,n_1}}\in J_{n_2}^1\cup J_{n_2}^2$ if such $k^j_{n_2,n_1}$ exist, where $J_{n_2}^1$ and $J_{n_2}^2$ are as in the proof for $\Xi_{\mathrm{gap}}$. If $k^1_{n_2,n_1}$ exists with $l^1_{k^1_{n_2,n_1}}\in J_{n_2}^2$, set $\Gamma_{n_2,n_1}(A)=1$. Otherwise, if any of $k^m_{n_2,n_1}$ exists with $l^m_{k^m_{n_2,n_1}}\in J_{n_2}^2$ for $2\leq m \leq n_2$, set $\Gamma_{n_2,n_1}(A)=2$. Suppose that neither of these two cases hold. In this case, compute $b_{n_2,n_1}(A)$. If $b_{n_2,n_1}(A)\geq 1/{n_2}$, set $\Gamma_{n_2,n_1}(A)=3$, otherwise set $\Gamma_{n_2,n_1}(A)=4$. We now must show this provides a $\Pi^A_2$ tower solving our problem.

First we show convergence of the first limit. Fix $n_2$ and consider large $n_1$. The separation of the intervals $J_{n_2}^1$ and $J_{n_2}^2$ ensures that each sequence $\{l^j_{n}\}_{n\in\mathbb{N}}$ cannot visit each interval infinitely often. Since $b_{n_1,n_2}(A)$ is non-increasing in $n_1$, we also see that the question whether $b_{n_2,n_1}(A)\geq 1/{n_2}$ eventually has a constant answer. These observations ensure convergence of the first limit $\Gamma_{n_2}(A)=\lim_{n_1\rightarrow\infty}\Gamma_{n_2,n_1}(A)$.

If $\Xi_{\mathrm{class}}(A)=1$, then for large $n_2$, $l^1_{n_1}$ must eventually be in $J_{n_2}^2$ and hence
$\Gamma_{n_2}(A)=1$. It is also clear that if $\Gamma_{n_2}(A)=1$, then $l^1_{n_1}$ converges to a positive constant, which implies $\Xi_{\mathrm{class}}(A)=1$. If $\Xi_{\mathrm{class}}(A)=2$, then for large $n_2$, $l^m_{n_1}$ eventually lies in $J_{n_2}^2$ for some $2\leq m \leq n_2$, but $l^1_{n_1}$ eventually in $J_{n_2}^1$. It follows that $\Gamma_{n_2}(A)=2$. If $\Gamma_{n_2}(A)=2$, we know that there exists some $l^m_{n_1}$ convergent to $l\geq 1/n_2$ and hence we know $\Xi_{\mathrm{class}}(A)$ is either $1$ or $2$.

Suppose that $\Xi_{\mathrm{class}}(A)=3$, then for fixed $n_2$ and any $1\leq m \leq n_2$, $l^m_{n_1}$ eventually lies in $J_{n_2}^1$ and hence our lowest level of the tower must eventually depend on whether $b_{n_2,n_1}(A)\geq 1/{n_2}$. From Lemma \ref{cs_and_bs}, $b_{n_2}(A)=c_{n_2}(A)=1/n_2$ for large $n_2$. It follows that for large $n_2$, $b_{n_2}(A)\geq 1/n_2$ for all $n_1$ and $\Gamma_{n_2}(A)=3$. If $\Gamma_{n_2}(A)=3$ we know that $c_{n_2}(A)\geq b_{n_2}(A)\geq 1/n_2$, which implies $\Xi_{\mathrm{class}}(A)\neq 4$. Finally, note that if $\Xi_{\mathrm{class}}(A)=4$ but there exists $n_2$ with $\Gamma_{n_2}(A)\neq4$, then the above implies the contradiction $\Xi_{\mathrm{class}}(A)\neq 4$. The partial converses proven above imply $\Gamma_{n_2,n_1}$ realises the $\Pi^A_2$ classification.
\end{proof}

\section{Computational Examples}
\label{num_test}

We now demonstrate that, as well as being optimal from a foundations point of view, the algorithms constructed in this paper can be efficiently implemented for large scale computations. The algorithms have desirable convergence properties, with some converging monotonically or eventually constant as captured by the $\Sigma/\Pi$ classification. They are also completely local and hence \textit{parallelisable}. Pseudocodes are given in Appendix \ref{append_pseudo}.

\begin{remark}
Although we only stated results for the graph case, $l^2(V(\mathcal{G}))$, in \S \ref{gen_unbd}, the ideas used to prove the results show that all the classification results and algorithms in \S \ref{gen_unbd}, \S \ref{spec_gap_se_class} and \S \ref{disc_results} extend to general separable Hilbert spaces $\mathcal{H}$. Once a basis is chosen (so that matrix elements make sense), we can introduce concepts like bounded dispersion etc.
\end{remark}

\subsection{Polynomial coefficients}

We first consider computing spectra and pseudospectra of partial differential operators of the form
\begin{equation}
\label{formal}
T=P(x_1,...,x_d,\partial_1,...,\partial_d),
\end{equation}
for a polynomial $P$. In this case, the algorithms of \S \ref{dfohjg} (which use a Hermite basis in the proof) reduce to computing the spectrum/pseudospectrum of infinite matrices $A$ acting on $l^2(\mathbb{N})$. From the comments in Example \ref{findf3} and recurrence relations for Hermite functions, we can choose a basis such that $f(n)-n\sim Cn^{(d-1)/d}$, where $f$ is the dispersion function in (\ref{bd_disp}) and $C(d)$ is a constant. We also choose $f$ so that it describes the off-diagonal sparsity structure of $A$, so that $A_{n,k}=A_{k,n}=0$ if $k>f(n)$. Hence, this section also showcases the algorithms presented in \S \ref{gen_unbd}, and the two different methods become equivalent.

Our algorithms are built around routines such as \texttt{DistSpec} (Appendix \ref{append_pseudo}) that compute smallest singular values of matrices. For the examples in this section, all error bounds were \textit{verified with interval arithmetic} using the package INTLAB \cite{Ru99a} that runs in MATLAB. The most efficient way to do this is as follows. First, we compute a candidate smallest singular value $\lambda$ and right-singular vector $v$ of the corresponding finite matrix $B$ in standard double precision.\footnote{Sometimes if $B$ contains entries spanning several orders of magnitude, quadruple precision is used if high accuracy is desired. This is not due to any instabilities in our algorithm, but is simply an intrinsic problem of dealing with such matrices.} This computation can be done using a search routine such as \texttt{DistSpec}, or, often more efficiently if the matrix is sparse, using iterative methods. Let $\tilde B$ be the stored approximation of $B$. Once a candidate pair $(\lambda,v)$ has been computed, we compute a bound on the norm of the residual $\|(\tilde B-\lambda)v\|$ using interval arithmetic. This step is typically faster than the computation of $\lambda$ and $v$. Finally, to obtain an upper bound on the smallest singular value of $B$, we add error bounds corresponding to $\|B-\tilde B\|$ (and also the approximate normalisation of the vector $v$).

\subsubsection{Anharmonic oscillators}
First, consider operators of the form
$$
H=-\Delta+V(x)=-\Delta+\sum_{j}^d a_jx_j+\sum_{j,k=1}^d b_{j,k}x_jx_k+\sum_{\alpha\in\mathbb{Z}_{\geq0}^d,\left|\alpha\right|\leq M}c(\alpha)x^\alpha,
$$
where $a_j,b_{j,k},c(\alpha)\in\mathbb{R}$ and the multi-indices $\alpha$ are chosen such that $\sum_{\left|\alpha\right|\leq M}c(\alpha)x^\alpha$ is bounded below. The Faris--Lavine theorem \cite[Theorem X.28]{reed1975methods} shows that $H$ is essentially self-adjoint. Anharmonic oscillators have attracted interest in quantum research for over three decades \cite{bender2013advanced,weniger1996convergent,bender1973anharmonic,fernandez1989tight}. Amongst their uses are approximations of potentials near stationary points. The problem of developing efficient algorithms to compute their spectra has received renewed interest due to advances in asymptotic analysis and symbolic computing algebra \cite{gaudreau2015computing,barakat2005asymptotic,turbiner2010double}. Current methods are rich and diverse but lack uniformity. We show that we can obtain error control for general anharmonic operators in a computationally efficient manner.

We begin with comparisons to some known results in one dimension \cite{chaudhuri1991improved}:
\begin{align*}
V_1(x)&=x^2-4x^4+x^6, \quad &E_0&=-2,\\
V_2(x)&=4x^2-6x^4+x^6, \quad &E_1&=-9,\\
V_3(x)&=(105/64)x^2-(43/8)x^4+x^6-x^8+x^{10}, \quad &E_0&=3/8,\\
V_4(x)&=(169/64)x^2-(59/8)x^4+x^6-x^8+x^{10}, \quad &E_1&=9/8.
\end{align*}
Following the physicists' convention, if the spectrum is discrete and bounded below, we list the energy levels as $E_0\leq E_1\leq...$. The algorithm $\Gamma_n(A)$ for the spectrum is described by the routine \texttt{CompSpecUB}, shown as pseudocode in Appendix \ref{append_pseudo}. This relies on the approximation of $\left\|R(z,A)\right\|^{-1}$ in Theorem \ref{unif_conv_gamma} given by the routine \texttt{DistSpec}. Other methods such as finite section (of the matrices constructed using Hermite functions) will converge in this case since the spectrum is discrete. However, they do not provide the sharp $\Sigma_1$ classification. We found that the grid resolution of the search routine and the search accuracy for the smallest singular values, not the matrix size, were the main deciding factors in the final error bound. Once we know roughly where the eigenvalues are, we can speed up computations using the fact that the algorithm is local. Furthermore, the computational time of the search routine only grows \textit{logarithmically} in its precision. Hence we set the grid spacing and the spacing of the search routine to $(10^5n)^{-1}$.

Table \ref{test_run} shows the results. All values were computed using a local search grid. In this simple example, the output agrees precisely with the eigenvalues since they lie on the search grid. Note that we quickly gain convergence and the error bounds become the precision of the search routine in \texttt{DistSpec} (namely, $(10^5 n)^{-1}$). This is usually a pessimistic estimate of $\sigma_{\mathrm{inf}}(P_{f(n)}(H-zI)P_n)$. For example, using $n=500$ for the first potential $V_1$, the estimate for $\sigma_{\mathrm{inf}}(P_{f(n)}(H-zI)P_n)$ with $z=-2$ obtained by iterative methods and then verified with interval arithmetic is $5.8\times 10^{-12}$. Finally, for potentials with large polynomial order and for large truncation parameter $n$, the truncation of the matrix involves entries of large and small modulus. For example, for $n=1000$, the truncated matrix corresponding to the potential $V_4$ has a maximum entry modulus of approximately $7.85\times 10^{15}$. Hence we found it necessary to use quadruple precision for such values. This is not due to any instabilities in our algorithm, but is simply an intrinsic problem of dealing with the matrices of this example. Even when using quadruple precision, the total computation time (including verification with interval arithmetic) to compute any of the entries in Table \ref{test_run} was less than 20 seconds on a 3.9GHz desktop computer without parallelisation.

\begin{table}
\centering
\begin{tabular}{|c||c|c|c|}
 \hline
Potential & Exact & $n=500$ &$n=1000$\\
\hline
 $V_1$   & $-2$ &$ -2 \pm 2\times10^{-8} $&$ -2 \pm 10^{-8} $\\
 $V_2$   & $-9$ &$ -9 \pm 2\times10^{-8}$&$ -9 \pm 10^{-8}$\\
 $V_3$   & $0.375$&$ 0.375 \pm 1.62\times 10^{-4} $& $0.375 \pm  10^{-7}$\\
 $V_4$   & $1.125$ & $1.125 \pm 6.02\times 10^{-4}$&$1.125 \pm 2.4\times 10^{-7}$\\
  \hline
\end{tabular}
\vspace{4mm}
\caption{Test of our algorithm on some potentials with known eigenvalues. Note that we quickly converge to the eigenvalue with error bounds computed by the algorithm (through \texttt{DistSpec}) and verified using interval arithmetic.}
\label{test_run}
\end{table}

Next, we demonstrate how the algorithm can be used in more than one spatial dimension. We consider
$$
H_1=-\Delta+x_1^2x_2^2,
$$
which is a classic example of a potential that does not blow up at infinity in every direction, yet still induces an operator with compact resolvent \cite{simon1983some}. Figure \ref{2d_harmonic} shows the convergence of the estimate of $\left\|R(z,H_1)\right\|^{-1}$ from above, as well as finite section estimates. As expected from variational methods, the finite section method produces eigenvalues converging to the true eigenvalues from above (there is no essential spectrum and the operator is positive). Furthermore, the areas where $\texttt{DistSpec}$ has converged correspond to areas where finite section has converged. We also show rigorous error bounds computed using \texttt{DistSpec} for different $n$ for the first five eigenvalues. These are computed using an adaptive grid spacing to resolve the local minima of the approximation of $\left\|R(z,H_1)\right\|^{-1}$ using rectangular truncations. For $n=10^4$, it took about a second to locate the candidate eigenvalue and eigenvector pair (the candidate singular value and left-singular vector of the truncation) and on the order of milliseconds to verify with interval arithmetic. Both timings were on a 3.9GHz desktop computer without parallelisation.

\begin{figure}
\centering
\includegraphics[width=0.48\textwidth,trim={30mm 90mm 35mm 90mm},clip]{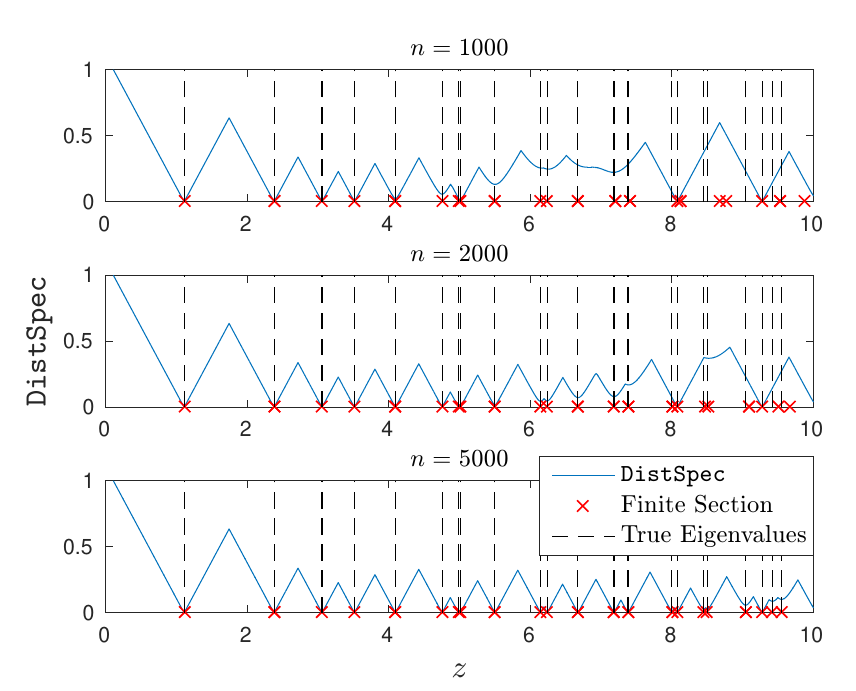}
\includegraphics[width=0.48\textwidth,trim={30mm 90mm 35mm 90mm},clip]{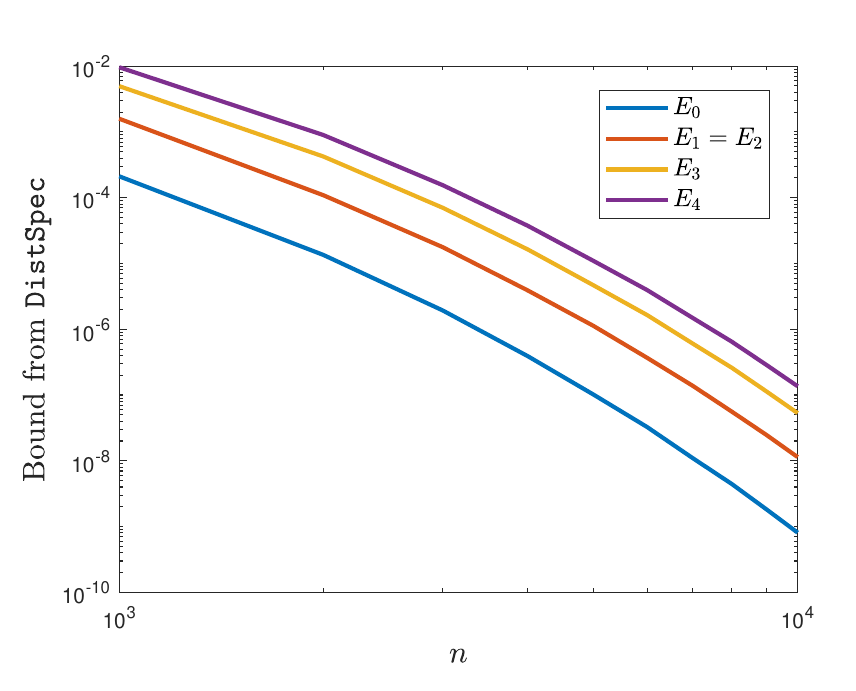}
\caption{Left: The convergence of the algorithm (shown as \texttt{DistSpec}) and finite section to the true eigenvalues on the interval $[0,10]$. Note that points with reliable finite section eigenvalues correspond to points where the estimate of the resolvent norm is well resolved. Right: Error bounds computed using \texttt{DistSpec} (with an adaptive grid spacing) and verified with interval arithmetic.}
\label{2d_harmonic}
\end{figure}

\subsubsection{Pseudospectra and $\mathcal{PT}$ symmetry}
\begin{figure}
\centering
\includegraphics[width=0.48\textwidth,trim={32mm 90mm 35mm 97mm},clip]{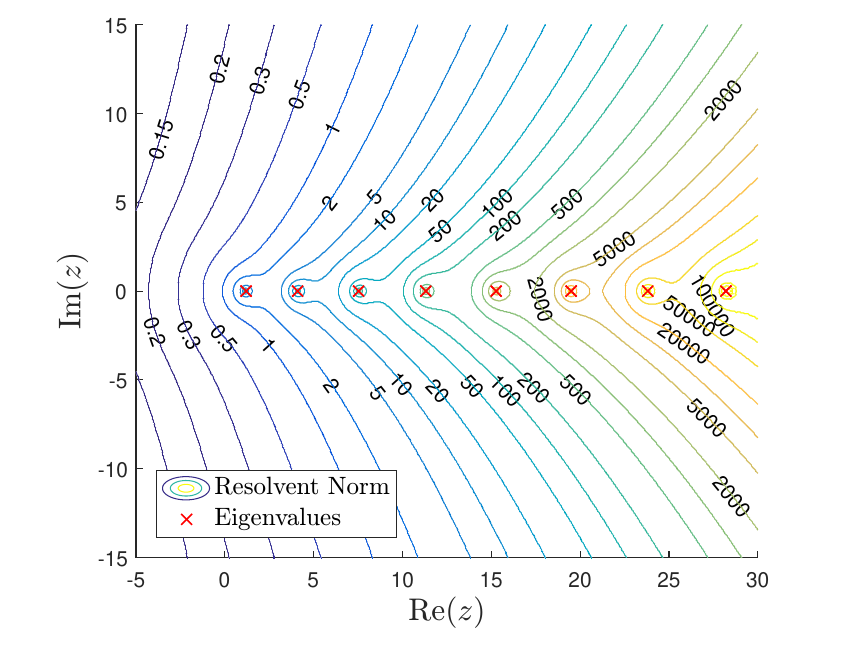}
\includegraphics[width=0.48\textwidth,trim={32mm 90mm 35mm 97mm},clip]{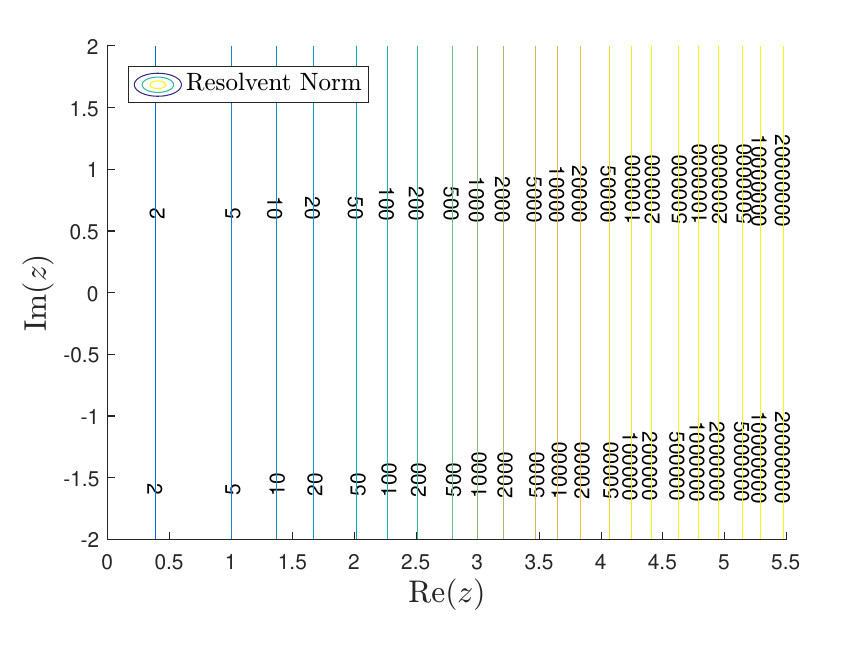}
\caption{Left: Computed pseudospectrum for the imaginary cubic oscillator. Note the clear presence of eigenvalues. Right: Computed pseudospectrum for imaginary Airy operator. Both figures were produced using $n=1000$ and verified with interval arithmetic.}
\label{cubic_osc}
\end{figure}

We now turn to pseudospectra and consider $\mathcal{PT}$-symmetric non-self-adjoint operators $T$ (we consider examples for which compactly supported smooth functions form a core of $T$ and $T^*$ \cite{edmunds1987spectral}). The first example is the imaginary cubic oscillator in one dimension,
$$
H_2=-d^2/{dx^2}+ix^3.
$$
This operator is the most studied example of a $\mathcal{PT}$-symmetric operator\footnote{Meaning $[H_2,\mathcal{PT}]=0$ with $(\mathcal{P}f)(x)=f(-x)$ and $(\mathcal{T}f)(x)=\overline{f(x)}$.}\cite{bender1998real,bender2002complex}, as well as appearing in statistical physics and quantum field theory \cite{fisher1978yang}. The resolvent is compact \cite{caliceti1980perturbation}, and all of the eigenvalues are simple and in $\mathbb{R}_{\geq 0}$ \cite{dorey2001spectral,tai2006simpleness}. The eigenvectors are complete but do not form a Riesz basis \cite{siegl2012metric}. Figure \ref{cubic_osc} (left) shows the computed pseudospectrum using $n=1000$. This demonstrates the instability of the spectrum of the operator. For this example, it took about $0.01$s to estimate the resolvent norm using $n=1000$ and on the order of milliseconds to verify with interval arithmetic, per point, on a 3.9GHz desktop computer without parallelisation.

Next, we consider the imaginary Airy operator
$$
H_3=-d^2/{dx^2}+ix.
$$
This is known to have an empty spectrum \cite{helffer2013spectral} and so demonstrates that the algorithm is effective in this case also. Note that any finite section method will necessarily overestimate the pseudospectrum in regions of the complex plane due to the presence of false eigenvalues. $H_3$ is $\mathcal{PT}$-symmetric and has compact resolvent. The resolvent norm $\left\|R(z,H_3)\right\|$ only depends on the real part of $z$ and blows up exponentially as $\mathrm{Re}(z)\rightarrow +\infty$. We have shown the computed pseudospectrum for $n=1000$ in Figure \ref{cubic_osc} (right). Timings for this example were similar to the imaginary cubic oscillator.

We do not need to discretise anything to apply the above method. Up to numerical errors in the testing of positive definiteness, all computed pseudospectra are guaranteed to be inside the correct pseudospectra. In fact, in our case, we checked results using interval arithmetic and obtained a verified lower bound on the resolvent norm.\footnote{To turn this into a formal proof, one would need a proof that the implementation of interval arithmetic is correct and a proof that our code uses the algorithms correctly. Both of these are beyond the scope of this paper.} This reliability is in contrast to the numerical experiments conducted in, for example \cite{davies1999pseudo}, where the operator is discretised. It is also easy to construct examples where discretisations fail dramatically, either not capturing the whole spectrum or suffering from spectral pollution. Even without spectral pollution, figuring out which parts of computations are trustworthy can be very difficult for finite section and related methods \cite{howmany}. Algorithms such as \texttt{PseudoSpec} are a valuable tool to test the reliability of such outputs.

\subsection{Partial differential operators with general coefficients}

\subsubsection{Perturbed harmonic oscillator}

As a first set of examples, we consider
$$
T=-\Delta +x^2 +V(x),
$$
on $L^2(\mathbb{R})$, where $V$ is a bounded potential. Such operators have discrete spectra. The perturbation $V$ causes the eigenvalues to shift relative to the classical harmonic oscillator, whose spectrum is the set of odd positive integers. Table \ref{shift_run} shows the first five eigenvalues for a range of potentials. Each entry in the table is computed with an error bound at most $10^{-9}$ provided by \texttt{DistSpec}. The truncation size is chosen adaptively to achieve this error, and computational times were on the order of seconds on a 3.9GHz desktop computer without parallelisation.

\begin{table}
\centering
\begin{tabular}{|c||c|c|c|c|c|}
 \hline
Potential $V$ & $E_0$ & $E_1$ &$E_2$&$E_3$&$E_4$\\
\hline
 $\cos(x)$   & $1.7561051579$  & $3.3447026910$ & $5.0606547136$ & $6.8649969390$ & $8.7353069954$ \\
 $\tanh(x)$   &$0.8703478514$ & $2.9666370800$& $4.9825969775$& $6.9898951678$& $8.9931317537$\\
 $\exp(-x^2)$   & $1.6882809272$& $3.3395578680$ & $5.2703748823$&$7.2225903394$&$9.1953373991$\\
 $(1+x^2)^{-1}$   & $1.7468178026$ & $3.4757613534$ &$5.4115076464$& $7.3503220313$& $9.3168983920$\\
  \hline
\end{tabular}
\vspace{4mm}
\caption{First five computed eigenvalues for different potentials. Each eigenvalue $E_k$ is computed with an error bound at most $10^{-9}$ via \texttt{DistSpec} with an adaptive truncation size. The eigenvalue $E_k$ is a perturbation of the harmonic oscillator eigenvalue $2k+1$.}
\label{shift_run}
\end{table}

\subsubsection{Fourth-order operator}

Next, we consider the operator
$$
T_\lambda=\frac{d^4}{dx_1^4}+\left(-i\frac{d}{dx_2}+\frac{x_1}{2}\right)^4+\frac{2\lambda x_2+\lambda^2}{1+x_2^2},
$$ 
on $L^2(\mathbb{R}^2)$, as an example with gaps in the essential spectrum. Figure \ref{beam} shows a portion of the spectrum, as well as the output of finite section, using $100$ basis functions in each spatial dimension. The maximum error bound provided by \texttt{DistSpec} is bounded by $10^{-2}$. For each value of $\lambda$, the computational time was on the order of five minutes for a grid spacing of $0.005$ (approximately $2000$ test points\footnote{Though we did not do so, one can do the following to reduce the number of grid points. Suppose one has access to finite section eigenvalue approximations of a self-adjoint operator. In that case, one can consider grid points close to the computed eigenvalues. This method works because finite sections of a self-adjoint operator will approximate all of the spectrum (though of course, suffer from spectral pollution). Another strategy is to choose the grid adaptively. Once an estimate of the distance $r$ to the spectrum has been computed at a point $z$, we can ignore grid points in a ball of radius $r$ around $z$.}) when executed with parallelisation using 20 CPU cores. Finite section produces heavy spectral pollution in the gaps of the essential spectrum. The spectrum for $\lambda=0$ is shown as red circles, and consists of isolated eigenvalues of infinite multiplicity. As $\lambda$ increases, these fan out to produce the essential spectra shown. 

\begin{figure}
\centering
\begin{overpic}[width=0.49\textwidth]{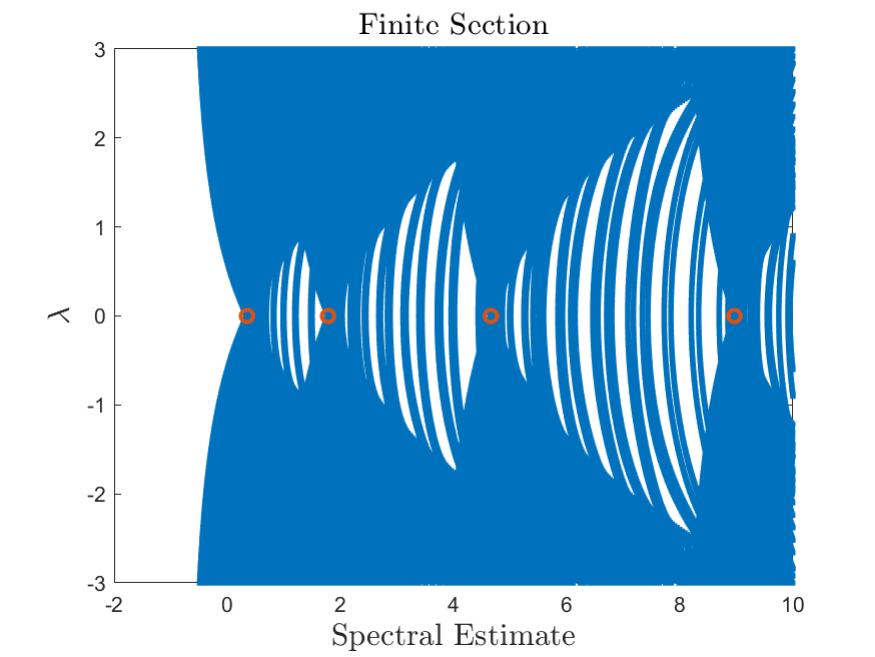}
\footnotesize
		\put (13.5,53) {\tiny{Spectral}}
		\put (13.5,50) {\tiny{Pollution}}
		\normalsize
		\put(22,49){\line(1,-1){10}}
		\put(22,49){\line(2,-1){20}}
		\put(22,49){\line(3,-1){30}}
		\put(22,49){\line(4,-1){40}}
		\put(22,49){\line(5,-1){50}}
\end{overpic}
\includegraphics[width=0.49\textwidth]{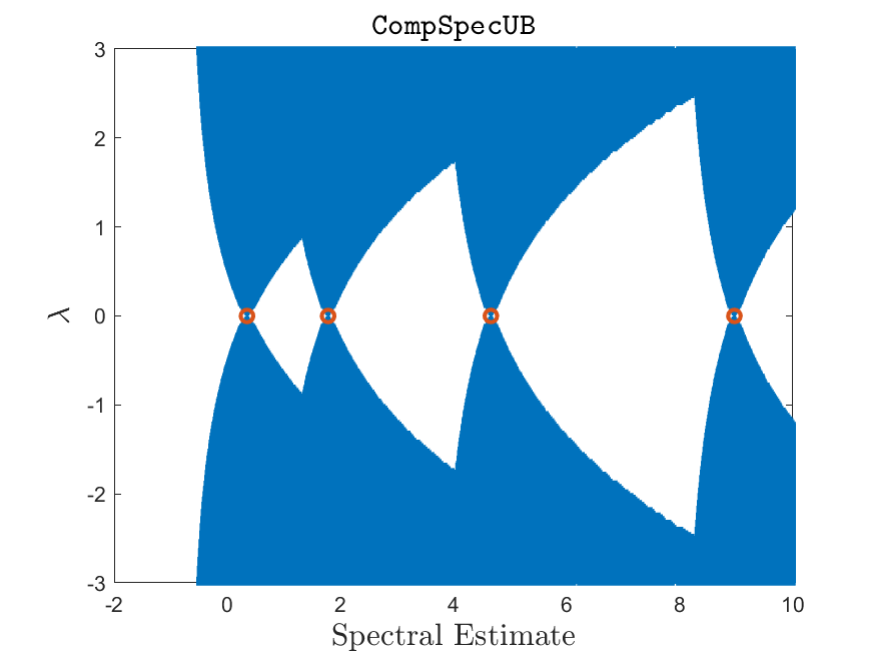}
\caption{Left: Output of finite section showing severe spectral pollution in the gaps of the essential spectrum. Right: Output of \texttt{CompSpecUB}. Both plots use $100$ basis functions in each spatial dimension. The spectrum of $T_0$ is shown as red circles.}
\label{beam}
\end{figure}

\subsection{Example for discrete Spectra}
\label{disc_num}
We now turn to the computation of discrete spectra. Although it is hard to analyse the convergence of a height two tower, we can take advantage of the extra structure in this problem. The routine \texttt{DiscreteSpec} in Appendix \ref{append_pseudo} computes $\Gamma_{n_2,n_1}(A)$ such that $\lim_{n_1\rightarrow\infty}\Gamma_{n_2,n_1}(A)$ is a finite subset of $\mathrm{Sp}_d(A)$. Furthermore, for each $z\in\mathrm{Sp}_d(A)$, there is at most one point in $z_{n_1}\in\Gamma_{n_2,n_1}(A)$ approximating $z$. We can use the routine \texttt{DistSpec} to gain an error bound of $\mathrm{dist}(z_{n_1},\mathrm{Sp}(A))$, which, for large $n_1$, will be equal to $\left|z-z_{n_1}\right|$ since $z$ is isolated. As we increase $n_2$, more and more of the discrete spectrum (in general portions nearer the essential spectrum) are approximated. 

Our example is the Almost Mathieu Operator on $l^2(\mathbb{Z})$ given by
\begin{equation*}
(H_{\alpha}x)_n=x_{n-1}+x_{n+1}+2\lambda\cos(2\pi n\alpha)x_n,
\end{equation*}
where we set $\lambda=1$ (critical coupling). For rational choices of $\alpha$, the operator is periodic and its spectrum is purely absolutely continuous. For irrational $\alpha$ the spectrum is a Cantor set (Ten Martini Problem). To generate a discrete spectrum, we add a perturbation of the potential of the form
\begin{equation}
\label{dsfghklg}
V(n)=V_n/(\left|n\right|+1),
\end{equation}
where $V_n$ are independent and uniformly distributed in $[-2,2]$. The perturbation is compact so preserves the essential spectrum. This type of problem is well studied, for example, in the more general setting of Jacobi operators \cite{teschl2000jacobi,hundertmark2002lieb}.

\begin{figure}
\centering
\includegraphics[width=1\textwidth,trim={10mm 0mm 10mm 0mm},clip]{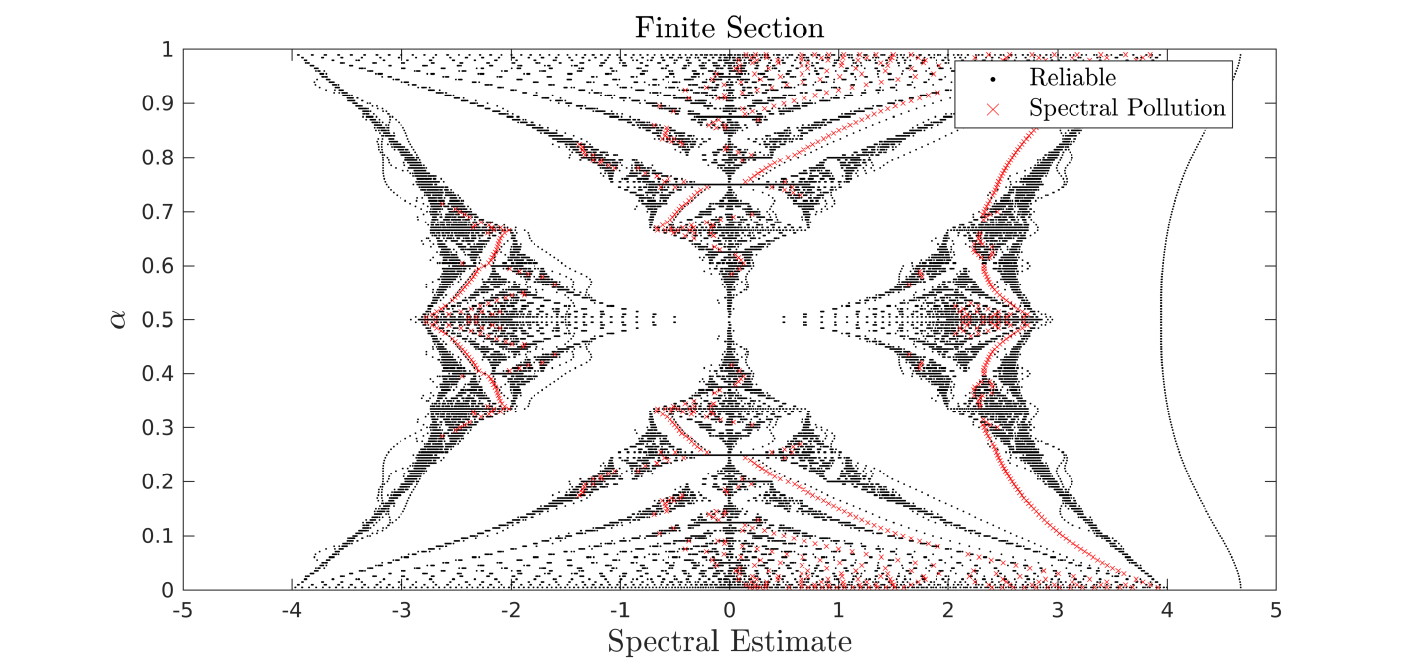}\\
\includegraphics[width=1\textwidth,trim={10mm 0mm 10mm 0mm},clip]{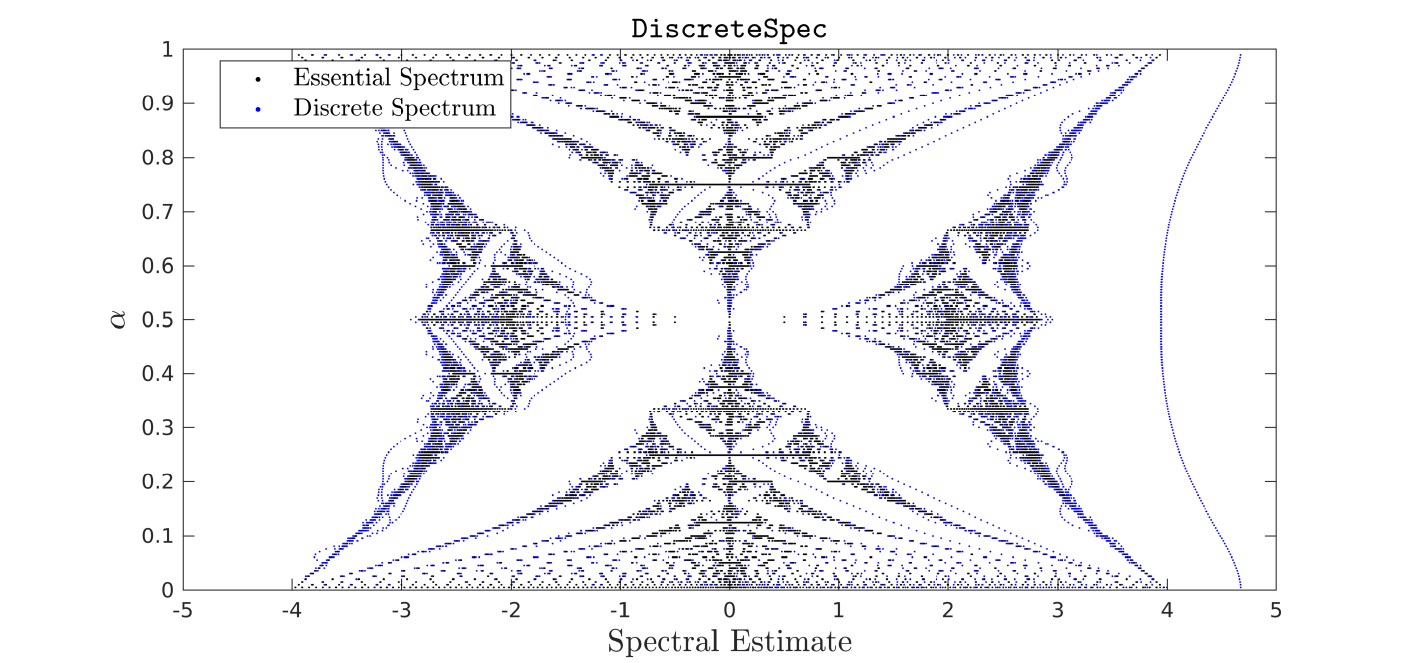}
\caption{Top: Output of finite section. Spectral pollution detected by our algorithm is shown as red crosses. Bottom: Output of \texttt{DiscreteSpec} and the splitting into the essential spectrum and the discrete spectrum. The output captures the discrete spectrum down to a distance $\approx 0.01$ away from the essential spectrum, which can be made smaller for larger $n_2$.}
\label{disc1}
\end{figure}

Figure \ref{disc1} shows a typical result for a realisation of the random potential. The top panel of the figure shows the output of finite section and our algorithm (with a uniform error bound of $10^{-2}$) for computing the total spectrum. The bottom panel of the figure shows the output of \texttt{DiscreteSpec}, which separates the discrete spectrum from the essential spectrum. For each value of $\alpha$, the computational time was on the order of tens of seconds on a 3.9GHz desktop computer without parallelisation. For each $\alpha$, we took $n_2$ large enough for an expected limit inclusion $\mathrm{Sp}_d(A)\subset\Gamma_{n_2}(A)+B_{0.01}(0)$ (obtained by comparing with the output of the height two tower for computing the essential spectrum). Taking $n_2$ larger caused sharper inclusion bounds. Additionally, we confirmed the accuracy of the results using a height one tower to compute the spectrum with and without the random potential. Note that it is difficult to detect spectral pollution when using finite section with the additional perturbation (\ref{dsfghklg}). In contrast, \texttt{DiscreteSpec} computes the discrete spectrum without spectral pollution and allows us to separate the discrete spectrum from the essential spectrum.

The error bounds provided by \texttt{DistSpec} (applied to the output of \texttt{DiscreteSpec}) can also be translated into computing approximates of the eigenvectors of an operator $A$, specifically those corresponding to the discrete spectrum, with an error bound in the following manner. The routine \texttt{ApproxEigenvector} in Appendix \ref{append_pseudo} computes a vector $x_{n_1}$ of norm $\approx1$ such that (in this case taking $\delta\downarrow0$, $c_n=0$)
\[
\left\|(A-z_{n_1}I)x_{n_1}\right\|\leq\texttt{DistSpec}(A,n_1,f(n_1),z_{n_1}).
\]
We write
$
x_{n_1}=x^{d}_{n_1}+y_{n_1},
$
where $x^{d}_{n_1}$ is an eigenvector of $A$ with eigenvalue $z$, and $y_{n_1}$ is perpendicular to the eigenspace associated with $z$ and $z_{n_1}\rightarrow z$. It follows that
$$
\left\|(A-zI)y_{n_1}\right\|\leq \left|z-z_{n_1}\right|+\texttt{DistSpec}(A,n_1,f(n_1),z_{n_1})\leq 2\times\texttt{DistSpec}(A,n_1,f(n_1),z_{n_1}),
$$
for large $n_1$. But $A-zI$ is bounded below on the orthogonal complement of the eigenspace, with lower bound $\mathrm{dist}(z,\mathrm{Sp}(A)\backslash \{z\})$. Hence,
$$
\left\|y_{n_1}\right\|\leq\frac{2\times\texttt{DistSpec}(A,n_1,f(n_1),z_{n_1})}{\mathrm{dist}(z,\mathrm{Sp}(A)\backslash \{z\})}
$$
for large $n_1$. This bound also bounds the $l^2$ distance of $x_{n_1}$ to the eigenspace, and can be estimated by approximating the spectrum of $A$. It is also straightforward to adjust this procedure to eigenvalues of multiplicity greater than $1$ and approximate the whole eigenspace. For the above example, all the eigenvalues were found to have multiplicity $1$, as expected for a random perturbation. Finally, the method of computing eigenvectors and error bounds can also be used for unbounded operators when $z$ lies in the discrete spectrum.

\renewcommand{\baselinestretch}{1.15}

\bibliographystyle{abbrv}      
\bibliography{PhD_bib}

\newpage

\appendix
\newgeometry{left=1in,right=1in,top=1.15in,bottom=0.4in}
\section{Computational Routines}
\label{append_pseudo}

We provide pseudocode for the algorithms of this paper, all of which provide sharp classifications in the SCI hierarchy.

\small

\vspace{1mm}

 \begin{algorithm}[H]
 \SetKwInOut{Input}{Input}\SetKwInOut{Output}{Output}
\SetKwProg{Fn}{Function}{}{end} 
 \SetKwFunction{CompInvg}{CompInvg}
 \SetKwData{Left}{left}\SetKwData{This}{this}\SetKwData{Up}{up}
 \Fn{\CompInvg{$n$, $y$, $g$}}{
  \Input{$n \in \mathbb{N},$ $y \in \mathbb{R}_+$, $g: \mathbb{R}_+ \rightarrow \mathbb{R}_+$}
  \Output{$m \in \mathbb{R}_+$, an approximation of $g^{-1}(y)$}
 \BlankLine
$m = \min\{k/n:k\in\mathbb{N},g(k/n)>y\}$ 
 \BlankLine
   }

  \SetKwInOut{Input}{Input}\SetKwInOut{Output}{Output}
\SetKwProg{Fn}{Function}{}{end} 
 \SetKwFunction{DistSpec}{DistSpec}
  \SetKwFunction{IsPosDef}{IsPosDef}
 \SetKwData{Left}{left}\SetKwData{This}{this}\SetKwData{Up}{up}
\SetKwFunction{Union}{Union}\SetKwFunction{FindCompress}{FindCompress}
 \Fn{\DistSpec{$A$,$n$,$z$,$f(n)$}}{
  \Input{$n \in \mathbb{N},$ $f(n) \in \mathbb{N}$, matrix $A$, $z \in \mathbb{C}$}
  \Output{$y \in \mathbb{R}_+$, an approximation of $\left\|R(z,A)\right\|^{-1}$}
 \BlankLine
 $B = (A-zI)(1:f(n),1:n)$\\
 $C = (A-zI)^*(1:f(n),1:n)$\\
 $S = B^*B$\\
 $T = C^*C$\\
 $\nu =  1$, $l = 0$\\
 \While{$\nu = 1$}{
 $l = l+1$\\
 $p = \IsPosDef(S - \frac{l^2}{n^2})$\\
 $q = \IsPosDef(T - \frac{l^2}{n^2})$\\
 $\nu = \min(p,q)$
     }
 $y = \frac{l}{n}$
   }

 \SetKwInOut{Input}{Input}\SetKwInOut{Output}{Output}
 \BlankLine
\SetKwProg{Fn}{Function}{}{end} 
 \SetKwFunction{CompSpecUB}{CompSpecUB}
  \SetKwFunction{CompInvg}{CompInvg}
\SetKwFunction{Grid}{Grid}
 \SetKwData{Left}{left}\SetKwData{This}{this}\SetKwData{Up}{up}
  \SetKwFunction{DistSpec}{DistSpec}
\SetKwFunction{Union}{Union}\SetKwFunction{FindCompress}{FindCompress}
 \Fn{\CompSpecUB{$A$, $n$, $\{g_m\}$, $f(n)$, $c_n$}}{
  \Input{$n\in\mathbb{N}$, $f(n) \in \mathbb{N}$, $c_n\in\mathbb{R}_+$ (bound on dispersion), $g_m: \mathbb{R}_+ \rightarrow \mathbb{R}_+$, $A \in \Omega_g$}
  \Output{$\Gamma_n(A) \subset \mathbb{C}$, an approximation of $\mathrm{Sp}(A)$, and $E_n(A) \in \mathbb{R}_+$, the error bound}
  \BlankLine
$G =$ \Grid{$n$} (see (\ref{grid_def_un}))\\

\For{$z \in G$}{$F(z) = \DistSpec{$A$, $n$, $z$, $f(n)$}$\\
 \eIf{$F(z) \leq (\left|z\right|^2+1)^{-1}$}{
 
\For{$w_j \in B_{\CompInvg{$n$, $F(z)$, $g_{\left\lceil \left|z\right|\right\rceil}$}}(z)\cap{}G=\{w_1,...,w_k\}$}{ 
 $F_{j} = \DistSpec{$A$, $n$, $w_j$, $f(n)$}$
 }
 $M_z = \{w_{j} : F_{j} = \min_{q}\{F_{q}\}\}$
 }
{
$M_z = \emptyset$
}
 }
 $\Gamma_n(A) = \cup_{z \in G} M_z$\\
 $E_n(A) = \max_{z \in \Gamma_n(A)} \{\CompInvg{$n$, \DistSpec{$A$, $n$, $z$, $f(n)$}\!$+c_n$, $g_{\left\lceil \left|z\right|\right\rceil}$}\}$
 }
 \centering
\caption{The routine \texttt{CompSpecUB} computes spectra of unbounded operators on $l^2(\mathbb{N})$ (or, more generally, graphs) using the subroutines \texttt{CompInvg} and \texttt{DistSpec} described above, and provides $\Sigma_1$ error control. The subroutine \texttt{IsPosDef} checks whether a matrix is positive definite and is a standard routine that can be implemented in a myriad of ways. In practice, the while loop in \texttt{DistSpec} is replaced by a much more efficient interval bisection method. An alternative method for sparse matrices (which, however, does not rigorously guarantee an error bound on the smallest singular values but still gives an upper bound) is to compute the smallest singular values of the rectangular matrices using iterative methods.}\label{alg__1}
 \end{algorithm}
\newpage
 
\begin{algorithm}[H]
 \SetKwInOut{Input}{Input}\SetKwInOut{Output}{Output}
 \BlankLine
\SetKwProg{Fn}{Function}{}{end} 
 \SetKwFunction{PseudoSpecUB}{PseudoSpecUB}
\SetKwFunction{Grid}{Grid}
 \SetKwData{Left}{left}\SetKwData{This}{this}\SetKwData{Up}{up}
  \SetKwFunction{DistSpec}{DistSpec}
\SetKwFunction{Union}{Union}\SetKwFunction{FindCompress}{FindCompress}
 \Fn{\PseudoSpecUB{$A$, $n$, $f(n)$, $c_n$, $\epsilon$}}{
  \Input{$n\in\mathbb{N}$, $f(n) \in \mathbb{N}$, $c_n \in \mathbb{R}_{+}$, $A \in \hat\Omega$, $\epsilon > 0$}
  \Output{$\Gamma_n(A) \subset \mathbb{C}$, an approximation of $\mathrm{Sp}_{\epsilon}(A)$}
  \BlankLine
$G =$ \Grid{$n$}\\
\For{$z \in G$}{$F(z) = \DistSpec{$A$, $n$, $z$, $f(n)$}\!+c_n$
 }
 $\Gamma_n(A) = \bigcup \{z \in G \, \vert F(z)<\epsilon\}$
 } 
 \caption{\texttt{PseudoSpecUB} computes $\Gamma_n(A)\subset\mathrm{Sp}_\epsilon(A)$ with $\lim_{n\rightarrow\infty}\Gamma_n(A)=\mathrm{Sp}_\epsilon(A)$.}
\end{algorithm} 
 
\begin{algorithm}[H]
 \SetKwInOut{Input}{Input}\SetKwInOut{Output}{Output}
\SetKwProg{Fn}{Function}{}{end} 
 \SetKwFunction{TestSpec}{TestSpec}
\Fn{\TestSpec{$n_1$, $n_2$, $K_{n_2}$, $\gamma_{n_1}(z,A)$}}{
  \Input{$n_1,n_2 \in \mathbb{N}$, $K_{n_2}$ an approximation of $K$, access to evaluation of $\gamma_{n_1}(z,A)$.}
  \Output{$\Gamma_{n_2,n_1}(A)$, an approximation of $\Xi_3(A)$.}
	$\Gamma_{n_2,n_1}(A)=\text{``Does there exist some $z\in K_{n_2}$ such that $\gamma_{n_1}(z,A)<1/2^{n_2}$?''}$
 }

 \SetKwInOut{Input}{Input}\SetKwInOut{Output}{Output}
\SetKwProg{Fn}{Function}{}{end} 
 \SetKwFunction{TestPseudoSpec}{TestPseudoSpec}
\Fn{\TestPseudoSpec{$n_1$, $n_2$, $K_{n_2}$, $\gamma_{n_1}(z,A)$, $\epsilon$}}{
  \Input{$n_1,n_2 \in \mathbb{N}$, $K_{n_2}$ an approximation of $K$, access to evaluation of $\gamma_{n_1}(z,A)$, $\epsilon>0$.}
  \Output{$\Gamma_{n_2,n_1}(A)$, an approximation of $\Xi_4(A)$.}
	$\Gamma_{n_2,n_1}(A)=\text{``Does there exist some $z\in K_{n_2}$ such that $\gamma_{n_1}(z,A)<1/2^{n_2}+\epsilon$?''}$
 }

\caption{\texttt{TestSpec} solves $\{\Xi_3,\hat\Omega\times\mathcal{K}(\mathbb{C})\}$ (does $K$, compact, intersect $\mathrm{Sp}(A)$) with input $K_{n_2}$ and access to $\gamma_{n_1}(z,A)$ (e.g., via \texttt{DistSpec}). Similarly, \texttt{TestPseudoSpec} solves $\{\Xi_4,\hat\Omega\times\mathcal{K}(\mathbb{C})\}$ (does $K$, compact, intersect $\mathrm{Sp}_{\epsilon}(A)$) with input $K_{n_2}$, $\epsilon>0$ and access to $\gamma_{n_1}(z,A)$.}   
\end{algorithm}

\begin{algorithm}[H]
 \SetKwInOut{Input}{Input}\SetKwInOut{Output}{Output}
\SetKwProg{Fn}{Function}{}{end} 
 \SetKwFunction{SpecGap}{SpecGap}
 \SetKwData{Left}{left}\SetKwData{This}{this}\SetKwData{Up}{up}
\SetKwFunction{Union}{Union}\SetKwFunction{FindCompress}{FindCompress}
 \Fn{\SpecGap{$n_1$, $n_2$, $P_{n_1}AP_{n_1}$}}{
  \Input{$n_1,n_2\in\mathbb{N}$, $P_{n_1}AP_{n_1}$ the square truncation of the matrix $A$}
  \Output{$\Gamma_{n_2,n_1}(A)$, an approximation of $\Xi_{\mathrm{gap}}(A)$.}
 \BlankLine

\eIf{$n_1=1$}{
 Set $\Gamma_{n_2,n_1}(A)=1$\\
 }
{\For{$k\in\{2,...,n_1\}$}{Compute $l_k=\mu_2^{(k)}-\mu_1^{(k)}+\epsilon_k$, $\left|\epsilon_k\right|\leq 1/k$,\\
using Corollary \ref{eigs_fin_mat} and notation of Lemmas \ref{case1_fs}, and \ref{case2_fs} applied to $P_kAP_k$.\\
}
Set $J_{n_2}^1=[0,1/(2n_2)]$ and $J_{n_2}^2=(1/n_2,\infty)$\\
\eIf{$\{l_k:k\in\{1,...,n_1\}\cap( J_{n_2}^1\cup J_{n_2}^2)\}=\emptyset$}{
 Set $\Gamma_{n_2,n_1}(A)=0$\\
 }
{Let $\tilde{k}\leq n_1$ be maximal with $l_{\tilde{k}}\in J_{n_2}^1\cup J_{n_2}^2$.\\
\eIf{$l_{\tilde{k}}\in J_{n_2}^1$}{
 Set $\Gamma_{n_2,n_1}(A)=0$.
 }
{Set $\Gamma_{n_2,n_1}(A)=1$.
 }
 }
}
}
\caption{\texttt{SpecGap} solves the spectral gap problem in Theorem \ref{class_thdfjlwdjfkl}, and requires an eigenvalue solver to implement Corollary \ref{eigs_fin_mat} to compute all $n$ eigenvalues of $P_nAP_n$ to arbitrary precision.}
\end{algorithm}
\vspace{0.3cm}

\vspace{0.3cm}
\begin{algorithm}[H]
 \SetKwInOut{Input}{Input}\SetKwInOut{Output}{Output}
\SetKwProg{Fn}{Function}{}{end} 
 \SetKwFunction{SpecClass}{SpecClass}
 \SetKwData{Left}{left}\SetKwData{This}{this}\SetKwData{Up}{up}
\SetKwFunction{Union}{Union}\SetKwFunction{FindCompress}{FindCompress}
 \Fn{\SpecClass{$n_1$, $n_2$, $A$, $f$}}{
  \Input{$n_1,n_2\in\mathbb{N}$, $A \in \widehat{\Omega}_{\mathrm{SA}}^f$, $f$ the dispersion bounding function}
	\Output{$\Gamma_{n_2,n_1}(A)$, an approximation of $\Xi_{\mathrm{class}}(A)$.}
\eIf{$n_1\leq n_2$}{
 Set $\Gamma_{n_2,n_1}(A)=1$\\
 }
{\For{$n\in\{1,...,n_1\}$ and $j\in\{1,...,n-1\}$}{Compute $l_{n}^j=\mu_{j+1}^{(n)}-\mu_1^{(n)}+\epsilon_{n}^j$, $\left|\epsilon_{n}^j\right|\leq 1/{n}$,\\
using Corollary \ref{eigs_fin_mat} and notation of Lemmas \ref{case1_fs}, and \ref{case2_fs} applied to $P_{n}AP_{n}$.\\
}
Set $J_{n_2}^1=[0,1/(2n_2)]$ and $J_{n_2}^2=(1/n_2,\infty)$\\
{\For{$j\in\{1,...,n_2\}$}{Let $k^j_{n_2,n_1}$ be maximal with $1\leq k^j_{n_2,n_1}<n_1$ such that $l^j_{k_{n_2,n_1}^j}\in J_{n_2}^1\cup J_{n_2}^2$ if such $k_{n_2,n_1}^j$ exists.\\
}
\eIf{$k_{n_2,n_1}^1$ exists with $l^1_{k_{n_2,n_1}^1}\in J_{n_2}^2$}{
 Set $\Gamma_{n_2,n_1}(A)=1$\\
 }
{\eIf{any $k_{n_2,n_1}^m$ exists with $l^m_{k_{n_2,n_1}^m}\in J_{n_2}^2$ for $2\leq m\leq n_2$}{
 Set $\Gamma_{n_2,n_1}(A)=2$\\
 }
{
\For{$k\in\{1,...,n_1\}$}{Set $a_k(A)=\min_{x\in\hat{\Gamma}_k(A)}\{x+E(k,x)\}$,\\
Set $q_k=E(k,a_k(A)+1/{n_2})+1/k$.\\
}
Set $b_{n_2,n_1}(A)=\min\{q_k:1\leq k \leq n_1\}$.\\
\eIf{$b_{n_2,n_1}(A)\geq 1/{n_2}$}{
 Set $\Gamma_{n_2,n_1}(A)=3$\\
 }
{Set $\Gamma_{n_2,n_1}(A)=4.$
}
}
}
}
}
}
\caption{\texttt{SpecClass} solves spectral classification problem in Theorem \ref{class_thdfjlwdjfkl}. As well as an eigenvalue solver to implement Corollary \ref{eigs_fin_mat}, we need the algorithm \texttt{CompSpecUB} denoted by $\hat{\Gamma}_n$, which computes the spectrum together with an error bound $E(n,\cdot)$ on the output.}
\end{algorithm}

\newpage

\begin{algorithm}[H]
 \SetKwInOut{Input}{Input}\SetKwInOut{Output}{Output}
\SetKwProg{Fn}{Function}{}{end} 
 \SetKwFunction{DiscreteSpec}{DiscreteSpec}
  \SetKwFunction{IsPosDef}{IsPosDef}
 \SetKwData{Left}{left}\SetKwData{This}{this}\SetKwData{Up}{up}
\SetKwFunction{Union}{Union}\SetKwFunction{FindCompress}{FindCompress}
 \Fn{\DiscreteSpec{$n_1$, $n_2$, $\hat{\Gamma}_{n_1}(A)$, $E({n_1},\cdot)$, $\tilde{\Gamma}_{n_2,n_1}(A)$}}{
  \Input{$n_1,n_2\in\mathbb{N}$, $\hat{\Gamma}_{n_1}(A)$ an approximation of $\mathrm{Sp}(A)$, error estimate $E({n_1},\cdot)$ over $\hat{\Gamma}_{n_1}(A)$, $\tilde{\Gamma}_{n_2,n_1}(A)$ an approximation of $\mathrm{Sp}_{\mathrm{ess}}(A)$}
  \Output{$\Gamma_{n_2,n_1}(A)$, an approximation of $\mathrm{cl}(\mathrm{Sp}_d(A))$.}
 \BlankLine
\eIf{$n_2\leq n_1$}{
\For{$n_2\leq k\leq n_1$}{ 
 $\zeta_{k,n_1}(A)=\{z\in\hat{\Gamma}_{n_1}(A):E({n_1},z)<\mathrm{dist}(z,\tilde{\Gamma}_{k,n_1}(A))-1/k\}$\\

\For{$z,w \in \zeta_{k,n_1}(A)$}{$z\sim_{n_1} w$ if and only if $B_{E({n_1},w_j)}(w_j)\cap B_{E({n_1},w_{j+1})}(w_{j+1})\neq\emptyset$ for some $z=w_1,w_2,...,w_n=w\in\zeta_{k,n_1}(A)$\\
}
This gives equivalence classes $[z_1],...,[z_m]$\\
\For{$j\in\{1,...,m\}$}{Choose $z_{k_j}\in[z_j]$ of minimal $E({n_1},\cdot)$\\
}
\eIf{$\cup_{j\in\{1,...,m\}} \{z_{k_j}\}\neq \emptyset$}{
$\Phi_{k,n_1}(A) = \cup_{j\in\{1,...,m\}} \{z_{k_j}\}$\\
}
{$\Phi_{k,n_1}(A)=\emptyset$.\\
}
 }
\eIf{At least one of $\Phi_{k,n_1}(A)\neq\emptyset$}{
$\Gamma_{n_2,n_1}(A)=\Phi_{k,n_1}(A)=\emptyset$ with $k$ minimal such that $\zeta_{k,n_1}(A)\neq\emptyset$\\}
{$\Gamma_{n_2,n_1}(A)=\{0\}$\\}
}
{$\Gamma_{n_2,n_1}(A)=\{0\}$.\\}
}
\SetKwInOut{Input}{Input}\SetKwInOut{Output}{Output}
\SetKwProg{Fn}{Function}{}{end} 
 \SetKwFunction{Multiplicity}{Multiplicity}
  \SetKwFunction{IsPosDef}{IsPosDef}
 \SetKwData{Left}{left}\SetKwData{This}{this}\SetKwData{Up}{up}
\SetKwFunction{Union}{Union}\SetKwFunction{FindCompress}{FindCompress}
 \Fn{\Multiplicity{$A$, $n_1$, $n_2$, $f(n_1)$, $z_{n_1}$, $d_{n_1}$}}{
  \Input{$n_1,n_2 \in \mathbb{N},$ $f(n_1) \in \mathbb{N}$, $A \in \Omega_{\mathrm{N}}^d$, $z_{n_1} \in \mathbb{C}$, $d_{n_1}$}
  \Output{$h_{n_2,n_1}(A,z_{n_1})$, an integer approximation of $h(A,z)$, where $z_{n_1}\rightarrow z$.}
 \BlankLine
 $B=[(A-z_{n_1}I)(1:f(n_1),1:{n_1})]^*[(A-z_{n_1}I)(1:f({n_1}),1:{n_1})]-(1/{n_1}-d_{n_1}) I$\\
$[L,D,P^T]=\texttt{ldl}(B)$ (compute $L,D,P$ such that $PBP^T=LDL^*$)

 \eIf{$D$ is diagonal}{
 
Find $J$ the set of $j$ with $D(j,j)<0$\\
$h_{n_2,n_1}(A,z_{n_1})=\left|J\right|$
 }
{
Find $J_1$ the set of $j$ with size $1$ block $D(j,j)<0$\\
Find $J_2$ the number of negative eigenvalues corresponding to size $2$ blocks by looking at trace and\\determinant\\
$h_{n_2,n_1}(A,z_{n_1})=\left|J_1\right|+\left|J_2\right|$
 }
 }  
\caption{\texttt{DiscreteSpec} computes the closure of the discrete spectrum of $A$. The approximation of the essential spectrum, $\tilde{\Gamma}_{n_2,n_1}(A)$, is described in the proof of convergence and was given in \cite{ben2015can}. Moreover, $\lim_{n_1\rightarrow\infty}\Gamma_{n_2,n_1}(A)\subset \mathrm{Sp}_d(A)$ (see (\ref{dfgviwh})), and converges up to $\mathrm{cl}(\mathrm{Sp}_d(A))$ as $n_2\rightarrow\infty$. Given $z_{n_1}\rightarrow z$, \texttt{Multiplicity} computes the multiplicity, $h(A,z)$, of the eigenvalue $z$.}
\end{algorithm}

\newpage

\begin{algorithm}[H]
 \SetKwInOut{Input}{Input}\SetKwInOut{Output}{Output}
\SetKwProg{Fn}{Function}{}{end} 
 \SetKwFunction{DiscSpecEmpty}{DiscSpecEmpty}
  \SetKwFunction{IsPosDef}{IsPosDef}
 \SetKwData{Left}{left}\SetKwData{This}{this}\SetKwData{Up}{up}
\SetKwFunction{Union}{Union}\SetKwFunction{FindCompress}{FindCompress}
 \Fn{\DiscSpecEmpty{$n_1$, $n_2$, $\hat{\Gamma}_{n_1}(A)$, $E({n_1},\cdot)$, $\tilde{\Gamma}_{n_2,n_1}(A)$}}{
  \Input{$n_1,n_2\in\mathbb{N}$, $\hat{\Gamma}_{n_1}(A)$ an approximation of $\mathrm{Sp}(A)$, error estimate $E({n_1},\cdot)$ over $\hat{\Gamma}_{n_1}(A)$, $\tilde{\Gamma}_{n_2,n_1}(A)$ an approximation of $\mathrm{Sp}_{\mathrm{ess}}(A)$}
  \Output{$\Gamma_{n_2,n_1}(A)$, an approximation of $\Xi_2^d(A)$.}
 \BlankLine
 $\zeta_{n_2,n_1}(A)=\{z\in\hat{\Gamma}_{n_1}(A):E({n_1},z)<\mathrm{dist}(z,\tilde{\Gamma}_{n_2,n_1}(A))-1/n_2\}$\\

\eIf{$\zeta_{n_2,n_1}(A)\neq \emptyset$}{
$\Gamma_{n_2,n_1}(A) = 1$\\
}
{$\Gamma_{n_2,n_1}(A)=0$.
}
 }
\caption{\texttt{DiscSpecEmpty} computes $\Xi_2^d(A)$ (is the discrete spectrum non-empty) in two limits for $A\in\Omega_{\mathrm{N}}^f$, the class of bounded normal operators with known dispersion bounding function. The inputs are the algorithm $\hat{\Gamma}_n$ computing the spectrum (for example, \texttt{CompSpecUB}), the error control $E(n,z)$ (that converges to the true error uniformly on compact subsets of $\mathbb{C}$) and the height two tower $\tilde{\Gamma}_{n_2,n_1}$ presented in \cite{ben2015can} to compute the essential spectrum. }
\end{algorithm}

\begin{algorithm}[H]
 \SetKwInOut{Input}{Input}\SetKwInOut{Output}{Output}
\SetKwProg{Fn}{Function}{}{end} 
 \SetKwFunction{ApproxEigenvector}{ApproxEigenvector}
  \SetKwFunction{IsPosDef}{IsPosDef}
 \SetKwData{Left}{left}\SetKwData{This}{this}\SetKwData{Up}{up}
\SetKwFunction{Union}{Union}\SetKwFunction{FindCompress}{FindCompress}
 \Fn{\ApproxEigenvector{$A$, $n$, $f(n)$, $z_n$, $E(n,z_n)$, $\delta$}}{
  \Input{$n \in \mathbb{N},$ $f(n) \in \mathbb{N}$, $A$, $z_n \in \mathbb{C}$, error bound $E(n,z_n)$ and tolerance $\delta>0$}
  \Output{$x_n \in \mathbb{C}^n$, a vector satisfying $\left\|(A-z_nI)x_n\right\|\leq \left\|x_n\right\|(E(n,z_n)+c_n+\delta)$}
 \BlankLine
 $\epsilon=(E(n,z_n)+\delta)^2$\\
 $B=[(A-z_nI)(1:f(n),1:n)]^*[(A-z_nI)(1:f(n),1:n)]-\epsilon I$\\
 $[L,D,P^T]=\texttt{ldl}(B)$ (compute $L,D,P$ such that $PBP^T=LDL^*$)

 \eIf{$D$ is diagonal}{
 
Find $i$ with $D(i,i)<0$\\
$y=e_i$\\
 }
{
Find $y$ eigenvector of $D$ with eigenvalue $<0$ 
 }
Solve upper triangular system $y=L^*Px_n$ for $Px_n$, apply $P^T$ to obtain $x_n$, and then normalise to precision $\delta$.
 }
\caption{\texttt{ApproxEigenvector} takes as input $A$, $n$, $f(n)$, $z_n$ and the bound $E(n,z_n)$ where $\sigma_{\mathrm{inf}}(P_{f(n)}(A-z_{n}I)|_{P_{n}(l^2(\mathbb{N}))})\leq E(n,z_n).$ Given $\delta>0$, it computes an approximate eigenvector $x_n$ (of finite support) satisfying $\left\|(A-z_{n}I)x_{n}\right\|\leq \left\|x_n\right\|(E(n,z_n)+c_n+\delta)\text{ and }1-\delta<\left\|x_n\right\|<1+\delta.$}
\end{algorithm}

\end{document}